\documentclass[12pt]{amsart}

\usepackage[colorlinks=true, linkcolor=blue, citecolor=blue]{hyperref}

\usepackage{amssymb}
\usepackage{amsmath, graphicx, rotating}
\usepackage{color}
\usepackage{soul}
\usepackage[dvipsnames]{xcolor}

\usepackage[T1]{fontenc}
\usepackage{lmodern}
\usepackage[english]{babel}

\usepackage{ upgreek }
\usepackage{stmaryrd}
\SetSymbolFont{stmry}{bold}{U}{stmry}{m}{n}
\usepackage{amsthm}
\usepackage{float}

\usepackage{ bbm }
\usepackage{ stmaryrd }
\usepackage{ mathrsfs }
\usepackage{ frcursive }
\usepackage{ comment }

\usepackage{pgf, tikz}
\usetikzlibrary{shapes}
\usepackage{varioref}
\usepackage{enumitem}

\usepackage[margin=1in]{geometry}
\newtheorem{theorem}{Theorem}[section]
\newtheorem*{theorem*}{Theorem}
\newtheorem{lemma}[theorem]{Lemma}

\newtheorem{proposition}[theorem]{Proposition}

\labelformat{hypothesis}{\textbf{M\kern-0.1mm#1}}

\newtheorem{condition}{Condition}

\theoremstyle{definition}

\numberwithin{equation}{section}

\newcommand*{\abs}[1]{\left\lvert#1\right\rvert}
\newcommand*{\norm}[1]{\left\lVert#1\right\rVert}

\newcommand*{\pent}[1]{\left\lfloor#1\right\rfloor}
\newcommand*{\sachant}[2]{\left.#1 \,\middle|\,#2\right.}

\def\bb#1{\mathbb{#1}}
\def\bs#1{\boldsymbol{#1}}
\def\bf#1{\mathbf{#1}}
\def\scr#1{\mathscr{#1}}
\def\bbm#1{\mathbbm{#1}}
\def\tt#1{\tilde{#1}}
\def\tbf#1{\tilde{\mathbf{#1}}}
\def\tbs#1{\tilde{\boldsymbol{#1}}}
\def\tbb#1{\tilde{\mathbb{#1}}}
\def\geq{\geqslant}
\def\leq{\leqslant}
\def\phi{\varphi}

\newcommand\ee{\varepsilon}
\renewcommand\ll{\lambda}

\DeclareMathOperator{\LL}{L}
\DeclareMathOperator{\dd}{d\!}
\DeclareMathOperator{\e}{e}

\DeclareMathOperator{\id}{id}
\DeclareMathOperator{\supp}{supp}

\begin{document}
\title[Branching processes in Markovian environment]{The survival probability of critical and subcritical branching processes in finite state space Markovian environment}

\author{Ion Grama}
\curraddr[Grama, I.]{ Universit\'{e} de Bretagne-Sud, LMBA UMR CNRS 6205,
Vannes, France}
\email{ion.grama@univ-ubs.fr}

\author{Ronan Lauvergnat}
\curraddr[Lauvergnat, R.]{Universit\'{e} de Bretagne-Sud, LMBA UMR CNRS 6205,
Vannes, France}
\email{ronan.lauvergnat@univ-ubs.fr}

\author{\'Emile Le Page}
\curraddr[Le Page, \'E.]{Universit\'{e} de Bretagne-Sud, LMBA UMR CNRS 6205,
Vannes, France}
\email{emile.le-page@univ-ubs.fr}

\date{\today}

\begin{abstract}
Let $(Z_n)_{n\geq 0}$ be a branching process in a random environment 
defined by a Markov chain $(X_n)_{n\geq 0}$ 
with values in a finite state space $\bb X$
starting at $X_0=i \in\mathbb X.$
We extend from the i.i.d.\ environment to the Markovian one the classical classification of the branching processes into critical and strongly, intermediate and weakly subcritical states.
In all these cases, we study the asymptotic behaviour of the probability that $Z_n>0$ as $n\to+\infty$.
\end{abstract}

\maketitle

\section{Introduction}
Galton-Watson branching process is one of the most used models in the dynamic of populations. 
It has numerous applications in different areas such as biology, medicine, physics, economics etc; 
for an introduction we refer to 
Harris \cite{harris2002theory} or
Athreya and Ney \cite{athreya_branching_1972} and to the references therein.
A significant advancement in the theory and practice was made with the introduction of the branching process in which the offspring distributions vary according to a random environment,
see Smith and Wilkinson \cite{smith_branching_1969} and
Athreya and Karlin \cite{athreya1971branching1, athreya1971branching2}.
This allowed a more adequate modeling and turned out to be very fruitful from the practical as well as from the mathematical 
points of view. 
The recent advances in the study of conditioned limit theorems for sums of functions defined on Markov chains in \cite{grama_conditioned_2016}, \cite{GLLP_affine_2016}, \cite{grama_limit_2016-1} and \cite{GLLP_CLLT_2017}
open the way  to treat some unsolved questions in the case of Markovian environments.  
The problem we are interested here is to study the asymptotic behaviour of the survival probability.

Assume first that on the probability space  $\left( \Omega, \scr F, \bb P \right)$
we are given a branching process $\left( Z_n \right)_{n\geq 0}$ in a random environment 
represented by the i.i.d.\ sequence $\left( X_n \right)_{n\geq 0}$ with values in the space $\mathbb X.$ 
Let $f_i(\cdot)$ be the probability generating function of the offspring distributions of $\left( Z_n \right)_{n\geq 0}$, 
provided the value of the environment is $i\in \mathbb X.$
In a remarkable series of papers
Afanasyev \cite{afanasyev_limit_2009}, Dekking \cite{dekking_survival_1987}, Kozlov \cite{kozlov_asymptotic_1977},  
Liu \cite{liu1996survival},
D'Souza and Hambly \cite{dsouza_survival_1997}, 
Geiger and Kersting \cite{geiger_survival_2001}, 
Guivarc'h and Liu \cite{guivarch_proprietes_2001} 
and Geiger, Kersting and Vatutin \cite{geiger_limit_2003}
under various assumptions 
have determined the asymptotic behaviour as $n\to+\infty$ of the survival probability $\mathbb P (Z_n>0)$.
Let $\phi(\lambda)$ be the Laplace transform of the random variable $\ln f'_{X_1}(1)$:
$\phi(\ll)=\bb E \left(e^{\ll \ln f'_{X_1}(1)} \right)$, $\ll \in \bb R,$
where $\mathbb E$ is the expectation pertaining to $\mathbb P$. 
In function of the values of the derivatives 
$\phi'(0)=\mathbb E(\ln f'_{X_1}(1))$ and $\phi'(1)=\mathbb E(f'_{X_1}(1)\ln f'_{X_1}(1))$ 
and under some additional moment assumptions on the variables $\ln f'_{X_1}(1)$ and $Z_1,$
the following asymptotic results have been found. 
In the critical case, $\phi'(0)=0$, it was shown in \cite{kozlov_asymptotic_1977} and \cite{geiger_survival_2001} that  
$\mathbb P (Z_n>0)\sim \frac{c}{\sqrt{n}}$; 
hereafter $c$ stands for a constant and $\sim$ means equivalence of sequences as $n\to +\infty.$ 
The behaviour in the subcritical case, $\phi'(0)<0$, turns out to depend on the value $\phi'(1)$. 
The strongly subcritical case, $\phi'(0)<0$ \& $\phi'(1)<0$, has been studied in \cite{dsouza_survival_1997} and \cite{guivarch_proprietes_2001}
where it was shown that 
$\mathbb P (Z_n>0) \sim c \phi(1)^n$,
with $0< \phi(1)=\mathbb E f'_{X_1}(1)<1$. 
In the intermediate and weakly subcritical cases, $\phi'(0)<0$ \& $\phi'(1)=0$ and $\phi'(0)<0$ \& $\phi'(1)>0$, respectively, it was shown in   
\cite{geiger_limit_2003} that 
$\mathbb P (Z_n>0)\sim c n^{-1/2} \phi(1)^n$
and
$\mathbb P (Z_n>0)\sim c n^{-3/2}\phi(\ll)^n$, 
where $\ll$ is the unique critical point of $\phi$: $\phi'(\ll)=0.$

The goal of the present paper is to determine the asymptotic behaviour as 
$n\to +\infty$ of the survival probability $\mathbb P_i (Z_n>0)$ 
when the environment
$\left( X_n \right)_{n\geq 0}$ is a Markov chain with values in a finite state space $\mathbb X.$
Hereafter $\mathbb P_i$ and $\mathbb E_i$ are the probability and expectation generated by the trajectories of $\left( X_n \right)_{n\geq 0}$
starting at $X_0=i \in \bb X.$
Set $\rho(i) = \ln  f_i'(1)$, $i \in \bb X$.
Consider the associated Markov walk  
$S_n =  \sum_{k=1}^n \rho\left( X_1 \right)$, $n \geq 0$.
In the case of a Markovian environment the behaviour of the survival probability $\mathbb P_i (Z_n>0)$ depends on the function  
\[
k(\ll) := \lim_{n\to +\infty} \bb E_i^{1/n} \left( \e^{\ll S_n} \right),
\]
which is well defined, analytic in $\ll \in \bb R$ and does not depend on $i \in \bb X$ (see Section \ref{nenuphar}).
In some sense the function $k$ plays the same role that the function $\phi$ in the case of i.i.d.\ environment.

Let us present briefly the main results of the paper. 
Under appropriate conditions, we show the asymptotic behaviour of the survival probability 
$\mathbb P_i (Z_n>0)$ in function of the following classification:
\begin{itemize}
\item Critical case:
if $k'(0)=0$, then, 
for any $i,j \in \bb X$,
\[
\bb P_i \left( Z_n > 0 \,,\, X_n = j \right) \underset{n \to +\infty}{\sim} \frac{\bs \nu (j) u(i)}{\sqrt{n}},
\]
where $u(i)$ is a constant depending on $i$ and $\bs \nu$ is the stationary probability measure of the Markov chain 
$\left( X_n \right)_{n\geq 0}$.
\item Strongly subcritical case:  if $k'(0)<0$ and $k'(1)<0$, then, for any $i,j \in \bb X$,
\[
\bb P_i \left( Z_n > 0 \,,\, X_n = j \right) \underset{n \to +\infty}{\sim}  v_1(i)u(j) k(1)^n .
\]
where $u(j)$ and $v_1(i)$ are depending only on $j$ and $i$ respectively.
\item Intermediate subcritical case: if $k'(0)<0$ and $k'(1)=0$, then, for any $i,j \in \bb X$,
\[
\bb P_i \left(  Z_n > 0 \,,\, X_n = j \right) \underset{n \to +\infty}{\sim}  v_1(i) u(j)  \frac{k(1)^n}{\sqrt{n}}.
\]
where $u(i)$ depends only on $i$. 
\item Weakly subcritical case: if $k'(0)<0$ and $k'(1)>0$, then, for any $i,j \in \bb X$,
\[
\bb P_i \left( Z_n > 0 \,,\, X_n = j \right) \underset{n \to +\infty}{\sim} k(\ll)^n \frac{u(i,j)}{n^{3/2}},
\]
where $u(i,j)$ depends only on $i$ and $j$ and $\ll$ is the critical point of $k$: $k'(\ll)=0.$
\end{itemize}
The critical case has been considered in Le Page and Ye \cite{le_page_survival_2010} in a more general setting. However, the conditions in their paper do not cover the present situation and the employed method is different from ours.

From the results of  Section \ref{nenuphar} it follows that the classification stated above coincides with the usual classification for branching processes when the environment is i.i.d. 
Indeed, Lemma \ref{mulet} implies that
$k'(0) = \mathbb E_{\bs \nu} \left(  \ln f_{X_1}'(1) \right)$,
where $\mathbb E_{\bs \nu} $ is the expectation generated by the finite dimensional distributions 
of the Markov chain $\left( X_n \right)_{n\geq 0}$ in the stationary regime.
For an i.i.d.\ environment this is exactly $\mathbb E(\ln f'_{X_1}(1))=\phi'(0).$
The value $k'(1)$ can also be related to  the first moment of the random variable $\ln f'_{X_1}(1)$. 
For this we need the  transfer operator $P_{\ll}$ related to the Markov chain $\left( X_n \right)_{n\geq 0}$,
see Section \ref{nenuphar} for details.
The normalized transfer operator $\tbf P_{\ll}$ generates a Markov chain whose invariant probability is denoted by 
$\tbs \nu_{\ll}.$ 
Again by Lemma \ref{mulet}, it holds
$\frac{k'(1)}{k(1)}  = \tbb E_{\tbs \nu_{\ll}} \left( \ln  f_{X_1}'(1) \right)$, where $\tbb E_{\tbs \nu_{\ll}} $ is the expectation generated by the finite dimensional distributions of the Markov chain $( X_n )_{n\geq 0}$ with transition probabilities $\tbf P_{\ll}$
in the stationary regime. For an i.i.d.\ environment, we have $\frac{k'(1)}{k(1)} =\mathbb E \left(  f_{X_1}'(1) \ln  f_{X_1}'(1)  \right)=\phi'(1),$ which shows that both classifications are equivalent.

Now we shall shortly explain the approach of the paper. 
We start with a well known relation between the survival probability $\mathbb P_i(Z_n>0)$ 
and the associated random walk 
$\left( S_n \right)_{n\geq 0}$
which goes back to Agresti \cite{agresti_bounds_1974} and which is adapted it to the Markov environment as follows: for any initial state $X_0=i,$ 
\begin{equation}
\label{Agre001}
	\mathbb P_i(Z_n>0) = \mathbb E_i (q_n), \quad \mbox{where} \quad q_n^{-1}= \e^{-S_n} + \sum_{k=0}^{n-1} \e^{-S_k} \eta_{k+1,n}
\end{equation}
and under the assumptions of the paper the random variables $\eta_{k+1,n}$ are bounded.
Our proof is essentially based on three tools: 
conditioned limit theorems for Markov chains which have been obtained recently
in \cite{grama_limit_2016-1} and  \cite{GLLP_CLLT_2017},  
the exponential change of measure which is defined with the help of the transfer operator,
see Guivarc'h and Hardy \cite{guivarch_theoremes_1988}, 
and the duality for Markov chains which we develop in Section \ref{batailleBP}.

Let us first consider the critical case. Let $\tau_y$ be the first moment when the random walk $\left(y+ S_n \right)_{n\geq 0}$ becomes negative.
In the critical case, one can show that only the trajectories that stay positive (i.e.\ when $\tau_y>n$) 
have impact on the survival probability, so that the probability $\sqrt{n}\bb P\left( Z_n>0, \tau_y\leq n \right)$ is negligible as  $n\to+\infty$ and $y\to+\infty$.    
This permits to replace the expectation $\sqrt{n}\mathbb E_i (q_n)$ 
by $\sqrt{n}\bb E_i\left( q_n \,;\, \tau_y > n \right)=\sqrt{n}\bb E_i\left( \sachant{q_n}{\tau_y > n} \right) \bb P_i\left(\tau_y > n \right)$.
The asymptotic of $\sqrt{n}\mathbb P_i\left(\tau_y > n \right)$ is given in \cite{grama_limit_2016-1}
and using the local limit theorem from \cite{GLLP_CLLT_2017}
we show that the expectation $\bb E_i\left( \sachant{q_n}{\tau_y > n} \right)$ converges to a positive constant.

The subcritical case is much more delicate. Using the normalized transfer operator $\tbf P_{\ll}$ we apply a change of the probability measure, say $\tbb P_i$, under which \eqref{Agre001} reduces to the study of the expectation 
\[
k(\ll)^n  \tbb E_i \left( e^{-\ll S_n} q_n \right).
\]
Choosing $\ll=1,$ we have $\tbb E_i \left( e^{- S_n} q_n \right) = \tbb E^*_i \left( q^*_n \right)$,
where 
$\tbb E^*_i$ is the expectation generated by the dual Markov walk  $\left( S^*_n \right)_{n\geq 0}$,
\begin{equation}
\label{Agredual001}
 (q^*_n)^{-1}= 1 + \sum_{k=1}^{n} \e^{-S^*_k} \eta^*_k
\end{equation}
and the random variables $\eta^*_k$ are bounded.
In the strongly subcritical case the series in
\eqref{Agredual001} converges by the law of large numbers for $\left( S^*_n \right)_{n\geq 0}$,
so the resulting rate of convergence is determined only by $k(1)^{n}.$
To find the asymptotic behaviour of the expectation
$\tbb E^*_i \left( q^*_n \right)$
in the intermediate subcritical case we proceed basically in the same way as in the critical case
which explains the apparition of the factor $n^{-1/2}$.
In the weakly subcritical case we choose $\ll$ to be the critical point of $k$: $k'(\ll)=0$. 
We make use of the conditioned local limit theorem which, in addition to $k(\ll)^{n}$, contributes with the factor $n^{-3/2}$.   
 
The outline of the paper is as follows: 
\begin{itemize}
\item Section \ref{sec not res}: We give the necessary notations and formulate the main results.
\item Section \ref{prliminrez}: 
Introduce the associated Markov chain and relate it to the survival probability.
Introduce the dual Markov chain. 
State some useful assertions for walks on Markov chains conditioned to stay positive and on the transfer operator.
\item Sections \ref{critcase}, \ref{lagon}, \ref{intermedcrit} and \ref{weaklysubcrit}: Proofs in the critical, strongly subcritical, intermediate subcritical and weakly subcritical cases, respectively.
\end{itemize}

Let us end this section by fixing some notations. The symbol $c$ will denote a positive constant depending on the all previously introduced constants. Sometimes, to stress the dependence of the constants on some parameters 
 $\alpha,\beta,\dots$ we shall use the notations $ c_{\alpha}, c_{\alpha,\beta},\dots$. All these constants are likely to change their values every occurrence. 
The indicator of an event $A$ is denoted by $\mathbbm 1_A$. For any bounded measurable function $f$ on $\bb X$, random variable $X$ in some measurable space $\bb X$ and event $A$, the integral $\int_{\bb X}  f(x) \bb P (X \in \dd x, A)$ means the expectation $\bb E\left( f(X); A\right)=\bb E \left(f(X) \mathbbm 1_A\right)$.

\section{Notations and main results} \label{sec not res}

Assume that $\left( X_n \right)_{n\geq 0}$  is a homogeneous Markov chain defined
on the probability space $\left( \Omega, \scr F, \bb P \right)$ 
with values in the finite state space $\bb X$. Let $\scr C$ be the set of functions from $\bb X$ to $\bb C$. 
Denote by $\bf P$ the transition operator of the chain $(X_n)_{n\geq 0}$: 
$
\bf P g(i) = \bb E_i \left( g(X_1) \right),
$
for any $g \in \scr C$ and $i \in \bb X$. 
Set $\bf P(i,j) = \bf P(\delta_j)(i)$, where $\delta_j(i) = 1$ if $i = j$ and $\delta_j(i) = 0$ else. 
Note that the iterated operator $\bf P^n$, $n \geq 0$ is given by
$
\bf P^ng(i) = \bb E_i \left( g(X_n) \right).
$
Let $\bb P_i$ be the probability on $\left( \Omega, \scr F \right)$  generated by the finite dimensional distributions 
of the Markov chain $\left( X_n \right)_{n\geq 0}$ starting at $X_0 = i$. 
Denote by $\bb E$ and $\bb E_i$ the corresponding expectation associated to $\bb P$ and $\bb P_i.$

We assume in the sequel that $\left( X_n \right)_{n\geq 0}$ is irreducible and aperiodic. This
is known to be equivalent to the following condition:
\begin{condition}
\label{primitif}
The matrix $\bf P$ is primitive, which means that there exists $k_0 \geq 1$ such that, for any non-negative and non-identically zero function $g\in \scr C$ and $i \in \bb X$,
\[
\bf P^{k_0} g(i) > 0.
\]
\end{condition}

By the Perron-Frobenius theorem, under Condition \ref{primitif}, there exist
positive constants $c_1$ and $c_2$, a unique positive $\bf P$-invariant probability 
$\bs \nu$ on $\bb X$ and an operator $Q$ on $\scr C$ such that for any $g \in \scr C$ and $n \geq 1$,
\[
\bf Pg(i) = \bs \nu(g) + Q(g)(i) \qquad \text{and} \qquad \norm{Q^n(g)}_{\infty} \leq c_1\e^{-c_2n} \norm{g}_{\infty},
\]
where 
$\bs \nu(g) := \sum_{i \in \bb X} g(i) \bs \nu(i)$,
$Q \left(1 \right) = \bs \nu \left(Q(g) \right) = 0$
and 
$\norm{g}_{\infty}= \max_{i \in \bb X} \abs{g(i)}$.
In particular, for any $(i,j) \in \bb X^2$, we have
\begin{equation}
	\label{soeur}
	\abs{\bf P^n(i,j) - \bs \nu(j)} \leq c_1\e^{-c_2 n}.
\end{equation}
 
The branching process in the Markov environment $\left( X_n \right)_{n\geq 0}$
is defined with the help of a collection of generating functions
\begin{equation}
	\label{jazz}
	f_i(s) := \bb E \left( s^{\xi_i} \right), \quad \forall i \in \bb X, \; s \in [0,1],
\end{equation}
where the random variable $\xi_i$ takes its values in $\bb N$ and means the total offspring of one individual when the environment is $i\in \bb X.$ For any $i \in \bb X$, let $( \xi_i^{n,j} )_{j,n \geq 1}$ be independent and identically distributed random variables with the same generating function $f_i$ living on the same probability space $\left( \Omega, \scr F, \bb P \right)$.
We assume that the sequence $( \xi_i^{n,j} )_{j,n \geq 1}$ is independent 
of the Markov chain $\left( X_n \right)_{n\geq 0}.$ 

Assume that the offspring distribution satisfies the following moment constraints.

\begin{condition}
\label{eglise}
For any $i \in \bb X$, the random variable $\xi_i$ is non-identically zero and has a finite variance:
\[
0 < \bb E \left( \xi_i \right) \qquad \text{and} \qquad \bb E ( \xi_i^2 ) < +\infty,  \qquad  \forall i \in \bb X.
\]
\end{condition}
Note that, under Condition \ref{eglise} we have,
\[
\forall i \in \bb X, \qquad 0< \bb E \left( \xi_i \right) = f_i'(1) < +\infty.
\]
and
\[
\forall i \in \bb X, \qquad f_i''(1) =\bb E ( \xi_i^2 )-\bb E \left( \xi_i \right)  < +\infty.
\]

Define the branching process $\left( Z_n \right)_{n\geq 0}$ iteratively: 
for each time $n=1,2,\dots$, given the environment $X_n = i$, the total offspring of each individual $j\in \{1, \dots Z_{n-1} \}$ is given by the random variable $\xi_{i}^{n,j},$ so that the total population is
\begin{equation}
\label{roseau}
Z_0 = 1 \qquad \text{and} \qquad Z_n = \sum_{j=1}^{Z_{n-1}} \xi_{X_n}^{n,j}, \qquad \forall n \geq 1.
\end{equation}

We shall consider branching processes $\left( Z_n \right)_{n\geq 0}$  in one of the following  two regimes:
critical or subcritical (see below for the precise definition). 
In both cases the probability that the population survives until the $n$-th generation tends to zero,  
$\bb P \left( Z_n > 0 \right) \to 0$ as $n \to +\infty$, 
see Smith and Wilkinson \cite{smith_branching_1970}. As noted in the introduction, when the environment is i.i.d., the question of determining the speed of this convergence was answered in \cite{geiger_survival_2001}, \cite{guivarch_proprietes_2001} and \cite{geiger_limit_2003}. The key point in establishing their results is a close relation between the branching process and the associated random walk.
Let us introduce the associated Markov walk corresponding to our setting. 
Define the real function $\rho$ on $\bb X$ by
\begin{equation}
	\label{pipeau}
	\rho(i) = \ln  f_i'(1) , \qquad  \forall i \in \bb X. 
\end{equation}
The associated Markov walk  $\left( S_n \right)_{n\geq 0}$ is defined as follows:
\begin{equation}
\label{petale}
S_0 := 0 \qquad \text{and} \qquad 
S_n :=  \ln \left( f_{X_1}'(1) \cdots f_{X_n}'(1) \right)
= \sum_{k=1}^n \rho\left( X_k \right),
 \quad \forall n \geq 1.
\end{equation}

In order to state the precise results we need one more condition, namely that
the Markov walk $(S_n)_{n\geq 0}$ is non-lattice:

\begin{condition}
\label{cathedrale}
For any $(\theta,a) \in \bb R^2$, there exist $x_0, \dots, x_n$ in $\bb X$ such that
\[
\bf P(x_0,x_1) \cdots \bf P(x_{n-1},x_n) \bf P(x_n,x_0) > 0
\]
and
\[
\rho(x_0) + \cdots + \rho(x_n) - (n+1)\theta \notin a\bb Z.
\]
\end{condition}

The following function plays an important role in determining the 
asymptotic behaviour of the branching processes when the environment is Markovian.  
It will be shown in Section \ref{nenuphar} that under Conditions \ref{primitif} and \ref{cathedrale}, for any $\ll \in \bb R$ and any $i \in \bb X$, the following limit exists and does not depend on the initial state of the Markov chain $X_0=i$:
\[
k(\ll) := \lim_{n\to +\infty} \bb E_i^{1/n} \left( \e^{\ll S_n} \right).
\]
Le us recall some facts on the function $k$ which will be discussed in details in Section \ref{nenuphar} 
and which are used here for the formulation of the main results. 
The function $k$ is closely related to the so-called transfer operator $\bf P_{\ll}$ which is defined for any $\ll \in \bb R$ on $\scr C$ by the relation
\begin{equation}
	\label{transfoper}
	\bf P_{\ll}g(i) := \bf P\left( \e^{\ll \rho} g \right)(i) = \bb E_i \left( \e^{\ll S_1} g(X_1) \right), 
	\quad \mbox{for}\quad g \in \scr C, i \in \bb X.
\end{equation}
In particular, $k(\ll)$ is an eigenvalue 
of the operator $\mathbf P_{\ll}$ corresponding to an eigenvector $v_{\ll}$ 
and is equal to its spectral radius.   
Moreover, the function $k(\ll)$ is analytic on $\mathbb R,$ see Lemma \ref{mulet}. Note also that the transfer operator $\mathbf P_{\ll}$ is not Markov, but it can be easily normalized so that 
the operator $\tbf P_{\ll}g = \frac{\bf P_{\ll}(gv_{\ll})}{k(\ll)v_{\ll}}$ is Markovian. We shall denote by $\tbs \nu_{\ll}$ its 
unique invariant probability measure. 

The branching process in Markovian environment is said to be \textit{subcritical} if $k'(0)<0$, \textit{critical} if $k'(0)=0$ and \textit{supercritical} if $k'(0)>0$.
This definition at first glance may appear different from what is expected 
 in the case of branching processes with i.i.d.\ environment.
With a closer look, however, the relation to the usual i.i.d.\ classification
becomes clear from the following identity, which is established in Lemma \ref{mulet}: 
\begin{equation}
\label{classifiid}
k'(0) = \bs \nu(\rho) = \bb E_{\bs \nu} \left( \rho(X_1) \right) = \bb E_{\bs \nu} \left(  \ln f_{X_1}'(1) \right),
\end{equation}
where $\mathbb E_{\bs \nu} $ is the expectation generated by the finite dimensional distributions 
of the Markov chain $\left( X_n \right)_{n\geq 0}$ in the stationary regime, 
i.e. when the starting point $X_0$ is a random variable  
distributed according to the $\bf P$-invariant measure $\bs \nu.$
In particular, when the environment $\left( X_n \right)_{n\geq 0}$ 
is just an i.i.d.\ sequence of random variables with common law $\bs \nu$, it follows from \eqref{classifiid} that the two classifications coincide.
 
We proceed to formulate our main result in the critical case.

\begin{theorem}[Critical case]
\label{prince}
Assume Conditions \ref{primitif}-\ref{cathedrale} and
\[
k'(0)  = 0.
\]
Then, there exists a positive function $u$ on $\bb X$ such that for any $(i,j) \in \bb X^2$,
\[
\bb P_i \left( Z_n > 0 \,,\, X_n = j \right) \underset{n \to +\infty}{\sim} \frac{\bs \nu (j) u(i)}{\sqrt{n}}.
\]
\end{theorem}

The asymptotic for the probability that   $ Z_n > 0 $ in the case of i.i.d.\ environment has been established earlier by Geiger and Kersting \cite{geiger_survival_2001}
under some moment assumptions on the random variable
$\rho(X_1)=\ln \left( f_{X_1}'(1) \right)$, 
which are weaker that our assumption on finiteness of the state space $\mathbb X.$
Since we deal with dependent environment, Theorem \ref{prince} is not covered by the results in \cite{geiger_survival_2001}.

Now we consider the subcritical case.
The classification of the asymptotic behaviours of the survival time of a branching process 
$\left( Z_n \right)_{n\geq 0}$ in the subcritical case $k'(0)<0$ 
is made in function of the values of $k'(1).$
We say that the branching process in Markovian environment is \textit{strongly subcritical} if $k'(0)<0, k'(1)<0$, \textit{intermediately subcritical} if $k'(0)<0, k'(1)=0$ 
and \textit{weakly subcritical} if $k'(0)<0, k'(1)>0$. In order to relate these definitions to the values of some moments of the random variable $\ln  f_{X_1}'(1)$,
we note that, again by Lemma \ref{mulet}, 
\begin{equation}
\label{subclassif}
\frac{k'(1)}{k(1)} = \tbs \nu_{1}(\rho) = \bb E_{\tbs \nu_{1}} \left( \rho(X_1) \right) = \bb E_{\tbs \nu_1} \left( \ln  f_{X_1}'(1) \right),
\end{equation}
where $\mathbb E_{\tbs \nu_{\ll}} $ is the expectation generated by the finite dimensional distributions 
of the Markov chain $( X_n )_{n\geq 0}$ with transition probabilities $\tbf P_{\ll}$
in the stationary regime, i.e. when the starting point $X_0$ is a random variable  
distributed according to the unique positive  $\tbf P_{\ll}$-invariant probability $\tbs \nu_{\ll}.$
Since $k(1)>0$, the equivalent classification can be done 
according to the value of the expectation $\bb E_{\tbs \nu_{1}} \left( \ln \left( f_{X_1}'(1) \right) \right)$. 
When the environment is an i.i.d.\ sequence of common law $\tbs \nu$ 
we have in addition
\begin{equation}
\label{equivalence001}
\frac{k'(1)}{k(1)} =\bb E_{\tbs \nu_{1}} \left( \ln  f_{X_1}'(1) \right)
=\bb E_{\bs \nu} \left(  f_{X_1}'(1) \ln  f_{X_1}'(1)  \right)=\phi'_{\bs \nu}(1),
\end{equation}
where $\phi_{\bs \nu}(\ll)=\bb E_{\bs \nu} \left( e^{\ll\ln  f_{X_1}'(1)} \right)$, $\ll \in \bb R.$
This shows that both classifications 
(the one according to the values of $k'(1)$ and the other according to the values of $\phi'_{\bs \nu}(1)$) 
for branching processes with i.i.d.\ environment are equivalent.  
We would like to stress that, in general, the identity
\eqref{equivalence001} is not fulfilled for a Markovian environment and therefore
the function $\phi_{\bs \nu}(\ll)$ is not the appropriate one for the classification.
For a Markovian environment the classification equally can be done using the function $K'(\ll)$, 
where $K(\ll)=\ln k(\ll),$ $\ll \in \mathbb R.$

Note that by Lemma \ref{mulet} the function 
$\ll \mapsto K(\ll)$ is strictly convex. 
In the strongly and intermediately subcritical cases, this implies that $0<k(1)<1.$

The following theorem gives the asymptotic behaviour of the survival probability jointly 
with the state of the Markov chain in the strongly subcritical case.   

\begin{theorem}[Strongly subcritical case]
\label{couronne}
Assume Conditions \ref{primitif}-\ref{cathedrale} and
\[
k'(0) < 0, \qquad k'(1) < 0.
\]
Then, there exists a positive function $u$ on $\bb X$ such that for any $(i,j) \in \bb X^2$,
\[
\bb P_i \left( Z_n > 0 \,,\, X_n = j \right) \underset{n \to +\infty}{\sim} k(1)^n v_1(i)u(j).
\]
\end{theorem}

Recall that $v_1$ is the eigenfunction of the transfer operator $\bb P_{1}$ 
(see also Section \ref{nenuphar} eq. \eqref{totem} for details).
Note also that in the formulation of the Theorem \ref{couronne} we can drop the assumption $k'(0) < 0$, 
since it is implied by the assumption $k'(1) < 0$, by strict convexity of $K(\ll)$. The corresponding result in the case when the environment is i.i.d.\ has been established by 
Guivarc'h and Liu \cite{guivarch_proprietes_2001} 
under some moment assumptions on the random variable
$\rho(X_1)=\ln \left( f_{X_1}'(1) \right)$. 
Our result extends \cite{guivarch_proprietes_2001} to finite dependent environments.

A break trough in determining the behaviour
of the survival probability for intermediate subcritical and weakly subcritical cases
for branching processes with i.i.d.\ environment was made by  
Geiger, Kersting and Vatutin \cite{geiger_limit_2003}.
Note that the original results in \cite{geiger_limit_2003}
have been established under some moment assumptions
on the random variable $\rho(X_1)=\ln \left( f_{X_1}'(1) \right)$. 
For these two cases and finite Markovian environments
 we give below the asymptotic of the survival probability jointly with
 the state of the Markov chain.

\begin{theorem}[Intermediate subcritical case]
\label{sceptre}
Assume Conditions \ref{primitif}-\ref{cathedrale} and
\[
k'(0) < 0, \qquad k'(1) = 0.
\]
Then, there exists a positive function $u$  on $\bb X$ such that for any $(i,j) \in \bb X^2$,
\[
\bb P_i \left(  Z_n > 0 \,,\, X_n = j \right) \underset{n \to +\infty}{\sim} k(1)^n \frac{v_1 (i) u(j)}{\sqrt{n}}.
\]
\end{theorem}

As in the previous Theorem \ref{couronne},  $k'(1) = 0$ implies the assumption $k'(0) < 0$, since the function 
$\ll \mapsto K(\ll) = \ln (k(\ll))$ is strictly convex (see Lemma \ref{mulet}).

\begin{theorem}[Weakly subcritical case]
\label{cape}
Assume Conditions \ref{primitif}-\ref{cathedrale} and
\[
k'(0) < 0, \qquad k'(1) > 0.
\]
Then, there exist a unique $\ll \in (0,1)$ satisfying $k'(\ll)=0$ and a positive function $u$  on $\bb X^2$ such that for any $(i,j) \in \bb X^2$,
\[
\bb P_i \left( Z_n > 0 \,,\, X_n = j \right) \underset{n \to +\infty}{\sim} k(\ll)^n \frac{u(i,j)}{n^{3/2}}.
\]
\end{theorem}

The existence and the unicity of $\ll \in (0,1)$ satisfying $k'(\ll)=0$ and $0< k(\ll) < 1$ in Theorem \ref{cape} is an obvious consequence of the strict convexity of $K$. Note that Theorems \ref{prince} , \ref{couronne}, \ref{sceptre} and \ref{cape} give the asymptotic behaviour of the joint probabilities $\bb P_i \left( Z_n > 0 \,,\, X_n = j\right)$. By summing both sides of the corresponding equivalences in $j$ we obtain the asymptotic behaviour of the survival probability $\mathbb P_i \left( Z_n>0 \right)$. The corresponding results for the survival probability when the Markovian environment is in the stationary regime are easily obtained by integrating the previous ones with respect to the invariant measure $\bs \nu$.

\section{Preliminary results on the associated Markov walk}
\label{prliminrez}

The aim of this section is to provide necessary assertions on the Markov chain $(X_n)_{n\geq 0}$ and on the associated Markov walk $(S_n)_{n\geq 0}$ defined by \eqref{petale} and to relate them to the survival probability of  $(Z_n)_{n\geq 0}$ at generation $n$. For the ease of the reader we recall the outline of the section:
\begin{itemize}
\item Subsection \ref{etang}: Relate the branching process $(Z_n)_{n\geq 0}$ to the associated Markov walk $(S_n)_{n\geq 0}$. 
\item Subsection \ref{batailleBP}: Construct the dual Markov chain $( X_n^* )_{n\geq 0}$. 
\item Subsection \ref{flamme}: Recall results on the Markov walks  conditioned to stay positive.
\item Subsection \ref{nenuphar}: Introduce the transfer operator of the Markov chain $(X_n)_{n\geq 0}$ and the change of the probability measure. State the properties of the associated Markov walk $(S_n)_{n\geq 0}$ under the changed measure.
\end{itemize}

\subsection{The link between the branching process and the associated Markov walk}
\label{etang}
In this section we recall some identities on the branching process. Some of them are stated for the commodity of the reader and are merely adaptations to the Markovian environments of the well-known statements in the i.i.d.\ case.

The first one is a representation of the conditioned probability generating function given the environment:

\begin{lemma}[Conditioned generating function]
\label{CondGenFun}
For any $s\in [0,1]$ and $n \geq 1$,
\[
\bb E_i \left( \sachant{s^{Z_{n}}}{X_1, \dots, X_{n} } \right) = f_{X_1} \circ \dots \circ f_{X_n} (s).
\]
\end{lemma}

\begin{proof}
For all $s\in[0,1]$, $n \geq 1$, $(z_1, \dots, z_{n-1}) \in \bb N^{n-1}$ and $(i_1, \dots,i_{n}) \in \bb X^{n}$, by \eqref{roseau}, we have
\[
\bb E_i \left( \sachant{s^{Z_{n}} }{Z_1=z_1, \dots, Z_{n-1}=z_{n-1}, X_1=i_1, \dots, X_{n} = i_{n}} \right) = \bb E \left( s^{\sum_{j=1}^{z_{n-1}} \xi_{i_{n}}^{n,j}} \right).
\] 
Since $\left( \xi_{i_{n}}^{n,j} \right)_{j \geq 1}$ are i.i.d., by \eqref{jazz},
\[
\bb E_i \left( \sachant{s^{Z_{n}}}{Z_1=z_1, \dots, Z_{n-1}=z_{n-1}, X_1=i_1, \dots, X_{n}=i_{n}} \right) = f_{i_{n}} (s)^{z_{n-1}}.
\]
From this we get,
\[
\bb E_i \left( \sachant{s^{Z_{n}}}{X_1=i_1, \dots, X_{n} = i_n} \right) 
= \bb E_i \left( \sachant{ f_{i_n} (s)^{Z_{n-1}} }{X_1=i_1, \dots, X_{n-1}=i_{n-1}} \right).
\]
By induction, for any $(i_1, \dots,i_{n}) \in \bb X^{n}$,
\[
\bb E_i \left( \sachant{s^{Z_{n}}}{X_1=i_1, \dots, X_{n}=i_{n} } \right) = f_{i_1} \circ \dots \circ f_{i_n} (s).
\]
and the assertion of the lemma follows.
\end{proof}

For any $n \geq 1$ and $s\in [0,1]$ set
\begin{equation}
	\label{jeux}
	q_n(s) := 1- f_{X_1} \circ \dots \circ f_{X_n} (s) \qquad \text{and} \qquad q_n := q_n(0).
\end{equation}
Lemma \ref{CondGenFun} implies that
\begin{equation}
	\label{ange}
	\bb P_i \left( \sachant{Z_n > 0}{ X_1, \dots, X_{n} } \right) = q_n.
\end{equation}
Taking the expectation in \eqref{ange}, we obtain the  well-known equality, which will be the starting point for our study:
\begin{equation}
	\label{rose}
	\bb P_i \left( Z_n > 0\right) = \bb E_i \left( q_n \right).
\end{equation}
Under Condition \ref{eglise}, for any $i\in \bb X$ and $s\in [0,1)$, we have $f_i(s) \in [0,1)$. 
Therefore $f_{X_1} \circ \cdots \circ f_{X_n} (s) \in [0,1)$ and in particular
\begin{equation}
	\label{histoire}
	q_n \in (0,1], \qquad  \forall n \geq 1.
\end{equation}
Introduce some additional notations, which will be used all over the paper:
\begin{align}
	\label{montre001}
	 f_{k,n} := f_{X_k} \circ \cdots \circ f_{X_n}, &\qquad  \forall n \geq 1,\;  \forall k \in \{1,\dots,n\},\\
	\label{montre002}
	 f_{n+1,n} := \id, &\qquad   \forall n \geq 1, \\
	\label{montre003}
	g_i(s) := \frac{1}{1-f_i(s)} - \frac{1}{f_i'(1)(1-s)},  &\qquad   \forall i \in \bb X,\;  \forall s \in [0,1), \\
	\label{montre004}
	\eta_{k,n}(s) := g_{X_k} \left( f_{k+1,n}(s) \right), &\qquad  \forall n \geq 1,\;  \forall k \in \{1,\dots,n\},\;  \forall s \in [0,1), \\
	\label{montre005}
	\eta_{k,n} := \eta_{k,n}(0) = g_{X_k} \left( f_{k+1,n}(0) \right), &\qquad  \forall n \geq 1,\;  \forall k \in \{1,\dots,n\} .
\end{align}

The key point in proving our main results is the following assertion which relies the random variable $q_n(s)$ 
to the associated Markov walk $(S_n)_{n\geq 0}$, see \eqref{petale}. 
This relation is known from Agresti \cite{agresti_bounds_1974}  in the case of linear fractional generating functions.
It turned out to be very useful for studying general branching processes 
and was generalized in Geiger and Kersting \cite{geiger_survival_2001}. 
We adapt their argument to the case when the environment is Markovian.

\begin{lemma}
\label{foin}
For any $s\in [0,1)$ and $n \geq 1$,
\[
q_n(s)^{-1} = \frac{\e^{-S_n}}{1-s} + \sum_{k=0}^{n-1} \e^{-S_k} \eta_{k+1,n}(s).
\]
\end{lemma}

\begin{proof}
With the notations \eqref{montre002}-\eqref{montre005} we write for any $s\in [0,1)$ and $n \geq 1$,
\begin{align*}
	q_n(s)^{-1} &:= \frac{1}{1-f_{X_1} \circ \cdots \circ f_{X_n} (s)} \nonumber\\
	&= \frac{1}{1-f_{1,n} (s)} \nonumber\\
	&= g_{X_1} \left( f_{2,n}(s) \right) + \frac{f_{X_1}'(1)^{-1}}{1-f_{2,n}(s)} \nonumber\\
	&= \dots \nonumber\\
	&= \frac{\left( f_{X_1}'(1) \cdots f_{X_n}'(1) \right)^{-1}}{1-s} + g_{X_1} \left( f_{2,n}(s) \right) + \sum_{k=2}^{n} \left( f_{X_1}'(1) \cdots f_{X_{k-1}}'(1) \right)^{-1} g_{X_k} \left( f_{k+1,n}(s) \right) \nonumber\\
	&= \frac{\e^{-S_n}}{1-s} + \sum_{k=0}^{n-1} \e^{-S_k} \eta_{k+1,n}(s).
\end{align*}
\end{proof}

Taking $s=0$ in Lemma \ref{foin} we obtain the following identity which will play the central role in the proofs:
\begin{equation}
	\label{ciel}
	q_n^{-1} = \e^{-S_n} + \sum_{k=0}^{n-1} \e^{-S_k} \eta_{k+1,n}, \qquad \forall n \geq 1.
\end{equation}
Since $f_i$ is convex on $[0,1]$ for all $i \in \bb X$, the function $g_i$ is non-negative,
\begin{equation}
	\label{yeux}
	g_i(s) = \frac{f_i'(1)(1-s) - \left( 1-f_i(s) \right)}{\left( 1-f_i(s) \right)f_i'(1)(1-s)} \geq 0, \qquad  \forall s \in [0,1),
\end{equation}
which, in turn, implies that the random variables $\eta_{k+1,n}$ are non-negative for any $n \geq 1$ and $k \in \{0, \dots,n-1 \}$.

\begin{lemma} Assume Condition \ref{eglise}.
\label{pieuvre}
For any $n \geq 2$, $( i_1, \dots, i_n) \in \bb X^n$ and $s\in[0,1)$, we have
\[
0 \leq g_{i_1} \left( f_{i_2} \circ \cdots \circ f_{i_n} (s) \right) \leq \eta := \max_{i \in \bb X} \frac{f_i''(1)}{f_i'(1)^2} < +\infty.
\]
Moreover, for any $( i_n )_{n \geq 1} \in \bb X^{\bb N^*}$, $s\in[0,1)$ and any $k \geq 1$,
\begin{equation}
	\label{dorade}
	\lim_{n\to+\infty} g_{i_k} \left( f_{i_{k+1}} \circ \cdots \circ f_{i_n} (s) \right)  \in [0,\eta].
\end{equation}
\end{lemma}

\begin{proof}
Fix $( i_n )_{n \geq 1} \in \bb X^{\bb N^*}$. For any $i \in \bb X$ and $s \in [0,1)$, we have $f_i (s) \in [0,1)$. So $f_{i_2} \circ \cdots \circ f_{i_n} (s) \in [0,1)$. In addition, by \eqref{yeux}, $g_i$ is non-negative on $[0,1)$ for any $i \in \bb X$, therefore $g_{i_1} \left( f_{i_2} \circ \cdots \circ f_{i_n} (s) \right) \geq 0$. Moreover by the lemma 2.1 of \cite{geiger_survival_2001},
for any $i \in \bb X$ and any $s\in [0,1)$,
\begin{equation}
	\label{yeux002}
	g_i(s) \leq \frac{f_i''(1)}{f_i'(1)^2}.
\end{equation}
By Condition \ref{eglise}, $\eta < +\infty$ and so $g_{i_1} \left( f_{i_2} \circ \cdots \circ f_{i_n} (s) \right) \in [0,\eta]$, for any $s\in[0,1)$.

Since $f_i$ is increasing on $[0,1)$ for any $i \in \bb X$, it follows that for any $k \geq 1$ and any $n \geq k+1$,
\[
0 \leq f_{i_{k+1}} \circ \cdots \circ f_{i_n} (s) \leq f_{i_{k+1}} \circ \cdots \circ f_{i_n} \circ f_{i_{n+1}} (s) \leq 1,
\]
and the sequence $\left( f_{i_{k+1}} \circ \cdots \circ f_{i_n} (s) \right)_{n\geq k+1}$ converges to a limit, say $l \in [0,1]$. For any $i \in \bb X$, the function $g_i$ is continuous on $[0,1)$ and we have
\begin{align}
	\lim_{\substack{s\to 1\\s<1}} g_i(s) 
	&= \lim_{\substack{s\to 1\\s<1}} \frac{f_i'(1)(1-s) - \left( 1-f_i(s) \right)}{f'(1)\left( 1-f_i(s) \right)(1-s)} \nonumber\\
	&= \lim_{\substack{s\to 1\\s<1}} \frac{1}{f_i'(1)} \frac{f_i(s) - 1 - f_i'(1)(s-1)}{(s-1)^2} \frac{1-s}{ 1-f_i(s) } \nonumber\\
	&= \frac{1}{f_i'(1)} \frac{f_i''(1)}{2} \frac{1}{ f_i'(1) } = \frac{f_i''(1)}{2 f_i'(1)^2} <+\infty.
	\label{champ}
\end{align}
Denoting $g_i(l) = \frac{f_i''(1)}{2 f_i'(1)^2}$ if $l=1$, we conclude that $g_{i_k} \left( f_{i_{k+1}} \circ \cdots \circ f_{i_n} (s) \right)$ converges to $g_{i_k}(l)$ as $n \to +\infty$. By \eqref{yeux} and \eqref{yeux002}, we obtain that $g_{i_k}(l) \in [0,\eta]$.
\end{proof}

\subsection{The dual Markov walk}
\label{batailleBP}

We will introduce the dual Markov chain $(X_n^*)_{n\geq 0}$ and the associated dual Markov walk $(S_n^*)_{n\geq 0}$
and state some of their properties.

Since $\bs \nu$ is positive on $\bb X$, the following dual Markov kernel $\bf P^*$ is well defined:
\begin{equation}
\label{statueBP}
\bf P^* \left( i,j \right) = \frac{\bs \nu \left( j \right)}{\bs \nu (i)} \bf P \left( j,i \right), \quad \forall (i,j) \in \bb X^2.
\end{equation}
Let $\left( X_n^* \right)_{n\geq 0}$ be a dual Markov chain, independent of the chain $\left( X_n \right)_{n\geq 0}$, defined on $(\Omega, \scr F, \bb P)$, living on $\bb X$  and with transition probability $\bf P^*$. We define the dual Markov walk by
\begin{equation}
\label{promenade001}
S_0^* = 0 \qquad \text{and} \qquad S_n^* = -\sum_{k=1}^n \rho \left( X_k^* \right), \quad \forall n \geq 1.
\end{equation}
For any $z\in \bb R$, let $\tau_z^*$ be the associated exit time:
\begin{equation}
\label{promenade002}
\tau_z^* := \inf \left\{ k \geq 1 : z+S_k^* \leq 0 \right\}.
\end{equation}
For any $i\in \bb X$, denote by $\bb P_i^*$ and $\bb E_i^*$ the probability, respectively the expectation generated by the finite dimensional distributions of the Markov chain $( X_n^* )_{n\geq 0}$ starting at $X_0^* = i$.

It is easy to see that $\bs \nu$ is also $\bf P^*$-invariant and for any $n \geq 1$, $(i,j) \in \bb X^2$, 
\[
\left(\bf  P^* \right)^n (i,j) = \bf P^n (j,i) \frac{\bs \nu(j)}{\bs \nu(i)}.
\]
This last formula implies in particular the following result.

\begin{lemma}
\label{sourire}
Assume Conditions \ref{primitif} and \ref{cathedrale} for the Markov kernel $\bf P$. Then Conditions \ref{primitif} and \ref{cathedrale} hold also for dual kernel $\bf P^*$.
\end{lemma}

Similarly to \eqref{soeur}, we have for any $(i,j) \in \bb X^2$,
\begin{equation}
\label{vautour}
\abs{\left( \bf P^* \right)^n (i,j) - \bs \nu (j)} \leq  c\e^{-cn}.
\end{equation}

Note that the operator $\bf P^*$ is the adjoint of $\bf P$ in the space $\LL^2 \left( \bs \nu \right) :$ for any functions $f$ and $g$ on $\bb X,$
\[
\bs \nu \left( f \left(\bf P^*\right)^n g \right) = \bs \nu \left( g \bf P^n f \right).
\]
For any measure $\mathfrak{m}$ on $\bb X$, let $\bb E_{\mathfrak{m}}$ (respectively $\bb E_{\mathfrak{m}}^*$) be the expectation associated to the probability generated by the finite dimensional distributions of the Markov chain $\left( X_n \right)_{n\geq 0}$ (respectively $\left( X_n^* \right)_{n\geq 0}$) with the initial law $\mathfrak{m}$.

\begin{lemma}[Duality]
\label{dualityBP}
For any probability measure $\mathfrak{m}$ on $\bb X$, any $n\geq 1$ and any function $g$: $\bb X^n \to \bb C$,
\[
\bb E_{\mathfrak{m}} \left( g \left( X_1, \dots, X_n \right) \right) = \bb E_{\bs \nu}^* \left( g \left( X_n^*, \dots, X_1^* \right) \frac{\mathfrak{m} \left( X_{n+1}^* \right)}{\bs \nu \left( X_{n+1}^* \right)} \right).
\]
Moreover, for any $n\geq 1$ and any function $g$: $\bb X^n \to \bb C$,
\[
\bb E_i \left( g \left( X_1, \dots, X_n \right) \,;\, X_{n+1} = j \right) = \bb E_j^* \left( g \left( X_n^*, \dots, X_1^* \right) \,;\, X_{n+1}^* = i \right) \frac{\bs \nu(j)}{\bs \nu(i)}.
\]
\end{lemma}

\begin{proof}
The first equality is proved in Lemma 3.2 of \cite{GLLP_CLLT_2017}. 
The second can be deduced from the first as follows. Taking $\mathfrak{m} = \bs \delta_i$ and $\tt g(i_1,\cdots,i_n,i_{n+1}) = g(i_1,\cdots,i_n)\bbm 1_{\{ i_{n+1} = j \}}$, 
from the first equality of the lemma, we see that
\begin{align*}
	\bb E_i \left( g \left( X_1, \dots, X_n \right) \,;\, X_{n+1} = j \right) &= \bb E_{\bs \nu}^* \left( \tt g \left( X_{n+1}^*, \dots, X_1^* \right) \,;\, X_{n+2}^* = i \right) \frac{1}{\bs \nu(i)} \\
	&= \bb E_{\bs \nu}^* \left( g \left( X_{n+1}^*, \dots, X_2^* \right) \,;\, X_1^* = j \,,\, X_{n+2}^* = i \right) \frac{1}{\bs \nu(i)}.
\end{align*}
Since $\bs \nu$ is $\bf P^*$-invariant, we obtain
\begin{align*}
	\bb E_i \left( g \left( X_1, \dots, X_n \right) \,;\, X_{n+1} = j \right) &= \sum_{i_1 \in \bb X} \bb E_{i_1}^* \left( g \left( X_n^*, \dots, X_1^* \right) \,;\, X_{n+1}^* = i \right) \frac{1}{\bs \nu(i)} \bbm 1_{\{ i_1 = j \}} \bs \nu(i_1) \\
	&=\bb E_j^* \left( g \left( X_n^*, \dots, X_1^* \right) \,;\, X_{n+1}^* = i \right) \frac{\bs \nu(j)}{\bs \nu(i)}. 
\end{align*}
\end{proof}

\subsection{Markov walks conditioned to stay positive}
\label{flamme}

In this section we recall the main results from \cite{grama_limit_2016-1}  and \cite{GLLP_CLLT_2017} for Markov walks conditioned to stay positive. We complement these results by some new assertions which will be used in the proofs.

For any $y \in \bb R$ define the first time when the Markov walk $\left( S_n \right)_{n\geq 0}$ becomes non-positive by setting
\[
\tau_y := \inf \left\{ k \geq 1 : y+S_k \leq 0 \right\}.
\]
Under Conditions \ref{primitif}, \ref{cathedrale} and $\bs \nu(\rho) = 0$ the stopping time $\tau_y$ is well defined and finite $\bb P_i$-almost surely for any $i \in \bb X$. 

The following three assertions deal with the existence of the harmonic function, the limit behaviour of the probability of the exit time and of the law of the random walk $y+S_n$, conditioned to stay positive and are taken from \cite{grama_limit_2016-1}.
\begin{proposition}[Preliminary results, part I] 
Assume Conditions \ref{primitif}, \ref{cathedrale} and $\bs \nu(\rho) = 0$.
\label{sable}
There exists a non-negative function $V$ on $\bb X \times \bb R$ such that
\begin{enumerate}[ref=\arabic*, leftmargin=*, label=\arabic*.]
	\item \label{sable001} For any $(i,y) \in \bb X \times \bb R$ and $n \geq 1$,
	\[
	\bb E_i \left( V \left( X_n, y+S_n \right) \,;\, \tau_y > n \right) = V(i,y).
	\]
	\item \label{sable002} For any $i\in \bb X$, the function $V(i,\cdot)$ is non-decreasing and for any $(i,y) \in \bb X \times \bb R$,
	\[
	V(i,y) \leq c \left( 1+\max(y,0) \right).
	\]
	\item \label{sable003} For any $i \in \bb X$, $y > 0$ and $\delta \in (0,1)$,
	\[
	\left( 1- \delta \right)y - c_{\delta} \leq V(i,y) \leq \left(1+\delta \right)y + c_{\delta}.
	\]
\end{enumerate}
\end{proposition}

We define
\begin{equation}
\label{comete}
\sigma^2 := \bs \nu \left( \rho^2 \right) -  \bs \nu \left( \rho \right)^2 + 2 \sum_{n=1}^{+\infty} \left[ \bs \nu \left( \rho \bf P^n \rho \right) -  \bs \nu \left( \rho \right)^2 \right].
\end{equation}
It is known that under Conditions \ref{primitif} and \ref{cathedrale} we have $\sigma^2 > 0$, see Lemma 10.3 in \cite{GLLP_CLLT_2017}.

\begin{proposition}[Preliminary results, part II] 
Assume Conditions \ref{primitif}, \ref{cathedrale} and $\bs \nu(\rho) = 0$.
\begin{enumerate}[ref=\arabic*, leftmargin=*, label=\arabic*.]
\label{oreiller}
	\item \label{oreiller001} For any $(i,y) \in \bb X \times \bb R$,
	\[
	\lim_{n\to +\infty} \sqrt{n} \bb P_i \left( \tau_y > n \right) = \frac{2V(i,y)}{\sqrt{2\pi} \sigma},
	\]
	where $\sigma$ is defined by \eqref{comete}.
	\item \label{oreiller002} For any $(i,y) \in \bb X \times \bb R$ and $n\geq 1$,
	\[
	\bb P_i \left( \tau_y > n \right) \leq c\frac{ 1 + \max(y,0) }{\sqrt{n}}.
	\]
\end{enumerate}
\end{proposition}

We denote by $\supp(V) = \left\{ (i,y) \in \bb X \times \bb R : \right.$ $\left. V(i,y) > 0 \right\}$ the support of the function $V$. Note that from property \ref{sable003} of Proposition \ref{sable}, for any fixed $i\in \bb X$, the function $y \mapsto V(i,y)$ is positive for large $y$. For more details on the properties of $\supp (V)$ see \cite{grama_limit_2016-1}.

\begin{proposition}[Preliminary results, part III]
\label{racine}
Assume Conditions \ref{primitif}, \ref{cathedrale} and $\bs \nu(\rho) = 0$.
\begin{enumerate}[ref=\arabic*, leftmargin=*, label=\arabic*.]
	\item \label{racine001} For any $(i,y) \in \supp(V)$ and $t\geq 0$,
	\[
	\bb P_i \left( \sachant{\frac{y+S_n}{\sigma \sqrt{n}} \leq t }{\tau_y >n} \right) \underset{n\to+\infty}{\longrightarrow} \mathbf \Phi^+(t),
	\]
	where $\bf \Phi^+(t) = 1-\e^{-\frac{t^2}{2}}$ is the Rayleigh distribution function.
	\item \label{racine002} There exists $\ee_0 >0$ such that, for any $\ee \in (0,\ee_0)$, $n\geq 1$, $t_0 > 0$, $t\in[0,t_0]$ and $(i,y) \in \bb X \times \bb R$,
	\[
	\abs{ \bb P_i \left( y+S_n \leq t \sqrt{n} \sigma \,,\, \tau_y > n \right) - \frac{2V(i,y)}{\sqrt{2\pi n}\sigma} \bf \Phi^+(t) } \leq c_{\ee,t_0} \frac{\left( 1+\max(y,0)^2 \right)}{n^{1/2+\ee}}.
	\]
\end{enumerate}
\end{proposition}

The next assertions are two local limit theorems for the associated Markov walk $y+S_n$ from  \cite{GLLP_CLLT_2017}.

\begin{proposition}[Preliminary results, part IV] 
Assume Conditions \ref{primitif}, \ref{cathedrale} and $\bs \nu(\rho) = 0$.
\label{goliane}
\begin{enumerate}[ref=\arabic*, leftmargin=*, label=\arabic*.]
\item \label{liane} For any $i \in \bb X$, $a>0$, $y \in \bb R$, $z \geq 0$ and any non-negative function $\psi$: $\bb X \to \bb R_+$,
\begin{align*}
	\lim_{n\to +\infty} n^{3/2} &\bb E_i \left( \psi(X_n) \,;\, y+S_n \in [z,z+a] \,,\, \tau_y > n \right) \\
	&\qquad = \frac{2V(i,y)}{\sqrt{2\pi}\sigma^3} \int_z^{z+a} \bb E_{\bs \nu}^* \left( \psi(X_1^*) V^*\left( X_1^*, z'+S_1^* \right) \,;\, \tau_{z'}^* > 1 \right) \dd z'.
\end{align*}
\item \label{gorilleBP} Moreover, for any $a>0$, $y \in \bb R$, $z \geq 0$, $n \geq 1$ and any non-negative function $\psi$: $\bb X \to \bb R_+$,
\begin{align*}
	\sup_{i\in \bb X} \bb E_i &\left( \psi(X_n) \,;\, y+S_{n} \in [z,z+a] \,,\, \tau_y > n \right) \leq \frac{c \left( 1+a^3 \right)}{n^{3/2}} \norm{\psi}_{\infty} \left( 1+z \right)\left( 1+\max(y,0) \right).
\end{align*}
\end{enumerate}
\end{proposition}

Recall that the dual chain $( X_n^* )_{n\geq 0}$ is constructed independently of the chain $( X_n )_{n\geq 0}$. For any $(i,j) \in \bb X^2$, the probability generated by the finite dimensional distributions of the two dimensional Markov chain $(X_n,X_n^*)_{n\geq 0}$ starting at $(X_0,X_0^*)=(i,j)$ is given by $\bb P_{i,j} = \bb P_i \times \bb P_j$. Let $\bb E_{i,j}$ be the corresponding expectation. For any $l \geq 1$ we define $\scr C^+ \left( \bb X^l \times \bb R_+ \right)$ the set of non-negative function $g$: $\bb X^l \times \bb R_+ \to \bb R_+$ satisfying the following properties:
\begin{itemize}
\item for any $(i_1,\dots,i_l) \in \bb X^l$, the function $z \mapsto g(i_1,\dots,i_l,z)$ is continuous,
\item there exists $\ee >0$ such that $\max_{i_1,\dots i_l \in \bb X} \sup_{z\geq 0} g(i_1,\dots,i_l,z) (1+z)^{2+\ee} < +\infty$.
\end{itemize}

\begin{proposition}[Preliminary results, part V] Assume Conditions \ref{primitif}, \ref{cathedrale} and $\bs \nu(\rho) = 0$.
\label{sorcier}
For any $i \in \bb X$, $y \in \bb R$, $l \geq 1$, $m \geq 1$ and $g \in \scr C^+ \left( \bb X^{l+m} \times \bb R_+ \right)$,
\begin{align*}
&\lim_{n\to +\infty} n^{3/2} \bb E_i \left( g \left(X_1, \dots, X_l, X_{n-m+1}, \dots, X_n, y+S_n \right) \,;\, \tau_y > n \right) \\
&\qquad = \frac{2}{\sqrt{2\pi}\sigma^3} \int_0^{+\infty} \sum_{j \in \bb X} \bb E_{i,j} \left( g \left( X_1, \dots, X_l,X_m^*,\dots,X_1^*,z \right) \right. \\
&\hspace{4cm} \left. \times V \left( X_l, y+S_l \right) V^* \left( X_m^*, z+S_m^* \right) \,;\, \tau_y > l \,,\, \tau_z^* > m \right) \bs \nu(j) \dd z.
\end{align*}
\end{proposition}

We complete these results by determining the asymptotic behaviour of the law of the Markov chain $(X_n)_{n\geq 1}$ 
jointly with $\{ \tau_y > n\}.$

\begin{lemma} \label{moustique}
Assume Conditions \ref{primitif}, \ref{cathedrale} and $\bs \nu(\rho) = 0$. 
Then, for any $(i,y) \in \bb X \times \bb R$ and $j \in \bb X$, we have
\[
\lim_{n\to +\infty} \sqrt{n} \bb P_{i} \left( X_n = j \,,\, \tau_y > n \right) = \frac{2V(i,y) \bs \nu (j)}{\sqrt{2\pi} \sigma}. 
\]
\end{lemma}

\begin{proof}
Fix $(i,y) \in \bb X \times \bb R$ and $j \in \bb X$. We will prove that 
\begin{align*}
	\frac{2V(i,y) \bs \nu (j)}{\sqrt{2\pi} \sigma}  &\leq \liminf_{n\to+\infty} \sqrt{n} \bb P_i \left( X_n = j  \,,\, \tau_y > n  \right) \\
&\leq \limsup_{n\to+\infty} \sqrt{n} \bb P_i \left( X_n = j  \,,\, \tau_y > n  \right) \leq \frac{2V(i,y) \bs \nu (j)}{\sqrt{2\pi} \sigma}.
\end{align*}

\textit{The upper bound.} By the Markov property, for any $n \geq 1$ and $k=\pent{n^{1/4}}$ we have
\[
\bb P_i \left( X_n = j \,,\, \tau_y > n \right) \leq \bb P_i \left( X_n = j \,,\, \tau_y > n-k \right) = \bb E_i \left( \bf P^k \left( X_{n-k},j \right) \,;\, \tau_y > n-k \right).
\]
Using \eqref{soeur}, we obtain that
\[
\bb P_i \left( X_n = j \,,\, \tau_y > n \right) \leq \left( \bs \nu(j) + c\e^{-ck} \right) \bb P_i \left( \tau_y > n-k \right).
\]
Using the point \ref{oreiller001} of Proposition \ref{oreiller} and the fact that $k=\pent{n^{1/4}}$,
\begin{equation}
	\label{mare}
	\limsup_{n\to +\infty} \sqrt{n} \bb P_i \left( X_n = j \,,\, \tau_y > n \right) \leq \frac{2V(i,y) \bs \nu (j)}{\sqrt{2\pi} \sigma}.
\end{equation}

\textit{The lower bound.} Again, let $n \geq 1$  and $k=\pent{n^{1/4}}$. We have
\begin{equation}
	\label{canneton}
	\bb P_i \left( X_n = j \,,\, \tau_y > n \right) \geq \bb P_i \left( X_n = j \,,\, \tau_y > n-k \right) - \bb P_i \left( n-k < \tau_y \leq n \right).
\end{equation}
As for the upper bound, using the Markov property and \eqref{soeur},
\[
\bb P_i \left( X_n = j \,,\, \tau_y > n-k \right) = \bb E_i \left( \bf P^k \left( X_{n-k}, j \right) \,;\, \tau_y > n-k \right) \geq \left( \bs \nu(j) - c\e^{-ck} \right) \bb P_i \left( \tau_y > n-k \right).
\]
Using the point \ref{oreiller001} of Proposition \ref{oreiller} and using the fact that $k=\pent{n^{1/4}}$,
\begin{equation}
	\liminf_{n\to+\infty} \sqrt{n} \bb P_i \left( X_n = j \,,\, \tau_y > n-k \right) \geq \frac{2V(i,y) \bs \nu (j)}{\sqrt{2\pi} \sigma}.
	\label{canard}
\end{equation}
Furthermore, on the event $\left\{ n-k < \tau_y \leq n \right\}$, we have
\[
0 \geq \min_{n-k < i \leq n} y+S_i \geq y+S_{n-k} - k \norm{\rho}_{\infty},
\]
where $\norm{\rho}_{\infty}$ is the maximum of $\abs{\rho}$ on $\bb X$. Consequently,
\begin{align*}
	\bb P_i \left( n-k < \tau_y \leq n \right) &\leq \bb P_i \left( y+S_{n-k} \leq c k \,,\, \tau_y > n-k \right) \\
	&= \bb P_i \left( y+S_{n-k} \leq \frac{c k}{\sqrt{n-k}} \sqrt{n-k} \,,\, \tau_y > n-k \right).
\end{align*}
Now, using the point \ref{racine002} of Proposition \ref{racine} with $t_0 = \max_{n\geq 1} \frac{ck}{\sqrt{n-k}}$, we obtain that, for $\ee > 0$ small enough,
\[
\bb P_i \left( n-k < \tau_y \leq n \right) \leq \frac{2V(i,y)}{\sqrt{2\pi (n-k)}\sigma} \left( 1-\e^{-\frac{ck^2}{2(n-k)}} \right) + c_{\ee} \frac{\left( 1+y^2 \right)}{(n-k)^{1/2+\ee}}.
\]
Therefore, since $k = \pent{n^{1/4}}$,
\begin{equation}
	\label{canne}
	\lim_{n\to+\infty} \sqrt{n} \bb P_i \left( n-k < \tau_y \leq n \right) = 0.
\end{equation}
Putting together \eqref{canneton}, \eqref{canard} and \eqref{canne}, we conclude that
\[
\liminf_{n\to+\infty} \sqrt{n} \bb P_i \left( X_n = j \,,\, \tau_y > n \right) \geq \frac{2V(i,y) \bs \nu (j)}{\sqrt{2\pi} \sigma},
\]
which together with \eqref{mare} concludes the proof of the lemma.
\end{proof}

Now, with the help of the function $V$ from Proposition \ref{sable}, for any $(i,y) \in \supp(V)$, 
we define a new probability $\bb P_{i,y}^+$  on $\sigma\left( X_n, n \geq 1 \right)$ 
and the corresponding expectation $\bb E_{i,y}^+$, 
which are characterized by the following property: 
for any $n \geq 1$ and any $g$: $\bb X^n \to \bb C$,
\begin{equation}
	\label{soif}
	\bb E_{i,y}^+ \left( g \left( X_1, \dots, X_n \right) \right) := \frac{1}{V(i,y)} \bb E_i \left( g\left( X_1, \dots, X_n \right) V\left( X_n, y+S_n \right) \,;\, \tau_y > n \right).
\end{equation}
The fact that $\bb P_{i,y}^+$ is a probability measure and that it does not depend on  $n$ 
follows easily from the point \ref{sable001} of Proposition \ref{sable}.
The probability $\bb P_{i,y}^+$  is extended obviously to the hole 
probability space $\left( \Omega, \scr F, \bb P \right)$.
The corresponding expectation is again denoted by $\bb E_{i,y}^+$.

\begin{lemma} \label{cumulus}
Assume Conditions \ref{primitif}, \ref{cathedrale} and $\bs \nu(\rho) = 0$.
Let $m \geq 1$. For any $n \geq 1$, bounded measurable function $g$: $\bb X^m \to \bb C$, $(i,y) \in \supp(V)$ and $j \in \bb X$,
\[
\lim_{n\to +\infty} \bb E_i \left( \sachant{g\left( X_1, \dots, X_m \right) 
\,;\, X_n = j}{ \tau_y > n } \right) 
= \bb E_{i,y}^+ \left( g\left( X_1, \dots, X_m \right) 
\right) \bs \nu (j).
\]
\end{lemma}

\begin{proof}
For the sake of brevity, for any $(i,j) \in \bb X^2$, $y \in \bb R$ and $n \geq 1$, set
\[
J_n(i,j,y) := \bb P_i \left( X_n = j \,,\, \tau_y > n \right).
\]
Fix $m \geq 1$ and let $g$ be a function $\bb X^m \to \bb C$. By the point \ref{oreiller001} of Proposition \ref{oreiller}, it is clear that for any $(i,y) \in \supp(V)$ and $n$ large enough, $\bb P_i \left( \tau_y > n \right) > 0$. By the Markov property, for any $j \in \bb X$ 
and $n \geq m+1$ large enough,
\begin{align*}
I_0 &:= \bb E_i \left( \sachant{g\left( X_1, \dots, X_m \right)
 \,;\, X_n = j}{ \tau_y > n } \right)\\ 
&= \bb E_i \left( g\left( X_1, \dots, X_m \right) 
\frac{J_{n-m} \left( X_m,j,y+S_m \right)}{\bb P_i \left( \tau_y > n \right)} \,;\, \tau_y > m \right).
\end{align*}
Using Lemma \ref{moustique} and the point \ref{oreiller001} of Proposition \ref{oreiller}, by the Lebesgue dominated convergence theorem,
\begin{align*}
	\lim_{n\to+\infty} I_0 &= \bb E_i \left( g\left( X_1, \dots, X_m \right) \frac{V \left( X_m,y+S_m \right)}{V(i,y)} \,;\, \tau_y > m \right) \bs \nu(j) \\
	&= \bb E_{i,y}^+ \left( g\left( X_1, \dots, X_m \right) \right) \bs \nu (j).
\end{align*}

\end{proof}

\begin{lemma} \label{soir}
Assume Conditions \ref{primitif}, \ref{cathedrale} and $\bs \nu(\rho) = 0$. 
For any $(i,y) \in \supp(V)$, we have, for any $k \geq 0$,
\[
\bb E_{i,y}^+ \left( \e^{-S_k}\right) \leq \frac{c \left( 1+\max(y,0) \right)\e^{y}}{k^{3/2} V(i,y)}.
\]
In particular,
\[
\bb E_{i,y}^+ \left( \sum_{k=0}^{+\infty} \e^{-S_k} \right) \leq \frac{c \left( 1+\max(y,0) \right)\e^{y}}{V(i,y)}.
\]
\end{lemma}

\begin{proof} By \eqref{soif}, for any $k \geq 1$,
\[
\bb E_{i,y}^+ \left( \e^{-S_k}\right) = \bb E_i \left( \e^{-S_k} \frac{V\left( X_k, y+S_k \right)}{V(i,y)} \,;\, \tau_y > k \right).
\]
Using the point \ref{sable002} of Proposition \ref{sable},
\begin{align*}
	\bb E_{i,y}^+ \left( \e^{-S_k}\right) &\leq \e^{y} \bb E_i \left( \e^{-(y+S_k)} \frac{c\left( 1+\max \left(0,y+S_k \right) \right)}{V(i,y)} \,;\, \tau_y > k \right) \\
	&= \e^{y} \sum_{p=0}^{+\infty} \bb E_i \left( \e^{-(y+S_k)} \frac{c\left( 1+\max \left(0,y+S_k \right) \right)}{V(i,y)} \,;\, y+S_k \in (p,p+1] \,,\, \tau_y > k \right) \\
	&\leq \e^{y} \sum_{p=0}^{+\infty} \e^{-p} \frac{c( 1+p )}{V(i,y)} \bb P_i \left( y+S_k \in [p,p+1] \,,\, \tau_y > k \right).
\end{align*}
By the point \ref{gorilleBP} of Proposition \ref{goliane},
\begin{align*}
	\bb E_{i,y}^+ \left( \e^{-S_k}\right) &\leq \frac{c}{k^{3/2}} \sum_{p=0}^{+\infty} \e^{-p} ( 1+p )^2 \frac{\e^{y}\left( 1+\max(0,y) \right)}{V(i,y)} \\
	&= \frac{c \left( 1+\max(0,y) \right)\e^{y}}{k^{3/2}V(i,y)}.
\end{align*}
This proves the first inequality of the lemma. Summing both sides in $k$ and using the Lebesgue monotone convergence theorem, it proves also the second inequality of the lemma.
\end{proof}

\subsection{The change of measure related to the Markov walk}
\label{nenuphar}

In this section we shall establish some useful properties of the Markov chain under the exponential change of the probability measure, which will be crucial in the proofs of the results of the paper.

For any $\ll \in \bb R$, let $\bf P_{\ll}$ be the transfer operator defined on $\scr C$ by, for any $g \in \scr C$ and $i \in \bb X$,
\begin{equation}
	\label{ocean}
	\bf P_{\ll}g(i) := \bf P\left( \e^{\ll \rho} g \right)(i) = \bb E_i \left( \e^{\ll S_1} g(X_1) \right).
\end{equation}
From the Markov property, it follows easily that, for any $g \in \scr C$, $i \in \bb X$ and $n \geq 0$,
\begin{equation}
	\label{balancoire}
	\bf P_{\ll}^n g(i) = \bb E_i \left( \e^{\ll S_n} g(X_n) \right).
\end{equation}
For any non-negative function $g \geq 0$, $\ll \in \bb R$, $i\in \bb X$ and $n \geq 1$, we have
\begin{equation}
\label{pinson}
\bf P_{\ll}^n g(i) \geq \min_{x_1, \dots, x_n \in \bb X^n} \e^{\ll \left(\rho(x_1) + \cdots + \rho(x_n)\right)} \bf P^n g(i).
\end{equation}
Therefore the matrix $\bf P_{\ll}$ is primitive \textit{i.e.} satisfies the Condition \ref{primitif}. By the Perron-Frobenius theorem, there exists a positive number $k(\ll) > 0$, a positive function $v_{\ll}$ : $\bb X \to \bb R_+^*$, a positive linear form $\bs \nu_{\ll}$: $\scr C \to \bb C$ and  a linear operator $Q_{\ll}$ on $\scr C$ such that for any $g \in \scr C$, and $i \in \bb X$,
\begin{align}
	\label{marteau001}
	&\bf P_{\ll} g(i) = k(\ll)\bs \nu_{\ll}(g) v_{\ll}(i) + Q_{\ll}(g)(i), \\
	\label{marteau002}
	&\bs \nu_{\ll}\left( v_{\ll} \right) = 1 \qquad \text{and} \qquad Q_{\ll} \left(v_{\ll}\right) = \bs \nu_{\ll} \left(Q_{\ll}(g) \right) = 0,
\end{align}
where the spectral radius of $Q_{\ll}$ is strictly less than $k(\ll)$:
\begin{equation}
	\label{torrent}
\frac{\norm{Q_{\ll}^n(g)}_{\infty}}{k(\ll)^n} \leq c_{\ll} \e^{-c_{\ll}n} \norm{g}_{\infty}.
\end{equation}
Note that, in particular, $k(\ll)$ is equal to the spectral radius of $\bf P_{\ll},$ and,
moreover, $k(\ll)$  is an eigenvalue associated to the eigenvector $v_{\ll}$:
\begin{equation}
	\label{totem}
	\bf P_{\ll} v_{\ll} (i) = k(\ll) v_{\ll}(i).
\end{equation}
From \eqref{marteau001} and \eqref{marteau002}, we have for any $n \geq 1$,
\begin{equation}
	\label{psychedelique}
	\bf P_{\ll}^n g(i) = k(\ll)^n \bs \nu_{\ll}(g) v_{\ll}(i) + Q_{\ll}^n(g)(i).
\end{equation}
By \eqref{torrent}, for any $g \in \scr C$ and $i \in \bb X$,
\[
\lim_{n\to+\infty} \frac{\bf P_{\ll}^n g(i)}{k(\ll)^n} = \bs \nu_{\ll}(g) v_{\ll}(i)
\]
and so for any non-negative and non-identically zero function $g \in \scr C$ and $i \in \bb X$,
\begin{equation}
\label{colombe}
k(\ll) = \lim_{n\to+\infty} \left( \bf P_{\ll}^n g(i) \right)^{1/n} = \lim_{n\to+\infty} \bb E_i^{1/n} \left( \e^{\ll S_n} g(X_n) \right).
\end{equation}
Note that when $\ll = 0$, we have $k(0) = 1$, $v_0(i)=1$ and $\bs \nu_0(i) = \bs \nu(i)$, for any $i \in \bb X$. However, in general case, the operator $\bf P_{\ll}$ is no longer a Markov operator and we define $\tbf P_{\ll}$ for any $\ll \in \bb R$ by
\begin{equation}
	\label{lacBP}
	\tbf P_{\ll}g(i) = \frac{\bf P_{\ll}(gv_{\ll})(i)}{k(\ll)v_{\ll}(i)} = \frac{\bf P\left( \e^{\ll \rho}gv_{\ll} \right)(i)}{k(\ll)v_{\ll}(i)} = \frac{ \bb E_i \left( \e^{\ll S_1} g(X_1) v_{\ll}(X_1) \right)}{k(\ll)v_{\ll}(i)},
\end{equation}
for any $g \in \scr C$ and $i \in \bb X$. It is clear that $\tbf P_{\ll}$ is a Markov operator: by \eqref{totem},
\[
\tbf P_{\ll}v_0(i) = \frac{\bf P_{\ll}(v_{\ll})(i)}{k(\ll)v_{\ll}(i)} = 1,
\]
where for any $i \in \bb X$, $v_0(i) = 1$. Iterating \eqref{lacBP} and using \eqref{balancoire}, we see that  for any $n \geq 1$, $g \in \scr C$ and $i \in \bb X$.
\begin{equation}
	\label{horizon}
	\tbf P_{\ll}^n g(i) = \frac{\bf P_{\ll}^n(gv_{\ll})(i)}{k(\ll)^nv_{\ll}(i)} = \frac{ \bb E_i \left( \e^{\ll S_n} g(X_n) v_{\ll}(X_n) \right)}{k(\ll)^nv_{\ll}(i)}.
\end{equation}
In particular, as in \eqref{pinson},
\[
\tbf P_{\ll}^n g(i) \geq \min_{x_1, \dots, x_n \in \bb X^n} \e^{\ll \left(\rho(x_1) + \cdots + \rho(x_n)\right)}v_{\ll}(x_n) \frac{\bf P^n g(i)}{k(\ll)^nv_{\ll}(i)}.
\]
The following lemma is an easy consequence of this last inequality.

\begin{lemma}
\label{jument}
Assume Conditions \ref{primitif} and \ref{cathedrale} for the Markov kernel $\bf P$. Then for any $\ll \in \bb R$, Conditions \ref{primitif} and \ref{cathedrale} hold also for the operator $\tbf P_{\ll}$.
\end{lemma}
Using \eqref{psychedelique} and \eqref{horizon}, the spectral decomposition of $\tbf P_{\ll}$ is given by
\[
\tbf P_{\ll}^n g(i) = \bs \nu_{\ll} \left( gv_{\ll} \right)v_0(i) + \frac{Q_{\ll}^n(gv_{\ll})(i)}{k(\ll)^nv_{\ll}(i)} = \tbs \nu_{\ll}(g) v_0(i) + \tt Q_{\ll}^n(g)(i),
\]
with, for any $\ll \in \bb R$, $g \in \scr C$ and $i \in \bb X$,
\begin{equation}
	\label{ecorce}
	\tbs \nu_{\ll}(g) := \bs \nu_{\ll} \left( gv_{\ll} \right) \qquad \text{and} \qquad \tt Q_{\ll}(g)(i) := \frac{Q_{\ll}(gv_{\ll})(i)}{k(\ll)v_{\ll}(i)}.
\end{equation}
By \eqref{marteau002},
\[
\tbs \nu_{\ll} \left( \tt Q_{\ll}(g) \right) = \bs \nu_{\ll} \left( \frac{Q_{\ll}(g v_{\ll})}{k(\ll)} \right) = 0 \qquad \text{and} \qquad \tt Q_{\ll}(v_0) = \frac{Q_{\ll}(v_{\ll})(i)}{k(\ll)v_{\ll}(i)} = 0.
\]
Consequently, $\tbs \nu_{\ll}$ is the positive invariant measure of $\tbf P_{\ll}$ and since by \eqref{torrent},
\[
\norm{\tt Q_{\ll}^n(g)}_{\infty} \leq \frac{\norm{Q_{\ll}^n(gv_{\ll})}_{\infty}}{k(\ll)^n \min_{i \in \bb X} v_{\ll}} \leq c_{\ll} \e^{-c_{\ll}n} \norm{g}_{\infty},
\]
we can conclude that for any $(i,j) \in \bb X^2$,
\[
\abs{\tbf P_{\ll}^n (i,j) - \tbs \nu_{\ll}(j)} \leq c_{\ll} \e^{-c_{\ll}n}.
\]

Fix $\ll \in \bb R$ and let $\tbb P_i$ and $\tbb E_i$ be the probability, respectively the expectation, generated by the finite dimensional distributions of the Markov chain $(X_n)_{n\geq 0}$ with transition operator $\tbf P_{\ll}$ and starting at $X_0=i$. For any $n \geq 1$, $g$: $\bb X^n \to \bb C$ and $i \in \bb X$,
\begin{equation}
	\label{chandelle}
	\tbb E_i \left( g(X_1, \dots, X_n) \right) := \frac{\bb E_i \left( \e^{\ll S_n} g(X_1, \dots, X_n) v_{\ll}(X_n) \right)}{k(\ll)^n v_{\ll}(i)}.
\end{equation}

We are now interested in establishing some properties of the function $\ll \mapsto k(\ll)$ which are important to distinguish 
between the four different cases considered in this paper.
\begin{lemma}
\label{mulet}
Assume Conditions \ref{primitif} and \ref{cathedrale}. The function $\ll \mapsto k(\ll)$ is analytic on $\bb R$. Moreover the function $K$: $\ll \mapsto \ln\left( k(\ll) \right)$ is strictly convex and satisfies for any $\ll \in \bb R$,
\begin{equation}
	\label{chevalBP}
	K'(\ll) = \frac{k'(\ll)}{k(\ll)} = \tbs \nu_{\ll} (\rho),
\end{equation}
and
\begin{equation}
	\label{ane}
	K''(\ll) = \tbs \nu_{\ll} \left( \rho^2 \right) -  \tbs \nu_{\ll} \left( \rho \right)^2 + 2 \sum_{n=1}^{+\infty} \left[ \tbs \nu_{\ll} \left( \rho \tbf P_{\ll}^n \rho \right) -  \tbs \nu_{\ll} \left( \rho \right)^2 \right]  =:\tt \sigma_{\ll}^2.
\end{equation}
\end{lemma}

\begin{proof}
It is clear that $\ll \mapsto \bf P_{\ll}$ is analytic on $\bb R$ and consequently, by the  
perturbation theory for linear operators 
(see for example \cite{kato_perturbation_1976} or \cite{dunford_linear_1971}) $\ll \to k(\ll)$, $\ll \to v_{\ll}$ and $\ll \mapsto \bs \nu_{\ll}$ are also analytic on $\bb R$. 
In particular we write for any $h \in \bb R$,
\begin{align*}
	\bf P_{\ll+h} &= \bf P_{\ll} + h \bf P_{\ll}' + \frac{h^2}{2} \bf P_{\ll}'' +  o(h^2), \\
	v_{\ll+h} &= v_{\ll} + h v_{\ll}' + \frac{h^2}{2} v_{\ll}'' + o(h^2), \\
	k(\ll+h) &= k(\ll) + hk'(\ll) + \frac{h^2}{2} k''(\ll) +  o(h^2),
\end{align*}
where for any $h \in \bb R$, $o(h^2)$ refers to an operator, a function or a real such that $o(h^2)/h^2 \to 0$ as $h \to 0$. Since $v_{\ll+h}$ is an eigenvector of $\bf P_{\ll+h}$ we have $\bf P_{\ll+h} v_{\ll+h} = k(\ll+h) v_{\ll+h}$ and its development gives
\begin{align}
	\bf P_{\ll} v_{\ll} &= k(\ll) v_{\ll}, \nonumber\\
	\label{vampire}
	\bf P_{\ll} v_{\ll}' + \bf P_{\ll}' v_{\ll} &= k(\ll) v_{\ll}' + k'(\ll) v_{\ll}, \\
	\label{loupgarou}
	\frac{1}{2} \bf P_{\ll} v_{\ll}'' + \bf P_{\ll}' v_{\ll}' + \frac{1}{2} \bf P_{\ll}'' v_{\ll} &= \frac{1}{2} k(\ll) v_{\ll}'' + k'(\ll) v_{\ll}' + \frac{1}{2} k''(\ll) v_{\ll}.
\end{align}
Since $\bs \nu_{\ll}$ is an invariant measure, $\bs \nu_{\ll} \left( \bf P_{\ll}g \right) = k(\ll) \bs \nu_{\ll} (g)$ and \eqref{vampire} implies that
\[
k(\ll) \bs \nu_{\ll} \left( v_{\ll}' \right) + \bs \nu_{\ll} \left( \bf P_{\ll}' v_{\ll} \right) = k(\ll) \bs \nu_{\ll} \left( v_{\ll}' \right) + k'(\ll).
\]
In addition, by \eqref{ocean}, $\bf P_{\ll}' v_{\ll} = \bf P_{\ll} \left( \rho v_{\ll} \right)$. Therefore,
\[
k(\ll) \bs \nu_{\ll} \left( \rho v_{\ll} \right) = k'(\ll),
\]
which, with the definition of $\tbs \nu_{\ll}$ in \eqref{ecorce}, proves \eqref{chevalBP}.

From \eqref{loupgarou} and the fact that $\bs \nu_{\ll} \left( \bf P_{\ll}g \right) = k(\ll) \bs \nu_{\ll} (g)$, we have
\[
\frac{k(\ll)}{2} \bs \nu_{\ll} \left( v_{\ll}'' \right) + k(\ll) \bs \nu_{\ll} \left( \rho v_{\ll}' \right) + \frac{k(\ll)}{2} \bs \nu_{\ll} \left( \rho^2 v_{\ll} \right) = \frac{1}{2} k(\ll) \bs \nu_{\ll} \left( v_{\ll}'' \right) + k'(\ll) \bs \nu_{\ll} \left( v_{\ll}' \right) + \frac{1}{2} k''(\ll).
\]
So,
\[
\frac{k''(\ll)}{k(\ll)} = \bs \nu_{\ll} \left( \rho^2 v_{\ll} \right) + 2 \left[ \bs \nu_{\ll} \left( \rho v_{\ll}' \right) - \frac{k'(\ll)}{k(\ll)}\bs \nu_{\ll} \left( v_{\ll}' \right) \right].
\]
By \eqref{chevalBP}, we obtain that
\begin{equation}
	\label{orient}
	K''(\ll) = \frac{k''(\ll)}{k(\ll)} - \left( \frac{k'(\ll)}{k(\ll)} \right)^2 = \bs \nu_{\ll} \left( \rho^2 v_{\ll} \right) - \bs \nu_{\ll}^2 \left( \rho v_{\ll} \right) + 2 \left[ \bs \nu_{\ll} \left( \rho v_{\ll}' \right) - \bs \nu_{\ll} \left( \rho v_{\ll} \right) \bs \nu_{\ll} \left( v_{\ll}' \right) \right].
\end{equation}
It remains to determine $v_{\ll}'$. By \eqref{vampire}, we have
\[
v_{\ll}' -\frac{\bf P_{\ll} v_{\ll}'}{k(\ll)} = \frac{\bf P_{\ll} \left( \rho v_{\ll} \right)}{k(\ll)} - \frac{k'(\ll)}{k(\ll)} v_{\ll}
\]
and for any $n \geq 0$, using \eqref{chevalBP},
\begin{equation}
	\label{rivage}
	\frac{\bf P_{\ll}^n v_{\ll}'}{k(\ll)^n} - \frac{\bf P_{\ll}^{n+1} v_{\ll}'}{k(\ll)^{n+1}} = \frac{\bf P_{\ll}^{n+1} \left( \rho v_{\ll} \right)}{k(\ll)^{n+1}} - \bs \nu_{\ll} \left( \rho v_{\ll} \right) v_{\ll}.
\end{equation}
Note that
\begin{align*}
	\frac{\bf P_{\ll}^{n+1} \left( \rho v_{\ll} \right)}{k(\ll)^{n+1}} - \bs \nu_{\ll} \left( \rho v_{\ll} \right) v_{\ll} = \frac{Q_{\ll}^{n+1}\left( \rho v_{\ll} \right)}{k(\ll)^{n+1}}.
\end{align*}
By \eqref{torrent},
\[
\norm{\frac{\bf P_{\ll}^{n+1} \left( \rho v_{\ll} \right)}{k(\ll)^{n+1}} - \bs \nu_{\ll} \left( \rho v_{\ll} \right) v_{\ll}}_{\infty} \leq c_{\ll}\e^{-c_{\ll}(n+1)} \norm{\rho v_{\ll}}_{\infty} = c_{\ll}\e^{-c_{\ll}(n+1)}.
\]
Consequently, by \eqref{rivage}, the series $\sum_{n\geq 0} \left[ \frac{\bf P_{\ll}^n v_{\ll}'}{k(\ll)^n} - \frac{\bf P_{\ll}^{n+1} v_{\ll}'}{k(\ll)^{n+1}} \right]$ converges absolutely and we deduce that
\[
v_{\ll}' = \sum_{n=0}^{+\infty} \left[ \frac{\bf P_{\ll}^{n+1} \left( \rho v_{\ll} \right)}{k(\ll)^{n+1}} - \bs \nu_{\ll} \left( \rho v_{\ll} \right) v_{\ll} \right].
\]
In particular,
\[
\bs \nu_{\ll} \left( v_{\ll}' \right) = \sum_{n=0}^{+\infty} \left[ \bs \nu_{\ll} \left( \rho v_{\ll} \right) - \bs \nu_{\ll} \left( \rho v_{\ll} \right) \right] = 0,
\]
and
\[
\bs \nu_{\ll} \left( \rho v_{\ll}' \right) = \sum_{n=0}^{+\infty} \left[ \frac{\bs \nu_{\ll} \left( \rho \bf P_{\ll}^{n+1} \left( \rho v_{\ll} \right) \right)}{k(\ll)^{n+1}} - \bs \nu_{\ll} \left( \rho v_{\ll} \right)^2 \right].
\]
Therefore \eqref{orient} becomes
\[
K''(\ll) = \bs \nu_{\ll} \left( \rho^2 v_{\ll} \right) - \bs \nu_{\ll}^2 \left( \rho v_{\ll} \right) + 2 \sum_{n=0}^{+\infty} \left[ \frac{\bs \nu_{\ll} \left( \rho \bf P_{\ll}^{n+1} \left( \rho v_{\ll} \right) \right)}{k(\ll)^{n+1}} - \bs \nu_{\ll} \left( \rho v_{\ll} \right)^2 \right].
\]
To conclude the proof of the lemma, we establish that $K''(\ll) > 0$, 
from which the strict convexity of $K$ follows.
By \eqref{ecorce},
\begin{equation}
	\label{taniere}
	K''(\ll) =  \tbs \nu_{\ll} \left( \tt \rho_{\ll}^2 \right) + 2 \sum_{n=1}^{+\infty} \left[ \tbs \nu_{\ll} \left( \tt \rho_{\ll} \tbf P_{\ll}^n \tt \rho_{\ll} \right) \right],
\end{equation}
where for any $\ll \in \bb R$, $\tt \rho_{\ll} = \rho - \tbs \nu_{\ll}(\rho) v_0$.
Moreover, Conditions \ref{primitif} and \ref{cathedrale} and Lemma \ref{jument} imply that the normalized transfer operator $\tbf P_{\ll}$ together with the function $\tt \rho_{\ll}$ satisfies Conditions \ref{primitif} and \ref{cathedrale}. In conjunction with \eqref{taniere} and Lemma 10.3 of \cite{GLLP_CLLT_2017}, this proves that \eqref{taniere} and so \eqref{ane} are positive.
\end{proof}

\section{Proofs in the critical case} \label{critcase}

In this section we prove Theorem \ref{prince}. By equations \eqref{rose} and \eqref{ciel}, the survival probability of the branching process is related to the study of the sum $q_{n}^{-1}=\e^{-S_n} + \sum_{k=0}^{n-1} \e^{-S_k}\eta_{k+1,n}$ where $(S_n)_{n\geq 0}$ is a Markov walk defined by \eqref{petale}. 
Very roughly speaking, the sum $q_n^{-1}$ converges mainly when the walk stays positive: $S_k >0$ for any $k \geq 1$ and we will see that (at least in the critical case) only positive trajectories of the Markov walk $(S_n)_{n\geq 0}$ count for the survival of the branching process.

Recall that the hypotheses of Theorem \ref{prince} are Conditions \ref{primitif}-\ref{cathedrale} and $k'(0)=\bs \nu(\rho)=0$. Under these assumptions the conclusions of all the theorems of 
Section \ref{flamme} hold for the probability $\bb P_i$, for any $i \in \bb X$. 
Recall also that $\bb E_{i,y}^+$ is the expectation corresponding to the probability measure \eqref{soif}.
We carry out the proof through a series of lemmata. 

\begin{lemma}
Assume conditions of Theorem \ref{prince}.
\label{nuage}
For any $m\geq 1$, $(i,y) \in \supp(V)$, and $j \in \bb X$, we have
\[
\lim_{n\to +\infty} \bb P_i \left( \sachant{Z_m > 0 \,;\, X_n = j}{ \tau_y > n } \right) = \bb E_{i,y}^+ \left( q_m \right) \bs \nu (j).
\]
\end{lemma}

\begin{proof}
Fix $m \geq 1$, $(i,y) \in \supp(V)$, and $j \in \bb X$. By \eqref{ange}, for any $n \geq m+1$,
\begin{align*}
\bb P_i \left( Z_m > 0 \,,\, X_n = j \,,\, \tau_y > n \right) &= \bb E_i \left( \bb P_i \left( \sachant{Z_m > 0}{ X_1, \dots, X_n } \right) \,;\, X_n = j \,,\, \tau_y > n \right) \\
&= \bb E_i \left( \bb E_i \left( \sachant{q_m}{ X_1, \dots, X_n } \right) \,;\, X_n = j \,,\, \tau_y > n  \right) \\
&= \bb E_i \left( q_m \,;\, X_n = j \,,\, \tau_y > n \right).
\end{align*}
Using Lemma \ref{cumulus}, we conclude that
\[
\lim_{n\to +\infty} \bb P_i \left( \sachant{Z_m > 0 \,;\, X_n = j}{ \tau_y > n } \right) = \lim_{n\to +\infty} \bb E_i \left( \sachant{q_m \,;\, X_n = j }{ \tau_y > n } \right) = \bb E_{i,y}^+ \left( q_m \right) \bs \nu(j).
\]
\end{proof}

By Lemma \ref{pieuvre}, we have for any $(i,y) \in \supp(V)$, $k \geq 1$ and $n \geq k+1$,
\begin{equation}
	\label{aquarium}
	0 \leq \eta_{k,n} \leq \eta := \max_{x \in \bb X} \frac{f_x''(1)}{f_x'(1)^2} < +\infty \qquad \bb P_{i,y}^+\text{-a.s.}
\end{equation}
By \eqref{yeux} and \eqref{yeux002}, this equation holds also when $n=k$. Moreover, by Lemma \ref{pieuvre},
\begin{equation}
	\label{poissonBP}
	\eta_{k,\infty} := \lim_{n\to+\infty} \eta_{k,n}  \in [0,\eta] \qquad \bb P_{i,y}^+\text{-a.s.}
\end{equation}

Let $q_{\infty}$ be the following random variable:
\begin{equation}
	\label{voyage}
	q_{\infty} := \left[ \sum_{k=0}^{+\infty} \e^{-S_k} \eta_{k+1,\infty} \right]^{-1} \in [0,+\infty].
\end{equation} 
The random variable $q_{\infty}^{-1}$ is $\bb P_{i,y}^+$-integrable for any $(i,y) \in \supp (V)$: indeed by \eqref{poissonBP},
\[
q_{\infty}^{-1} \leq \sum_{k=0}^{+\infty} \e^{-S_k} \eta.
\]
Using Lemma \ref{soir}, for any $(i,y) \in \supp (V)$
\begin{equation}
	\label{temple}
	\bb E_{i,y}^+ \left( q_{\infty}^{-1} \right) \leq \eta \bb E_{i,y}^+ \left( \sum_{k=0}^{+\infty} \e^{-S_k} \right) \leq \eta \frac{c \left( 1+\max(y,0) \right)\e^{y}}{V(i,y)} < +\infty.
\end{equation}

\begin{lemma}
Assume conditions of Theorem \ref{prince}.
For any $(i,y) \in \supp (V)$,
\begin{equation}
	\label{telescope001}
	\lim_{m\to+\infty} \bb E_{i,y}^+ \left( \abs{q_m^{-1} - q_{\infty}^{-1}} \right) = 0,
\end{equation}
and
\begin{equation}
	\label{telescope002}
	\lim_{m\to+\infty} \bb E_{i,y}^+ \left( \abs{q_m - q_{\infty}} \right) = 0.
\end{equation}
\end{lemma}

\begin{proof}
Let $(i,y) \in \supp (V)$ and fix $l \geq 1$. By \eqref{ciel} and \eqref{voyage}, we have for all $m \geq l+2$,
\begin{align*}
	\bb E_{i,y}^+ \left( \abs{q_m^{-1} - q_{\infty}^{-1}} \right) &= \bb E_{i,y}^+ \left( \abs{\e^{-S_m} + \sum_{k=0}^{m-1} \e^{-S_k} \eta_{k+1,m} - \sum_{k=0}^{+\infty} \e^{-S_k} \eta_{k+1,\infty}}  \right) \\
	&\leq \bb E_{i,y}^+ \left( \e^{-S_m} \right) + \bb E_{i,y}^+ \left( \sum_{k=0}^{l} \e^{-S_k} \abs{\eta_{k+1,m} - \eta_{k+1,\infty}} \right) \\
	&\qquad + \bb E_{i,y}^+ \left( \sum_{k=l+1}^{m-1} \e^{-S_k} \abs{\eta_{k+1,m} - \eta_{k+1,\infty}} \right) + \bb E_{i,y}^+ \left( \sum_{k=m}^{+\infty} \e^{-S_k} \eta_{k+1,\infty} \right).
\end{align*}
By \eqref{aquarium} and \eqref{poissonBP},
\[
\bb E_{i,y}^+ \left( \abs{q_m^{-1} - q_{\infty}^{-1}} \right) \leq \bb E_{i,y}^+ \left( \e^{-S_m} \right) + \bb E_{i,y}^+ \left( \sum_{k=0}^{l} \e^{-S_k} \abs{\eta_{k+1,m} - \eta_{k+1,\infty}} \right) + \eta \bb E_{i,y}^+ \left( \sum_{k=l+1}^{+\infty} \e^{-S_k} \right).
\]
Using Lemma \ref{soir} and the Lebesgue monotone convergence theorem,
\begin{align*}
	\bb E_{i,y}^+ \left( \abs{q_m^{-1} - q_{\infty}^{-1}} \right) &\leq \frac{c \left( 1+\max(y,0) \right)\e^{y}}{V(i,y)} \left( \frac{1}{m^{3/2}} + \eta \sum_{k=l+1}^{+\infty} \frac{1}{k^{3/2}} \right) \\
	&\qquad + \bb E_{i,y}^+ \left( \sum_{k=0}^{l} \e^{-S_k} \abs{\eta_{k+1,m} - \eta_{k+1,\infty}} \right) \\
	&\leq \frac{c \left( 1+\max(y,0) \right)\e^{y}}{V(i,y)} \left( \frac{1}{m^{3/2}} + \frac{\eta}{\sqrt{l}} \right) + \bb E_{i,y}^+ \left( \sum_{k=0}^{l} \e^{-S_k} \abs{\eta_{k+1,m} - \eta_{k+1,\infty}} \right).
\end{align*}
Moreover, by \eqref{aquarium} and \eqref{poissonBP}, we have $\sum_{k=0}^{l} \e^{-S_k} \abs{\eta_{k+1,m} - \eta_{k+1,\infty}} \leq \eta \sum_{k=0}^{+\infty} \e^{-S_k}$ which is $\bb P_{i,y}^+$-integrable by Lemma \ref{soir}. Consequently, using the Lebesgue dominated convergence theorem and \eqref{poissonBP}, when $m \to+\infty$, we obtain that for any $l\geq 1$,
\[
\limsup_{m\to+\infty} \bb E_{i,y}^+ \left( \abs{q_m^{-1} - q_{\infty}^{-1}} \right) \leq \frac{c \eta \left( 1+\max(y,0) \right)\e^{y}}{V(i,y) \sqrt{l}}.
\]
Letting $l\to +\infty$ it proves \eqref{telescope001}.

Now, it follows easily from \eqref{histoire} that $q_{\infty} \leq 1$: for any $\ee > 0$ and $m \geq 1$, we write that $\bb P_{i,y}^+ \left( q_{\infty}^{-1} < 1-\ee \right) \leq \bb P_{i,y}^+ \left( q_{\infty}^{-1} - q_m^{-1} < -\ee \right)$. Since by \eqref{telescope001}, $q_{m}^{-1}$ converges in $\bb P_{i,y}^+$-probability to $q_{\infty}^{-1}$, it follows that for any $\ee > 0$, $\bb P_{i,y}^+ \left( q_{\infty}^{-1} < 1-\ee \right) = 0$ and so 
\begin{equation}
\label{echo}
q_{\infty} \leq 1 \qquad \bb P_{i,y}^+\text{-a.s.}
\end{equation}
Consequently, $\abs{q_m - q_{\infty}} = q_mq_{\infty}\abs{q_m^{-1} - q_{\infty}^{-1}} \leq \abs{q_m^{-1} - q_{\infty}^{-1}}$ and by \eqref{telescope001}, it proves \eqref{telescope002}.
\end{proof}

Let $U$ be a function defined on $\supp (V)$ by
\[
U(i,y) = \bb E_{i,y}^+ \left( q_{\infty} \right).
\]
Note that for any $(i,y) \in \supp(V)$, by \eqref{temple}, $q_{\infty} > 0$ $\bb P_{i,y}^+$-a.s.\ and so 
\begin{equation}
	\label{princesse}
	U(i,y) > 0.
\end{equation}
By \eqref{echo}, we have also $U(i,y) \leq 1$.

\begin{lemma}
Assume conditions of Theorem \ref{prince}.
\label{promesse}
For any $(i,y) \in \supp(V)$ and $j \in \bb X$, we have
\[
\lim_{m\to +\infty} \lim_{n\to +\infty} \bb P_i \left( \sachant{Z_m > 0 \,;\, X_n = j}{ \tau_y > n } \right) = \bs \nu(j) U(i,y).
\]
\end{lemma}

\begin{proof}
By Lemma \ref{nuage}, for any $(i,y) \in \supp(V)$, $j \in \bb X$ and $m \geq 1$, we have
\[
\lim_{n\to +\infty} \bb P_i \left( \sachant{Z_m > 0 \,;\, X_n = j}{ \tau_y > n } \right) = \bs \nu(j) \bb E_{i,y}^+ \left( q_m \right).
\]
By \eqref{telescope002}, we obtain the desired equality.
\end{proof}

\begin{lemma}
Assume conditions of Theorem \ref{prince}.
\label{prologue}
For any $(i,y) \in \supp (V)$ and $\theta \in (0,1)$,
\[
\lim_{m\to+\infty} \limsup_{n\to+\infty} \bb P_i \left( \sachant{Z_m > 0 \,,\, Z_{\pent{\theta n}} =0}{ \tau_y > n } \right) = 0.
\]
\end{lemma}

\begin{proof}
Fix $(i,y) \in \supp (V)$ and $\theta \in (0,1)$. For any $m \geq 1$ and any $n \geq 1$ such that $\pent{\theta n} \geq m+1$ we define $\theta_n = \pent{\theta n}$ and we write
\begin{align*}
	I_0 &:= \bb P_i \left( Z_m > 0 \,,\, Z_{\theta_n} =0 \,,\, \tau_y > n \right) \\
	&= \bb P_i \left( Z_m > 0 \,,\, \tau_y > n \right) - \bb P_i \left( Z_{\theta_n} > 0 \,,\, \tau_y > n \right) \\
	&= \bb E_i \left( \bb P_i \left( \sachant{Z_m > 0}{X_1, \dots, X_m} \right) \,;\, \tau_y > n \right) - \bb E_i \left( \bb P_i \left( \sachant{Z_{\theta_n} > 0}{X_1, \dots, X_{\theta_n}} \right) \,;\, \tau_y > n \right).
\end{align*}
By \eqref{ange},
\[
I_0 = \bb E_i \left( \abs{q_m - q_{\theta_n}} \,;\, \tau_y > n \right).
\]
We define $J_p(i,y) := \bb P_i \left( \tau_y > p \right)$ for any $(i,y) \in \bb X \times \bb R$ and $p \geq 0$ and consider
\[
I_1  := \bb P_i \left( \sachant{ Z_m > 0 \,,\, Z_{\theta_n} =0 }{ \tau_y > n } \right)
\]
for any $(i,y) \in \supp (V)$. By the Markov property, for any $(i,y) \in \supp (V)$,
\[
I_1 = \frac{I_0}{J_n(i,y)} = \bb E_i \left( \abs{ q_m - q_{\theta_n} } \frac{J_{n-\theta_n}\left( X_{\theta_n}, y+S_{\theta_n} \right)}{J_{n}(i,y)} \,;\, \tau_y > {\theta_n} \right).
\]
By the point \ref{oreiller002} of Proposition \ref{oreiller},
\[
I_1 \leq \frac{c}{\sqrt{(1-\theta)n} J_n(i,y)} \bb E_i \left( \abs{ q_m - q_{\theta_n} } \left( 1+ y+S_{\theta_n} \right) \,;\, \tau_y > {\theta_n} \right).
\]
Using also the point \ref{sable003} of Proposition \ref{sable}, we have
\[
I_1 \leq \frac{c}{\sqrt{(1-\theta)n} J_n(i,y)} \bb E_i \left( \abs{ q_m - q_{\theta_n} } \left( 1+ V\left( X_{\theta_n}, y+S_{\theta_n} \right) \right) \,;\, \tau_y > \theta_n \right).
\]
Using \eqref{histoire} and \eqref{soif}, we obtain that
\[
I_1 \leq \frac{c}{\sqrt{(1-\theta)n} J_n(i,y)} \left( \bb P_i \left( \tau_y > \theta_n \right) + V(i,y) \bb E_{i,y}^+ \left( \abs{ q_m - q_{\theta_n} } \right) \right).
\]
Using the point \ref{oreiller001} of Proposition \ref{oreiller}, for any $(i,y) \in \supp (V)$,
\[
\frac{1}{\sqrt{(1-\theta)n} J_n(i,y)} = \frac{1}{\sqrt{(1-\theta)n} \bb P_i \left( \tau_y > n \right)} \underset{n \to +\infty}{\sim} \frac{\sqrt{2\pi} \sigma}{2\sqrt{1-\theta}V(i,y)}.
\]
Moreover using again the point \ref{oreiller001} of Proposition \ref{oreiller} and using \eqref{telescope002},
\[
\bb P_i \left( \tau_y > \theta_n \right) + V(i,y) \bb E_{i,y}^+ \left( \abs{q_m - q_{\theta_n}} \right) \underset{n \to +\infty}{\longrightarrow} V(i,y) \bb E_{i,y}^+ \left( \abs{q_m - q_{\infty}} \right).
\]
Therefore, we obtain that, for any $m \geq 1$ and $\theta \in (0,1)$,
\[
\limsup_{n\to+\infty} I_1 \leq \frac{c}{\sqrt{1-\theta}} \bb E_{i,y}^+ \left( \abs{q_m - q_{\infty}} \right).
\]
Letting $m$ go to $+\infty$ and using \eqref{telescope002}, we conclude that
\[
\lim_{m\to+\infty} \limsup_{n\to+\infty} I_1 = \lim_{m\to+\infty} \limsup_{n\to+\infty} \bb P_i \left( \sachant{ Z_m > 0 \,,\, Z_{\theta_n} =0 }{ \tau_y > n } \right) =0.
\]
\end{proof}

\begin{lemma}
Assume conditions of Theorem \ref{prince}.
\label{rire}
For any $(i,y) \in \supp (V)$, $j \in \bb X$, and $\theta \in (0,1)$,
\[
\lim_{n\to+\infty} \bb P_i \left( \sachant{Z_{\pent{\theta n}} > 0 \,,\, X_n = j}{ \tau_y > n } \right) = \bs \nu(j) U(i,y).
\]
In particular,
\begin{equation}
	\label{chochutement}
	\lim_{n\to+\infty} \bb P_i \left( \sachant{Z_{\pent{\theta n}} > 0 }{ \tau_y > n } \right) = U(i,y).
\end{equation}
\end{lemma}

\begin{proof}
Fix $(i,y) \in \supp(V)$ and $j \in \bb X$. Let $\theta_n := \pent{\theta n}$ for any $\theta \in (0,1)$ and $n \geq 1$. For any $m \geq 1$ and $n \geq 1$ such that $\theta_n \geq m+1$, we write
\begin{align*}
	&\bb P_i \left( \sachant{Z_{\theta_n} > 0 \,,\, X_n = j}{ \tau_y > n } \right) \\
	&\qquad = \bb P_i \left( \sachant{Z_m > 0 \,,\, Z_{\theta_n} > 0 \,,\, X_n = j}{ \tau_y > n } \right) \\
	&\qquad = \bb P_i \left( \sachant{Z_m > 0 \,,\, X_n = j}{ \tau_y > n } \right) - \bb P_i \left( \sachant{Z_m > 0 \,,\, Z_{\theta_n} = 0 \,,\, X_n = j}{ \tau_y > n } \right).
\end{align*}
By Lemma \ref{prologue}, 
\begin{align*}
	&\lim_{m\to+\infty} \limsup_{n\to+\infty} \bb P_i \left( \sachant{Z_m > 0 \,,\, Z_{\theta_n} = 0 \,,\, X_n = j}{ \tau_y > n } \right) \\
	&\hspace{3cm} \leq \lim_{m\to+\infty} \limsup_{n\to+\infty} \bb P_i \left( \sachant{Z_m > 0 \,,\, Z_{\theta_n} = 0}{ \tau_y > n } \right) = 0.
\end{align*}
Therefore, using Lemma \ref{promesse}, it follows that
\[
\lim_{n\to+\infty} \bb P_i \left( \sachant{Z_{\theta_n} > 0 \,,\, X_n = j}{ \tau_y > n } \right) = \bs \nu(j) U(i,y).
\]
\end{proof}

\begin{lemma}
Assume conditions of Theorem \ref{prince}.
\label{main}
For any $(i,y) \in \supp (V)$,
\[
\lim_{p\to+\infty} \bb P_i \left( \sachant{ Z_p > 0 }{ \tau_y > p } \right) = U(i,y).
\]
\end{lemma}

\begin{proof}
Fix $(i,y) \in \supp (V)$. For any $p \geq 1$ and $\theta \in (0,1)$, we have
\[
\bb P_i \left( \sachant{ Z_p > 0 }{ \tau_y > p } \right)  = \frac{\bb P_i \left( Z_p > 0 \,,\, \tau_y > \frac{p}{\theta}+1 \right) + \bb P_i \left( Z_p > 0 \,,\, p < \tau_y \leq \frac{p}{\theta}+1 \right) }{\bb P_i \left( \tau_y > p \right)}.
\]
Let $n = \pent{\frac{p}{\theta}}+1$ and note that $\pent{\theta n} = p$. So, by \eqref{chochutement},
\[
\lim_{p\to+\infty} \bb P_i \left( \sachant{ Z_p > 0 }{ \tau_y > p } \right)  = U(i,y) \lim_{p\to+\infty} \frac{\bb P_i \left( \tau_y > n \right)}{\bb P_i \left( \tau_y > p \right)} + \lim_{p\to+\infty} \frac{\bb P_i \left( Z_p > 0 \,,\, p < \tau_y \leq n \right) }{\bb P_i \left( \tau_y > p \right)}.
\]
By the point \ref{oreiller001} of Proposition \ref{oreiller}, we obtain that
\[
\lim_{p\to+\infty} \bb P_i \left( \sachant{ Z_p > 0 }{ \tau_y > p } \right)  = U(i,y) \sqrt{\theta} + \lim_{p\to+\infty} \frac{\bb P_i \left( Z_p > 0 \,,\, p < \tau_y \leq n \right) }{\bb P_i \left( \tau_y > p \right)}.
\]
Moreover, using again the point \ref{oreiller001} of Proposition \ref{oreiller}, for any $\theta \in (0,1)$,
\[
\frac{\bb P_i \left( Z_p > 0 \,,\, p < \tau_y \leq n \right) }{\bb P_i \left( \tau_y > p \right)} \leq \frac{\bb P_i \left( \tau_y > p \right) - \bb P_i \left( \tau_y > n \right) }{\bb P_i \left( \tau_y > p \right)} \underset{p\to +\infty}{\longrightarrow} 1-\sqrt{\theta}.
\]
Letting $\theta \to 1$, we conclude that
\[
\lim_{p\to+\infty} \bb P_i \left( \sachant{ Z_p > 0 }{ \tau_y > p } \right)  = U(i,y).
\]
\end{proof}

\begin{lemma}
Assume conditions of Theorem \ref{prince}.
\label{chapeau}
For any $(i,y) \in \supp (V)$ and $\theta \in (0,1)$,
\[
\lim_{n\to+\infty} \bb P_i \left( \sachant{ Z_{\pent{\theta n}} > 0 \,,\, Z_n = 0 }{ \tau_y > n } \right) = 0.
\]
\end{lemma}

\begin{proof}
For any $(i,y) \in \supp (V)$, $\theta \in (0,1)$ and $n \geq 1$,
\[
\bb P_i \left( \sachant{ Z_{\pent{\theta n}} > 0 \,,\, Z_n = 0 }{ \tau_y > n } \right) = \bb P_i \left( \sachant{ Z_{\pent{\theta n}} > 0 }{ \tau_y > n } \right) - \bb P_i \left( \sachant{ Z_n > 0 }{ \tau_y > n } \right).
\]
From \eqref{chochutement} and Lemma \ref{main}, it follows
\[
\bb P_i \left( \sachant{ Z_{\pent{\theta n}} > 0 \,,\, Z_n = 0 }{ \tau_y > n } \right) \underset{n\to+\infty}{\longrightarrow} U(i,y) - U(i,y) = 0.
\]
\end{proof}

\begin{lemma}
Assume conditions of Theorem \ref{prince}.
\label{trompette}
For any $(i,y) \in \supp (V)$ and $j \in \bb X$,
\[
\lim_{n\to+\infty} \bb P_i \left( \sachant{ Z_n > 0 \,,\, X_n = j }{ \tau_y > n } \right) = \bs \nu(j) U(i,y).
\]
\end{lemma}

\begin{proof}
For any $(i,y) \in \supp (V)$, $j \in \bb X$, $\theta \in (0,1)$ and $n \geq 1$,
\begin{align*}
	\bb P_i \left( \sachant{ Z_n > 0 \,,\, X_n = j }{ \tau_y > n } \right) &= \bb P_i \left( \sachant{ Z_{\pent{\theta n}} > 0 \,,\, X_n = j }{ \tau_y > n } \right) \\
	&\qquad -  \bb P_i \left( \sachant{ Z_{\pent{\theta n}} > 0 \,,\, Z_n = 0 \,,\, X_n = j }{ \tau_y > n } \right)
\end{align*}
Using Lemmas \ref{rire} and \ref{chapeau}, the result follows.
\end{proof}

\textbf{Proof of Theorem \ref{prince}.} Fix $(i,j) \in \bb X^2$. For any $y \in \bb R$, we have
\begin{equation}
	\label{chateau}
	0 \leq \bb P_i \left( Z_n > 0 \,,\, X_n = j \right) - \bb P_i \left( Z_n > 0 \,,\, X_n = j \,,\, \tau_y > n \right) \leq \bb P_i \left( Z_n > 0 \,,\, \tau_y \leq n \right).
\end{equation}
Using \eqref{ange},
\[
\bb P_i \left( Z_n > 0 \,,\, \tau_y \leq n \right) = \bb E_i \left( q_n \,;\, \tau_y \leq n \right).
\]
Moreover, by the definition of $q_n$ in \eqref{jeux}, for any $k \geq 1$,
\[
q_k \leq f_{X_k}'(1) \times \cdots \times f_{X_1}'(1) = \e^{S_k}.
\]
Since $( q_k )_{k\geq 1}$ is non-increasing, we have $q_n  = \min_{1\leq k \leq n} q_k  \leq \e^{\min_{1\leq k \leq n} S_k}$. Therefore
\begin{align}
	\bb P_i \left( Z_n > 0 \,,\, \tau_y \leq n \right) &\leq \bb E_i \left( \e^{\min_{1\leq k \leq n} S_k} \,;\, \tau_y \leq n \right) \nonumber\\
	&= \e^{-y} \sum_{p=0}^{+\infty} \bb E_i \left( \e^{\min_{1\leq k \leq n} \{y+S_k\}} \,;\, -(p+1) < \min_{1\leq k \leq n} \{y+S_k\} \leq -p \,,\, \tau_y \leq n \right) \nonumber\\
	&\leq \e^{-y} \sum_{p=0}^{+\infty} \e^{-p}  \bb P_i \left(  \tau_{y+p+1} > n \right).
	\label{butte}
\end{align}
By the point \ref{oreiller002} of Proposition \ref{oreiller},
\begin{equation}
	\label{squaw}
	\bb P_i \left( Z_n > 0 \,,\, \tau_y \leq n \right) = \frac{c \e^{-y}}{\sqrt{n}} \sum_{p=0}^{+\infty} \e^{-p} \left(  1+p+1+\max(y,0) \right) \leq \frac{c \e^{-y}\left(  1+\max(y,0) \right)}{\sqrt{n}}.
\end{equation}
Note that from the point \ref{sable003} of Proposition \ref{sable}, it is clear that there exits $y_0 = y_0(i) < +\infty$ such that for any $y \geq y_0$, we have $V(i,y) > 0$ i.e.\ $(i,y) \in \supp (V)$ (for more information on $\supp (V)$ see \cite{grama_limit_2016-1}). Using Lemma \ref{trompette} and the point \ref{oreiller001} of Proposition \ref{oreiller}, for any $y \geq y_0$,
\begin{equation}
	\label{cheyenne}
	\sqrt{n} \bb P_i \left( Z_n > 0 \,,\, X_n = j \,,\, \tau_y > n \right) \underset{n\to+\infty}{\longrightarrow} \frac{2\bs \nu(j) U(i,y) V(i,y)}{\sqrt{2\pi}\sigma}.
\end{equation}
Let 
\[
I(i,j) = \liminf_{n \to +\infty} \sqrt{n} \bb P_i \left( Z_n > 0 \,,\, X_n = j \right)
\]
and
\[
J(i,j) = \limsup_{n \to +\infty} \sqrt{n} \bb P_i \left( Z_n > 0 \,,\, X_n = j \right).
\]
Using \eqref{chateau}, \eqref{squaw} and \eqref{cheyenne}, we obtain that, for any $y \geq y_0(i)$,
\begin{align}
	\frac{2\bs \nu(j) U(i,y) V(i,y)}{\sqrt{2\pi}\sigma} &\leq I(i,j) \nonumber\\
	&\leq J(i,j) \leq \frac{2\bs \nu(j) U(i,y) V(i,y)}{\sqrt{2\pi}\sigma} + c \e^{-y}\left(  1+\max(y,0) \right) < +\infty.
	\label{domino}
\end{align}
From \eqref{cheyenne}, it is clear that $y \mapsto \frac{2 U(i,y) V(i,y)}{\sqrt{2\pi}\sigma}$ is non-decreasing and from \eqref{domino} the function is bounded by $I(i,j)/\bs \nu(j) < +\infty$. Therefore
\[
u(i) := \lim_{y\to +\infty} \frac{2 U(i,y) V(i,y)}{\sqrt{2\pi}\sigma}
\]
exists. Moreover by \eqref{princesse}, for any $y \geq y_0(i)$,
\[
u(i) \geq \frac{2 U(i,y) V(i,y)}{\sqrt{2\pi}\sigma} > 0.
\]
Taking the limit as $y \to +\infty$ in \eqref{domino}, we conclude that
\[
\lim_{n \to +\infty} \sqrt{n} \bb P_i \left( Z_n > 0 \,,\, X_n = j \right) = \bs \nu(j) u(i),
\]
which finishes the proof of Theorem \ref{prince}.

\section{Proofs in the strongly subcritical case}
\label{lagon}

Assume the hypotheses of Theorem \ref{couronne} that is  Conditions \ref{primitif}-\ref{cathedrale} and $k'(1)<0$. We fix $\ll = 1$ and define the probability $\tbb P_i$ and the corresponding expectation $\tbb E_i$ by \eqref{chandelle}, such that, for any $n \geq 1$ and any $g$: $\bb X^n \to \bb C$,
\begin{equation}
\label{chandelier}
\tbb E_i \left( g(X_1, \dots, X_n) \right) = \frac{\bb E_i \left( \e^{S_n} g(X_1, \dots, X_n) v_1(X_n) \right)}{k(1)^n v_1(i)}.
\end{equation}
By \eqref{ange}, we have, for any $(i,j) \in \bb X^2$ and $n \geq 1$,
\begin{align*}
\bb P_i \left( Z_{n+1} > 0 \,,\, X_{n+1} = j \right) &= \bb E_i \left( q_{n+1} \,,\, X_{n+1} = j \right) \\
&= \tbb E_i \left( \frac{\e^{-S_{n+1}}}{v_1 \left( X_{n+1} \right)} q_{n+1}  \,;\, X_{n+1} = j \right) k(1)^{n+1} v_1(i) \\
&= \tbb E_i \left( \e^{-S_n} q_n\left( f_j(0) \right)  \,;\, X_{n+1} = j \right) k(1)^{n+1} \frac{v_1(i) \e^{-\rho(j)}}{v_1(j)},
\end{align*}
where $q_n(s)$ is defined for any $s \in [0,1]$ by \eqref{jeux}. From Lemma \ref{foin}, we write
\begin{align}
	\e^{-S_n} q_n\left( f_j(0) \right) &= \left[ \frac{1}{1-f_j(0)} + \sum_{k=0}^{n-1} \e^{S_n -S_k} \eta_{k+1,n}\left( f_j(0) \right) \right]^{-1} \nonumber\\
	&= \left[ \frac{1}{1-f_j(0)} + \sum_{k=1}^{n} \e^{S_n -S_{n-k}} \eta_{n-k+1,n}\left( f_j(0) \right) \right]^{-1}.
	\label{potion001}
\end{align}
As in Section \ref{batailleBP}, we define the dual Markov chain $\left( X_n^* \right)_{n\geq 0}$, where the dual Markov kernel is given, for any $(i,j) \in \bb X^2$, by
\[
\tbf P_1^*(i,j) = \tbf P_1 (j,i) \frac{\tbs \nu_1 (j)}{\tbs \nu_1 (i)} = \bf P(j,i) \frac{\e^{\rho(i)} \bs \nu_1 (j)}{k(1) \bs \nu_1 (i)}.
\]
Let $(S^*_n)_{n\geq 0}$ be the associated Markov walk defined by \eqref{promenade001}
and 
\begin{equation}
\label{chemin}
q_n^*(j) := \left[ \frac{1}{1-f_j(0)} + \sum_{k=1}^{n} \e^{-S_k^*} \eta_k^*(j) \right]^{-1},
\end{equation}
where 
\begin{align}
\label{chemin003}
\eta_k^*(j) := g_{X_k^*} \left( f_{X_{k-1}^*} \circ \cdots \circ f_{X_1^*}\circ f_j (0) \right)  
\qquad \text{and} \qquad \eta_1^*(j) := g_{X_1^*} \left( f_j (0) \right).
\end{align}
Following the proof of Lemma \ref{foin}, we obtain
\begin{equation}
\label{cheminbis}
q_n^*(j) = \e^{S_n^*} \left( 1-f_{X_n^*} \circ \cdots \circ f_{X_1^*} \circ f_j (0) \right).
\end{equation}
We are going to apply duality Lemma \ref{dualityBP}. 
The following 
correspondences designed by the two-sided arrow $\longleftrightarrow$ are included for the ease of the reader:
\begin{align*}
X_k^* &\longleftrightarrow X_{n-k+1}, \\
S_k^* &  \longleftrightarrow S_{n-k}-S_n, \\
\eta_k^*(j) & \longleftrightarrow \eta_{n-k+1,n}\left( f_j (0) \right),\\
q_n^*(j) & \longleftrightarrow \e^{-S_n}q_{n}\left( f_j (0) \right).
\end{align*}
Now Lemma \ref{dualityBP} implies,
\begin{equation}
\label{coleoptere}
\bb P_i \left( Z_{n+1} > 0 \,,\, X_{n+1} = j \right) = \tbb E_j^* \left( q_n^*(j) \,;\, X_{n+1}^* = i \right) k(1)^{n+1} \frac{\tbs \nu_1(j) v_1(i) \e^{-\rho(j)}}{\tbs \nu_1(i) v_1(j)},
\end{equation}
where $\tbb E_j^*$ is the expectation generated by the trajectories of the chain $\left( X_n^* \right)_{n\geq 0}$ starting at $X_0^* = j$. 

Note that, under Condition \ref{eglise},  by Lemma \ref{pieuvre} we have, for any $j \in \bb X$ and $k \geq 1$,
\begin{equation}
0 \leq \eta_k^*(j) \leq \eta = \max_{i \in \bb X} \frac{f_i''(1)}{f_i'(1)^2} < +\infty \qquad \tbb P_j^*\text{-a.s.}
\label{ehtabound001}
\end{equation}
 In particular, by \eqref{chemin}, 
\[
q_n^*(j) \in (0,1], \quad \forall n \geq 1.
\]
For any $j \in \bb X$, consider the random variable
\begin{equation}
\label{potion002}
q_{\infty}^*(j) := \left[ \frac{1}{1-f_j(0)} + \sum_{k=1}^{\infty} \e^{-S_k^*} \eta_k^*(j) \right]^{-1} \in [0,1].
\end{equation}

\begin{lemma} Assume that the conditions of Theorem \ref{couronne} are satisfied.
For any $j \in \bb X$,
\begin{equation}
	\label{dino001}
	\lim_{n\to+\infty} q_n^*(j) = q_{\infty}^*(j) \in (0,1], \qquad \tbb P_j^*\text{-a.s.}
\end{equation}
and
\begin{equation}
	\label{dino002}
	\lim_{n\to+\infty} \tbb E_j^* \left( \abs{q_n^*(j) - q_{\infty}^*(j)} \right) = 0.
\end{equation}
\end{lemma}

\begin{proof}
Fix $j \in \bb X$. By the law of large numbers for finite Markov chains,
\[
\frac{S_k^*}{k} \underset{k \to +\infty}{\longrightarrow} \tbs \nu_1(-\rho), \qquad \tbb P_j^*\text{-a.s.}
\]
This means that there exists a set $N$ of null probability $\tbb P_j^*(N) = 0$, such that for any $\omega \in \Omega \setminus N$ and any $\ee > 0$, there exists $k_0(\omega,\ee)$ such that for any $k \geq k_0(\omega,\ee)$,
\[
\e^{-S_k^*(\omega)} \eta_k^*(j)(\omega) \leq \e^{k\tbs \nu_1(\rho)+k\ee} \eta,
\]
where for the last inequality we used the bound \eqref{ehtabound001}.
By Lemma \ref{mulet}, we have $\tbs \nu_1(\rho) = k'(1)/k(1) < 0$. Taking $\ee = -\tbs \nu_1(\rho)/2$ we obtain that, for any $k \geq k_0(\omega),$
\[
0 \leq \e^{-S_k^*(\omega)} \eta_k^*(j)(\omega) \leq \e^{k\frac{\tbs \nu_1(\rho)}{2}} \eta.
\]
Consequently, the series $\left( q_n^*(j) \right)^{-1}$ converges a.s.\ to $\left( q_{\infty}^*(j) \right)^{-1} \in [1,+\infty)$ which proves \eqref{dino001}.

Now the sequence $( q_n^*(j) )_{n\geq 1}$ belongs to $[0,1)$ a.s.\ and so by the Lebesgue dominated convergence theorem,
\[
\lim_{n\to+\infty} \tbb E_j^* \left( \abs{q_n^*(j) - q_{\infty}^*(j)} \right) = 0.
\]
\end{proof}

\begin{lemma} Assume that the conditions of Theorem \ref{couronne} are satisfied.
\label{marsupilami}
For any $(i,j) \in \bb X^2$,
\[
\lim_{n\to +\infty} \tbb E_j^* \left( q_n^*(j) \,;\, X_{n+1}^* = i \right) = \tbs \nu_1(i) \tbb E_j^* \left( q_{\infty}^*(j) \right).
\]
\end{lemma}

\begin{proof}
Let $m \geq 1$. For any $(i,j) \in \bb X^2$, and $n \geq m$,
\begin{equation}
\label{dessert001}
\tbb E_j^* \left( q_n^*(j) \,;\, X_{n+1}^* = i \right) = \tbb E_j^* \left( q_m^*(j) \,;\, X_{n+1}^* = i \right) + \tbb E_j^* \left( q_n^*(j) - q_m^*(j) \,;\, X_{n+1}^* = i \right).
\end{equation}
By the Markov property,
\[
\tbb E_j^* \left( q_m^*(j) \,;\, X_{n+1}^* = i \right) = \tbb E_j^* \left( q_m^*(j) \left(\tbf P_1^*\right)^{n-m+1} \left( X_m^*, i \right) \right).
\]
Using \eqref{vautour} (which holds also for $\tbf P_1^*$ by Lemmas \ref{jument} and \ref{sourire})
and \eqref{dino002}, we have
\begin{equation}
\label{dessert002}
\lim_{m\to +\infty} \lim_{n\to +\infty} \tbb E_j^* \left( q_m^*(j) \,;\, X_{n+1}^* = i \right) = \lim_{m\to +\infty}\tbb E_j^* \left( q_m^*(j) \right) \tbs \nu_1(i) = \tbb E_j^* \left( q_{\infty}^*(j) \right) \tbs \nu_1(i).
\end{equation}
Moreover, again by \eqref{dino002},
\begin{align*}
\lim_{m\to +\infty}\lim_{n\to +\infty} \abs{\tbb E_j^* \left( q_n^*(j) - q_m^*(j) \,;\, X_{n+1}^* = i \right)} &\leq \lim_{m\to +\infty}\lim_{n\to +\infty} \tbb E_j^* \left( \abs{q_n^*(j) - q_m^*(j)} \right) \\
&= \lim_{m\to +\infty} \tbb E_j^* \left( \abs{q_{\infty}^*(j) - q_m^*(j)} \right) \\
&= 0.
\end{align*}
Together with \eqref{dessert001} and \eqref{dessert002}, this concludes the lemma.
\end{proof}

\textbf{Proof of Theorem \ref{couronne}.} By \eqref{dino001}, the function 
\[
u(j) = \frac{\tbs \nu_1(j) \e^{-\rho(j)} \tbb E_j^* \left( q_{\infty}^*(j) \right)}{v_1(j)}
\]
is positive.
The result of the theorem follows from Lemma \ref{marsupilami} and the identity \eqref{coleoptere}.

\section{Proofs in the intermediate subcritical case}
\label{intermedcrit}
We assume the conditions of Theorem \ref{sceptre}, that is Conditions \ref{primitif}-\ref{cathedrale} and $k'(1)=0$. As in the critical case the proof is  carried out through a series of lemmata.

The beginning of the reasoning is the same as in the strongly subcritical case. Keeping the same notation as in Section \ref{lagon} (see \eqref{chandelier}-\eqref{coleoptere}), we have
\begin{equation}
\label{coleopterebis}
\bb P_i \left( Z_{n+1} > 0 \,,\, X_{n+1} = j \right) = \tbb E_j^* \left( q_n^*(j) \,;\, X_{n+1}^* = i \right) k(1)^{n+1} \frac{\tbs \nu_1(j) v_1(i) \e^{-\rho(j)}}{\tbs \nu_1(i) v_1(j)}.
\end{equation}

Under the hypotheses of Theorem \ref{sceptre}, the Markov walk $( S_n^* )_{n\geq 0}$ is centred under the probability $\tbb P_j^*$ for any $j \in \bb X$: indeed $\tbs \nu_1 (-\rho) = -k'(1)/k(1) = 0$ (see Lemma \ref{mulet}) and by Lemma \ref{jument}, Conditions \ref{primitif} and \ref{cathedrale} hold for $\tbf P_1$. In this case, by Lemma \ref{sourire}, Conditions \ref{primitif} and \ref{cathedrale} hold also for $\tbf P_1^*$. Therefore all the results of Section \ref{flamme} hold for the probability $\tbb P^*$. Let $\tau_z^*$ be the exit time of the Markov walk $( z+S_n^* )_{n\geq 0}$:
\[
\tau_z^* := \inf \left\{ k \geq 1 : z+S_k^* \leq 0 \right\}.
\]
Denote by $\tt V_1^*$ the harmonic function defined by Proposition \ref{sable} with respect to the probability $\tbb P^*$. 
As in \eqref{soif}, for any $(j,z) \in \supp(\tt V_1^*)$, define a new probability $\tbb P_{j,z}^{*+}$ and its associated expectation $\bb E_{j,z}^{*+}$ on $\sigma\left( X_n^*, n \geq 1 \right)$ by
\[
\tbb E_{j,z}^{*+} \left( g \left( X_1^*, \dots, X_n^* \right) \right) := \frac{1}{\tt V_1^*(j,z)} \tbb E_j^* \left( g\left( X_1^*, \dots, X_n^* \right) \tt V_1^*\left( X_n^*, z+S_n^* \right) \,;\, \tau_z^* > n \right),
\]
for any $n \geq 1$ and any $g$: $\bb X^n \to \bb C$.

\begin{lemma} Assume that the conditions of Theorem \ref{sceptre} are satisfied.
\label{jaguar}
For any $m\geq 1$, $(j,z) \in \supp(\tt V_1^*)$, and $i \in \bb X$, we have
\[
\lim_{n\to +\infty} \tbb E_j^* \left( \sachant{ q_m^*(j) \,;\, X_{n+1}^* = i }{ \tau_z^* > n+1 } \right) = \tbb E_{j,z}^{*+} \left( q_m^*(j) \right) \tbs \nu_1 (i).
\]
\end{lemma}

\begin{proof}
The equation \eqref{cheminbis} gives an explicit formula for $q_m^*(j)$ in terms of $\left( X_1^*, \dots, X_m^* \right)$. 
Therefore, the assertion of the lemma is a straightforward consequence of Lemma \ref{cumulus}.
\end{proof}

As in Section \ref{lagon}, using Lemma \ref{pieuvre} we have for any $(j,z) \in \supp(\tt V_1^*)$ and $k \geq 1$,
\begin{equation}
\label{sablier001}
0 \leq \eta_k^*(j) \leq \eta = \max_{i \in \bb X} \frac{f_i''(1)}{f_i'(1)^2} < +\infty \qquad \text{and} \qquad  q_n^*(j) \in (0,1], \qquad \tbb P_{j,z}^{*+}\text{-a.s.}
\end{equation}
Consider the random variable
\begin{equation}
\label{sablier002}
q_{\infty}^*(j) := \left[ \frac{1}{1-f_j(0)} + \sum_{k=1}^{+\infty} \e^{-S_k^*} \eta_k^*(j) \right]^{-1} \in [0,1].
\end{equation}

\begin{lemma} Assume that the conditions of Theorem \ref{sceptre} are satisfied.
\label{piano}
For any $(j,z) \in \supp (\tt V_1^*)$
\begin{equation}
\label{piano001}
\lim_{m\to+\infty} \tbb E_{j,z}^{*+} \left( \abs{\left(q_m^*(j)\right)^{-1} - \left(q_{\infty}^*(j)\right)^{-1}} \right) = 0,
\end{equation}
and
\begin{equation}
\label{piano002}
\lim_{m\to+\infty} \tbb E_{j,z}^{*+} \left( \abs{q_m^*(j) - q_{\infty}^*(j)} \right) = 0.
\end{equation}
\end{lemma}

\begin{proof}
Fix $(j,z) \in \supp(\tt V_1^*)$. By \eqref{chemin}, \eqref{sablier002} and \eqref{sablier001}, for any $m \geq 1$,
\[
\tbb E_{j,z}^{*+} \left( \abs{\left(q_m^*(j)\right)^{-1} - \left(q_{\infty}^*(j)\right)^{-1}} \right) \leq \eta \tbb E_{j,z}^{*+} \left( \sum_{k=m+1}^{+\infty} \e^{-S_k^*} \right).
\]
 From this bound, by Lemma \ref{soir} and the dominated convergence theorem when $m\to+\infty$, we obtain \eqref{piano001}.

Now by \eqref{sablier001} and \eqref{sablier002} we have for any $m\geq 1$,
\begin{align*}
	\tbb E_{j,z}^{*+} \left( \abs{q_m^*(j) - q_{\infty}^*(j)} \right) &= \tbb E_{j,z}^{*+} \left( \abs{q_m^*(j) q_{\infty}^*(j)} \abs{\left(q_m^*(j)\right)^{-1} - \left(q_{\infty}^*(j)\right)^{-1}} \right) \\
	&\leq \tbb E_{j,z}^{*+} \left( \abs{\left(q_m^*(j)\right)^{-1} - \left(q_{\infty}^*(j)\right)^{-1}} \right),
\end{align*}
which proves \eqref{piano002}.
\end{proof}

Let $U$ be the function defined on $\supp(\tt V_1^*)$ by
\[
U^*(j,z) = \tbb E_{j,z}^{*+} \left( q_{\infty}^*(j) \right).
\]
Using \eqref{sablier001} and Lemma \ref{soir}, we have
\begin{equation}
	\label{panda}
	\tbb E_{j,z}^{*+} \left( \left(q_{\infty}^*(j)\right)^{-1} \right) \leq \frac{1}{1-f_j(0)} +\eta \tbb E_{j,z}^{*+} \left( \sum_{k=1}^{+\infty} \e^{-S_k^*} \right) <+\infty.
\end{equation}
Therefore $q_{\infty}^* > 0$ $\bb P_{i,y}^+$-a.s.\ and so $U^*(j,z) > 0$. In addition, by \eqref{sablier002}, $U^*(j,z) \leq 1$. For any $(j,z) \in \supp(\tt V_1^*)$,
\begin{equation}
	\label{eventail}
	U^*(j,z) \in (0,1].
\end{equation}

\begin{lemma} Assume that the conditions of Theorem \ref{sceptre} are satisfied.
\label{recreation}
For any $(j,z) \in \supp(\tt V_1^*)$ and $i \in \bb X$, we have
\[
\lim_{m\to+\infty} \lim_{n\to +\infty} \tbb E_j^* \left( \sachant{ q_m^*(j) \,;\, X_{n+1}^* = i }{ \tau_z^* > n+1 } \right) = U^*(j,z) \tbs \nu_1 (i).
\]
\end{lemma}
\begin{proof}
The assertion of the lemma is straightforward consequence of  Lemmas \ref{jaguar} and \ref{piano}.
\end{proof}

\begin{lemma} Assume that the conditions of Theorem \ref{sceptre} are satisfied.
\label{castorBP}
For any $(j,z) \in \supp(\tt V_1^*)$ and $\theta \in (0,1)$, we have
\[
\lim_{m\to+\infty} \limsup_{n\to +\infty} \tbb E_j^* \left( \sachant{ \abs{q_m^*(j)-q_{\pent{\theta n}}^*(j)} }{ \tau_z^* > n+1 } \right) = 0.
\]
\end{lemma}

\begin{proof}
Fix $(j,z) \in \supp(\tt V_1^*)$ and $\theta \in (0,1)$. 
Let $m \geq 1$ and $n \geq 1$ be such that $\theta n \geq m+1$.  Set $\theta_n = \pent{\theta n}$. Denote 
\[
I_0 := \tbb E_j^* \left( \sachant{ \abs{q_m^*(j)-q_{\theta_n}^*(j)} }{ \tau_z^* > n+1 } \right) \qquad \text{and} \qquad J_n(j,z) := \tbb P_j^* \left( \tau_z^* > n \right).
\]
Note that by the point \ref{oreiller001} of Proposition \ref{oreiller}, we have $J_n(j,z) > 0$ for any $n$ large enough. 
By the Markov property and  the point \ref{oreiller002} of Proposition \ref{oreiller},
\begin{align*}
I_0 &= \frac{1}{J_{n+1}(j,z)} \tbb E_j^* \left( \abs{q_m^*(j)-q_{\theta_n}^*(j)} J_{n+1-\theta_n} \left( X_{\theta_n}^*,z+S_{\theta_n}^* \right) \,;\, \tau_z^* > \theta_n  \right) \\
&\leq \frac{c}{J_{n+1}(j,z) \sqrt{n+1-\theta_n}} \tbb E_j^* \left( \abs{q_m^*(j)-q_{\theta_n}^*(j)} \left( 1+z+S_{\theta_n}^* \right) \,;\, \tau_z^* > \theta_n  \right).
\end{align*}
Using the point \ref{sable003} of Proposition \ref{sable} and \eqref{sablier001},
\begin{align*}
I_0 &\leq \frac{c}{J_{n+1}(j,z) \sqrt{n(1-\theta)}} \tbb E_j^* \left( \abs{q_m^*(j)-q_{\theta_n}^*(j)} \left( 1+\tt V_1^*\left( X_{\theta_n}^*,z+S_{\theta_n}^* \right) \right) \,;\, \tau_z^* > \theta_n  \right) \\
&\leq \frac{c}{J_{n+1}(j,z) \sqrt{n(1-\theta)}} \left( \tbb P_j^* \left( \tau_z^* > \theta_n \right) + \tt V_1(j,z) \tbb E_{j,z}^{*+} \left( \abs{q_m^*(j)-q_{\theta_n}^*(j)} \right) \right).
\end{align*}
By the point \ref{oreiller001} of Proposition \ref{oreiller} and \eqref{piano002}, we obtain that
\[
\limsup_{n\to+\infty} I_0 \leq \limsup_{n\to+\infty} \frac{c\sqrt{n+1}}{\sqrt{n(1-\theta)}}  \tbb E_{j,z}^{*+} \left( \abs{q_m^*(j)-q_{\theta_n}^*(j)} \right) = \frac{c}{\sqrt{(1-\theta)}}  \tbb E_{j,z}^{*+} \left( \abs{q_m^*(j)-q_{\infty}^*(j)} \right).
\]
Taking the limit as $m \to +\infty$ and using \eqref{piano002}, we conclude that 
\[
\lim_{m\to+\infty} \limsup_{n\to +\infty} \tbb E_j^* \left( \sachant{ \abs{q_m^*(j)-q_{\pent{\theta n}}^*(j)} }{ \tau_z^* > n+1 } \right) = 0.
\]
\end{proof}

\begin{lemma} Assume that the conditions of Theorem \ref{sceptre} are satisfied.
\label{panier}
For any $(j,z) \in \supp(\tt V_1^*)$, $i \in \bb X$ and $\theta \in (0,1)$, we have
\[
\lim_{n\to +\infty} \tbb E_j^* \left( \sachant{ q_{\pent{\theta n}}^*(j) \,;\, X_{n+1}^* = i }{ \tau_z^* > n+1 } \right) = U^*(j,z) \tbs \nu_1 (i).
\]
\end{lemma}

\begin{proof}
For any $(j,z) \in \supp(\tt V_1^*)$, $i \in \bb X$, $\theta \in (0,1)$, $m \geq 1$ and $n \geq m+1$ such that $\pent{\theta n} \geq m$, we have
\begin{align*}
I_0 &:= \tbb E_j^* \left( \sachant{ q_{\pent{\theta n}}^*(j) \,;\, X_{n+1}^* = i }{ \tau_z^* > n+1 } \right) \\
&= \tbb E_j^* \left( \sachant{ q_m^*(j) \,;\, X_{n+1}^* = i }{ \tau_z^* > n+1 } \right) + \underbrace{\tbb E_j^* \left( \sachant{ q_{\pent{\theta n}}^*(j)-q_m^*(j) \,;\, X_{n+1}^* = i }{ \tau_z^* > n+1 } \right)}_{=:I_1}.
\end{align*}
By Lemma \ref{castorBP},
\[
\limsup_{m\to+\infty} \limsup_{n\to +\infty} \abs{I_1} \leq \lim_{m\to+\infty} \limsup_{n\to +\infty} \tbb E_j^* \left( \sachant{ \abs{q_{\pent{\theta n}}^*(j)-q_m^*(j)} }{ \tau_z^* > n+1 } \right) = 0.
\]
Consequently, using Lemma \ref{recreation},
\[
\lim_{n\to +\infty} I_0 = \lim_{m\to+\infty} \lim_{n\to +\infty} \tbb E_j^* \left( \sachant{ q_{m}^*(j) \,;\, X_{n+1}^* = i }{ \tau_z^* > n+1 } \right) = U^*(j,z) \tbs \nu_1 (i).
\]
\end{proof}

\begin{lemma} Assume that the conditions of Theorem \ref{sceptre} are satisfied.
\label{osier}
For any $(j,z) \in \supp(\tt V_1^*)$, we have
\[
\lim_{p\to +\infty} \tbb E_j^* \left( \sachant{ q_{p}^*(j) }{ \tau_z^* > p+1 } \right) = U^*(j,z).
\]
\end{lemma}

\begin{proof}
Fix $(j,z) \in \supp(\tt V_1^*)$. For any $p\geq 1$ and $\theta \in (0,1)$ set $n = \pent{p/\theta}+1$.  
Note that $p=\pent{\theta n}$. We write, for any $p \geq 1$,
\[
\tbb E_j^* \left( \sachant{ q_p^*(j) }{ \tau_z^* > p+1 } \right) = \frac{\tbb E_j^* \left( q_p^*(j) \,;\, \tau_z^* > n+1 \right) + \tbb E_j^* \left( q_p^*(j) \,;\, p+1 < \tau_z^* \leq n+1 \right)}{\tbb P_j^* \left( \tau_z^* > p+1 \right)}.
\]
By Lemma \ref{panier} and the point \ref{oreiller001} of Proposition \ref{oreiller},
\begin{align*}
	\frac{\tbb E_j^* \left( q_p^*(j) \,;\, \tau_z^* > n+1 \right)}{\tbb P_j^* \left( \tau_z^* > p+1 \right)} &= \sum_{i\in \bb X} \tbb E_j^* \left( \sachant{ q_p^*(j) \,;\, X_{n+1}^*=i }{ \tau_z^* > n+1 } \right) \frac{\tbb P_j^* \left( \tau_z^* > n+1 \right)}{\tbb P_j^* \left( \tau_z^* > p+1 \right)} \\
	&\underset{p\to+\infty}{\longrightarrow} U^*(j,z) \sqrt{\theta}.
\end{align*}
Moreover, using \eqref{sablier001} and the point \ref{oreiller001} of Proposition \ref{oreiller},
\[
\frac{\tbb E_j^* \left( q_p^*(j) \,;\, p+1 < \tau_z^* \leq n+1 \right)}{\tbb P_j^* \left( \tau_z^* > p+1 \right)} \leq 1- \frac{\tbb P_j^* \left(\tau_z^* > n+1 \right)}{\tbb P_j^* \left( \tau_z^* > p+1 \right)} \underset{p\to+\infty}{\longrightarrow} 1-\sqrt{\theta}.
\]
Therefore, for any $\theta \in (0,1)$,
\[
\abs{\lim_{p\to+\infty} \tbb E_j^* \left( \sachant{ q_p^*(j) }{ \tau_z^* > p+1 } \right) - U^*(j,z) \sqrt{\theta}} \leq 1-\sqrt{\theta}.
\]
Taking the limit as $\theta \to 1$ it concludes the proof.
\end{proof}

\begin{lemma} Assume that the conditions of Theorem \ref{sceptre} are satisfied.
\label{moulin}
For any $(j,z) \in \supp(\tt V_1^*)$ and $\theta \in (0,1)$, we have
\[
\lim_{n\to +\infty} \tbb E_j^* \left( \sachant{ \abs{ q_{\pent{\theta n}}^*(j)-q_n^*(j) } }{ \tau_z^* > n+1 } \right) = 0.
\]
\end{lemma}

\begin{proof}
Using the fact that $\eta_k^*(j)$ are non-negative and the definition of $q_n^*(j)$ in \eqref{chemin}, we see that $( q_n^*(j) )_{n\geq 1}$ is non-increasing. Therefore, using Lemmas \ref{panier} and \ref{osier},
\begin{align*}
I_0 &:= \lim_{n\to +\infty} \tbb E_j^* \left( \sachant{ \abs{ q_{\pent{\theta n}}^*(j)-q_n^*(j) } }{ \tau_z^* > n+1 } \right) \\
&=  \lim_{n\to +\infty} \sum_{i \in \bb X} \tbb E_j^* \left( \sachant{ q_{\pent{\theta n}}^*(j) \,;\, X_{n+1}^* = i }{ \tau_z^* > n+1 } \right) - \lim_{n\to +\infty} \tbb E_j^* \left( \sachant{ q_n^*(j) }{ \tau_z^* > n+1 } \right) \\
&= U^*(j,z) - U^*(j,z)  = 0.
\end{align*}
\end{proof}

\begin{lemma} Assume that the conditions of Theorem \ref{sceptre} are satisfied.
\label{aulne}
For any $(j,z) \in \supp(\tt V_1^*)$ and $i \in \bb X$, we have
\[
\lim_{n\to +\infty} \tbb E_j^* \left( \sachant{ q_n^*(j) \,;\, X_{n+1}^* = i }{ \tau_z^* > n+1 } \right) = U^*(j,z) \tbs \nu_1(i).
\]
\end{lemma}

\begin{proof}
By Lemmas \ref{panier} and \ref{moulin}, for any $(j,z) \in \supp(\tt V_1^*)$, $i \in \bb X$ and $\theta \in (0,1)$,
\begin{align*}
	I_0 &:= \lim_{n\to +\infty} \tbb E_j^* \left( \sachant{ q_n^*(j) \,;\, X_{n+1}^* = i }{ \tau_z^* > n+1 } \right) \\
	&= \lim_{n\to +\infty} \tbb E_j^* \left( \sachant{ q_{\pent{\theta n}}^*(j) \,;\, X_{n+1}^* = i }{ \tau_z^* > n+1 } \right) \\
	&\qquad+ \lim_{n\to +\infty} \tbb E_j^* \left( \sachant{ q_n^*(j) - q_{\pent{\theta n}}^*(j) \,;\, X_{n+1}^* = i }{ \tau_z^* > n+1 } \right) \\
	&= U^*(j,z) \tbs \nu_1(i).
\end{align*}
\end{proof}

\begin{lemma} Assume that the conditions of Theorem \ref{sceptre} are satisfied.
\label{dieu}
There exists $\tt u$ a positive function on $\bb X$ such that, for any $(i,j) \in \bb X^2$, we have
\[
\tbb E_j^* \left( q_n^*(j) \,;\, X_{n+1}^* = i \right) \underset{n\to+\infty}{\sim} \frac{\tt u(j) \tbs \nu_1(i)}{\sqrt{n}}.
\]
\end{lemma}

\begin{proof}
Fix $(i,j) \in \bb X^2$. For any $z \in \bb R$ and $n \geq 1$,
\begin{equation}
\label{gourmand001}
0 \leq \tbb E_j^* \left( q_n^*(j) \,;\, X_{n+1}^* = i \right) - \tbb E_j^* \left( q_n^*(j) \,;\, X_{n+1}^* = i \,,\, \tau_z^* > n+1 \right) \leq \tbb E_j^* \left( q_n^*(j) \,;\, \tau_z^* \leq n+1 \right).
\end{equation}
Since $q_n^*(j) \leq 1$ (see \eqref{sablier001}), we have
\begin{equation}
\label{bonheur001}
\tbb E_j^* \left( q_n^*(j) \,;\, \tau_z^* \leq n+1 \right) \leq \tbb E_j^* \left( q_n^*(j) \,;\, \tau_z^* \leq n \right) + \tbb P_j \left( \tau_z^* = n+1 \right).
\end{equation}
By \eqref{cheminbis}, $q_n^*(j) \leq \e^{S_n^*}$. Since $( q_n^*(j) )_{n\geq 1}$ is non-increasing, we have $q_n^*(j) = \min_{1\leq k \leq n} q_k^*(j) \leq \e^{\min_{1\leq k \leq n} S_k^*}$. Consequently,
\begin{align*}
\tbb E_j^* \left( q_n^*(j) \,;\, \tau_z^* \leq n \right) &\leq \e^{-z} \tbb E_j^* \left( \e^{\min_{1\leq k \leq n}  z+S_k^*} \,;\, \tau_z^* \leq n \right) \\
&\leq \e^{-z} \sum_{p=0}^{+\infty} \e^{-p}\tbb P_j^* \left( -(p+1) < \min_{1\leq k \leq n} z+S_k^* \leq -p \,,\, \tau_z^* \leq n \right) \\
&\leq \e^{-z} \sum_{p=0}^{+\infty} \e^{-p}\tbb P_j^* \left( \tau_{z+p+1}^* > n \right).
\end{align*}
Using the point \ref{oreiller002} of Proposition \ref{oreiller},
\begin{equation}
\label{bonheur002}
\tbb E_j^* \left( q_n^*(j) \,;\, \tau_z^* \leq n \right) \leq \frac{c \e^{-z} \left( 1+\max(0,z) \right)}{\sqrt{n}}.
\end{equation}
By the point \ref{sable003} of Proposition \ref{sable}, there exists $z_0 \in \bb R$ such that for any $z \geq z_0$, $\tt V_1^*(j,z) > 0$, which means that $(j,z) \in \supp(\tt V_1^*)$. 
Therefore, using the point \ref{oreiller001} of Proposition \ref{oreiller}, for any $z\geq z_0$,
\begin{equation}
\label{bonheur003}
\lim_{n\to+\infty} \sqrt{n}\tbb P_j \left( \tau_z^* = n+1 \right) = \lim_{n\to+\infty} \sqrt{n}\tbb P_j \left( \tau_z^* > n \right) - \lim_{n\to+\infty} \sqrt{n}\tbb P_j \left( \tau_z^* > n+1 \right) = 0.
\end{equation}
Putting together \eqref{bonheur001}, \eqref{bonheur002} and \eqref{bonheur003}, we obtain that, for any $z\geq z_0$,
\begin{equation}
\label{gourmand002}
\lim_{n\to+\infty} \sqrt{n} \tbb E_j^* \left( q_n^*(j) \,;\, \tau_z^* \leq n+1 \right) \leq c \e^{-z} \left( 1+\max(0,z) \right).
\end{equation}
Moreover, using Lemma \ref{aulne} and the point \ref{oreiller001} of Proposition \ref{oreiller},
\begin{equation}
\label{gourmand003}
\lim_{n\to+\infty} \sqrt{n} \tbb E_j^* \left( q_n^*(j) \,;\, X_{n+1}^* = i \,,\, \tau_z^* > n+1 \right) = \frac{2\tt V_1^*(j,z)}{\sqrt{2\pi} \tt \sigma_1} U^*(j,z) \tbs \nu_1(i),
\end{equation}
where $\tt \sigma_1$ is defined in \eqref{ane}.
Denoting
\[
I(i,j) = \liminf_{n\to +\infty} \sqrt{n}\tbb E_j^* \left( q_n^*(j) \,;\, X_{n+1}^* = i \right) \quad \text{and} \quad J(i,j) = \limsup_{n\to +\infty} \sqrt{n} \tbb E_j^* \left( q_n^*(j) \,;\, X_{n+1}^* = i \right),
\]
and using \eqref{gourmand001}, \eqref{gourmand002} and \eqref{gourmand003}, we obtain that, for any $z \geq z_0,$
\begin{align}
\label{armure}
\frac{2\tt V_1^*(j,z)}{\sqrt{2\pi} \tt \sigma_1} U^*(j,z) \tbs \nu_1(i) &\leq I(i,j) \\
&\leq J(i,j) \leq \frac{2\tt V_1^*(j,z)}{\sqrt{2\pi} \tt \sigma_1} U^*(j,z)  \tbs \nu_1(i) + c \e^{-z} \left( 1+\max(0,z) \right). \nonumber
\end{align}
By \eqref{gourmand003}, we observe that $z \mapsto \frac{2\tt V_1^*(j,z)U^*(j,z)}{\sqrt{2\pi} \tt \sigma_1}$ is non-decreasing and by \eqref{armure}, this function is bounded by $I(i,j)/ \tbs \nu_1(i)$. Consequently the limit
\[
\tt u(j) := \lim_{z\to+\infty} \frac{2\tt V_1^*(j,z)U^*(j,z)}{\sqrt{2\pi} \tt \sigma_1} 
\]
exists and for any $z \geq z_0$, by \eqref{eventail},
\begin{equation}
\label{luxe}
\tt u(j) \geq \frac{2\tt V_1^*(j,z)U^*(j,z)}{\sqrt{2\pi} \tt \sigma_1} >0.
\end{equation}
Taking the limit as $z \to +\infty$ in \eqref{armure}, we conclude that
\[
I(i,j) = J(i,j) = \tt u(j) \tbs \nu_1(i).
\]
\end{proof}

\textbf{Proof of Theorem \ref{sceptre}.} By \eqref{luxe} the function
\[
u(j) = \tt u(j) \frac{\tbs \nu_1(j) \e^{-\rho(j)}}{ v_1(j)}, \qquad \forall j \in \bb X,
\]
is positive on $\bb X$. The assertion of Theorem \ref{sceptre} is a consequence of \eqref{coleopterebis} and Lemma \ref{dieu}.

\section{Proofs in the weakly subcritical case}
\label{weaklysubcrit}

We assume the conditions of Theorem \ref{cape}, that is
 Conditions \ref{primitif}-\ref{cathedrale} and 
 $\bs \nu(\rho)=k'(0)<0$, $k'(1)>0$.
By Lemma \ref{mulet}, the function $\ll \mapsto K'(\ll)$ is increasing. Consequently, there exists $\ll \in (0,1)$ such that
\begin{equation}
\label{luth}
K'(\ll) = \frac{k'(\ll)}{k(\ll)} = \tbs \nu_{\ll} (\rho) = 0.
\end{equation}
For this $\ll$ and any $i \in \bb X$, define the changed probability measure $\tbb P_i$ and the corresponding 
expectation $\tbb E_i$ by \eqref{chandelle}, such that for any $n \geq 1$ and any $g$: $\bb X^n \to \bb C$,
\begin{equation}
\label{samovar}
\tbb E_i \left( g(X_1, \dots, X_n) \right) = \frac{\bb E_i \left( \e^{\ll S_n} g(X_1, \dots, X_n) v_{\ll}(X_n) \right)}{k(\ll)^n v_{\ll}(i)}.
\end{equation}

Our starting point is the following formula which is a consequence of \eqref{jeux}: for any $(i,j) \in \bb X^2$ and $n \geq 1$,
\begin{align}
&\bb E_i \left( q_{n+1} \,;\, X_{n+1} = j \,,\, \tau_y > n \right) \nonumber\\
&\hspace{2cm}= \tbb E_i \left( \e^{-\ll S_n} q_n\left( f_j(0) \right) \,;\, X_{n+1} = j \,,\, \tau_y > n \right) k(\ll)^{n+1} \frac{v_{\ll}(i)}{v_{\ll}(j)} \e^{-\ll \rho(j)}.
\label{falaise}
\end{align}
The transition probabilities of $\left( X_n \right)_{n\geq 0}$ under the changed measure are given by \eqref{lacBP}:
\[
\tbf P_{\ll} (i,j) = \frac{\e^{\ll \rho(j)} v_{\ll}(j)}{k(\ll)v_{\ll}(i)} \bf P(i,j).
\]
By \eqref{luth}, the Markov walk $(S_n)_{n\geq 0}$ is centred under $\tbb P_i$. Note that under the hypotheses of Theorem \ref{cape}, by Lemma \ref{jument}, Conditions \ref{primitif} and \ref{cathedrale} hold also for $\tbf P_{\ll}$. Therefore all the results of Section \ref{flamme} hold for the Markov walk $\left( S_n \right)_{n\geq 0}$ under $\tbb P_i$.

Let $\left( X_n^* \right)_{n\geq 0}$ be the dual Markov chain  independent of $\left( X_n \right)_{n\geq 0}$, 
with transition probabilities $\tbf P_{\ll}^*$ defined by (cp.\ \eqref{statueBP})
\begin{equation}
\label{monument}
\tbf P_{\ll}^*(i,j) = \frac{\tbs \nu_{\ll}(j)}{\tbs \nu_{\ll}(i)} \tbf P(j,i) = \frac{\bs \nu_{\ll}(j)}{\bs \nu_{\ll}(i)} \frac{\e^{\rho(i)}}{k(\ll)} \bf P(j,i).
\end{equation}
As in Section \ref{batailleBP}, we define the dual Markov walk $( S_n^* )_{n\geq 0}$  by \eqref{promenade001} and its exit time $\tau_z^*$ for any $z \in \bb R$ by \eqref{promenade002}. Let $\tbb P_{i,j}$ be the probability on $\left( \Omega, \scr F \right)$ generated by the finite dimensional distributions of $( X_n, X_n^* )_{n\geq 0}$ starting at $(X_0,X_0^*) = (i,j)$. By \eqref{luth}, the Markov walk $( S_n^* )_{n\geq 1}$ is centred under $\tbb P_{i,j}$:
\[
\tbs \nu_{\ll} (\rho) = \tbs \nu_{\ll} (-\rho) = 0
\]
and by Lemma \ref{sourire}, Conditions \ref{primitif} and \ref{cathedrale} hold for $\tbf P_{\ll}^*$. 
Let $\tt V_{\ll}$ and $\tt V_{\ll}^*$ be the harmonic functions of the Markov walks $\left( S_n \right)_{n\geq 0}$ 
and $\left( S_n^* \right)_{n\geq 0}$, respectively
(see Proposition \ref{sable}).

The idea of the proof is in line with that of the previous sections: the positive trajectories 
(corresponding to the event $\left\{ \tau_y > n \right\}$) affect the asymptotic behaviour of the survival probability. However, in the weakly subcritical case, the factor $\e^{-\ll S_n}$ 
in the expectation $\tbb E_i ( \e^{-\ll S_n} q_n\left( f_j(0) \right) \,;\, X_{n+1} = j )$
contributes in such a way that, only the trajectories starting at $y\in \bb R$ conditioned to 
stay positive and to finish nearby $0$, have an impact on the asymptotic 
of $\tbb E_i \left( \e^{-\ll S_n} q_n\left( f_j(0) \right) \,;\, X_{n+1} = j \right)$. 

We start by some preliminary bounds.
The following assertion is similar to Lemma \ref{soir}.

\begin{lemma} Assume that the conditions of Theorem \ref{cape} are satisfied.
\label{savaneBP}
For any $i \in \bb X$, $y \in \bb R$, $k \geq 1$ and $n\geq k+1$, we have
\[
n^{3/2} \tbb E_i \left( \e^{-S_k} \e^{-\ll S_n} \,;\, \tau_y > n \right) \leq \e^{(1+\ll) y} (1+\max(y,0)) \frac{c n^{3/2}}{(n-k)^{3/2}k^{3/2}}.
\]
\end{lemma}

\begin{proof}
Fix $i \in \bb X$, $y \in \bb R$, $k \geq 1$ and $n\geq k+1$. By the Markov property,
\begin{align*}
I_0 &:= n^{3/2} \tbb E_i \left( \e^{-S_k} \e^{-\ll S_n} \,;\, \tau_y > n \right) \\
&\leq \sum_{p=0}^{+\infty} n^{3/2} \e^{\ll y} \e^{-\ll p} \tbb E_i \left( \e^{-S_k} \,;\, y+S_n \in [p,p+1] \,,\, \tau_y > n \right) \\
&= \sum_{p=0}^{+\infty} n^{3/2} \e^{\ll y} \e^{-\ll p} \tbb E_i \left( \e^{-S_k} J_{n-k} \left( X_k, y+S_k \right) \,;\, \tau_y > k \right),
\end{align*}
where for any $i' \in \bb X$, $y' \in \bb R$ and $p \geq 1$
\[
J_{n-k} (i',y') = \tbb P_{i'} \left( y'+S_{n-k} \in [p,p+1] \,,\, \tau_{y'} > n-k \right).
\]
By the point \ref{gorilleBP} of Proposition \ref{goliane},
\[
J_{n-k} (i',y') \leq \frac{c}{(n-k)^{3/2}} (1+p)(1+\max(y',0)).
\]
Consequently,
\begin{align*}
I_0 &\leq \e^{\ll y} \frac{c n^{3/2}}{(n-k)^{3/2}} \tbb E_i \left( \e^{-S_k} \left( 1+y+S_k \right) \,;\, \tau_y > k \right) \sum_{p=0}^{+\infty} \e^{-\ll p} (1+p) \\
&\leq \e^{\ll y} \frac{c n^{3/2}}{(n-k)^{3/2}} \tbb E_i \left( \e^{-S_k} \left( 1+y+S_k \right) \,;\, \tau_y > k \right) \\
&\leq \e^{(1+\ll) y} \frac{c n^{3/2}}{(n-k)^{3/2}} \sum_{p=0}^{+\infty} \e^{-p}(2+p) \tbb P_i \left( y+S_k \in [p,p+1] \,;\, \tau_y > k \right).
\end{align*}
Again by the point \ref{gorilleBP} of Proposition \ref{goliane},
\[
I_0 \leq \e^{(1+\ll) y} (1+\max(y,0)) \frac{c n^{3/2}}{(n-k)^{3/2}k^{3/2}} \sum_{p=0}^{+\infty} \e^{-p}(2+p)(1+p). 
\]
This concludes the proof of the lemma. 
\end{proof}

For any $l \geq 1$ and $n \geq l+1$, set
\[
q_{l,n}\left( f_j(0) \right) := 1-f_{l+1,n}\left( f_j(0) \right) = 1-f_{X_{l+1}} \circ \cdots \circ f_{X_n} \circ f_j(0),
\]
In the same way as in Lemma \ref{foin}, we obtain:
\begin{equation}
	\label{noisette}
	q_{l,n}\left( f_j(0) \right)^{-1} = \frac{\e^{S_l-S_n}}{1-f_j(0)} + \sum_{k=l}^{n-1} \e^{S_l-S_k} \eta_{k+1,n}\left( f_j(0) \right),
\end{equation}
where $\eta_{k+1,n}(s)$ are defined by \eqref{montre004}. Moreover, similarly to \eqref{histoire}, we have for any $n \geq l+1 \geq 2$,
\begin{equation}
	\label{mouton}
	q_{l,n}\left( f_j(0) \right) \in (0,1] \qquad \tbb P_i\text{-a.s.}
\end{equation}
In addition, by Lemma \ref{pieuvre}, for any $k \leq n-1$,
\begin{equation}
\label{piscine}
0 \leq \eta_{k+1,n}\left( f_j(0) \right) \leq \eta \qquad \tbb P_i\text{-a.s.}
\end{equation}

\begin{lemma} Assume that the conditions of Theorem \ref{cape} are satisfied.
\label{constellation}
For any $(i,j) \in \bb X^2$ and $y \in \bb R$, we have
\[
\lim_{l,m \to +\infty} \limsup_{n\to+\infty} n^{3/2} \tbb E_i \left( \abs{\e^{-S_{n-m}} q_{n-m,n}\left( f_j(0) \right)^{-1} - \e^{-S_l} q_{l,n}\left( f_j(0) \right)^{-1}} \e^{-\ll S_n} \,;\, \tau_y > n \right) = 0.
\]
\end{lemma}

\begin{proof}
Fix $(i,j) \in \bb X^2$ and $y \in \bb R$. For any $l \geq 1$, $m \geq 1$ and $n \geq l+m+1$, we have
\begin{align*}
I_0 &:= n^{3/2} \tbb E_i \left( \abs{\e^{-S_{n-m}} q_{n-m,n}\left( f_j(0) \right)^{-1} - \e^{-S_l} q_{l,n}\left( f_j(0) \right)^{-1}} \e^{-\ll S_n} \,;\, \tau_y > n \right) \\
&= n^{3/2} \tbb E_i \left( \sum_{k=l}^{n-m-1} \e^{-S_k} \eta_{k+1,n}\left( f_j(0) \right) \e^{-\ll S_n} \,;\, \tau_y > n \right).
\end{align*}
Using \eqref{piscine} and Lemma \ref{savaneBP},
\[
I_0 \leq \eta \sum_{k=l}^{n-m-1} \e^{(1+\ll) y} (1+\max(y,0)) \frac{c n^{3/2}}{(n-k)^{3/2}k^{3/2}}.
\]
Let $n_1 := \pent{n/2}$. We note that
\begin{align*}
\sum_{k=l}^{n-m-1} \frac{c n^{3/2}}{(n-k)^{3/2}k^{3/2}} &\leq \frac{c n^{3/2}}{(n-n_1)^{3/2}} \sum_{k=l}^{n_1} \frac{1}{k^{3/2}} + \frac{c n^{3/2}}{n_1^{3/2}}\sum_{k=n_1+1}^{n-m-1} \frac{1}{(n-k)^{3/2}} \\
&\leq c \sum_{k=l}^{+\infty} \frac{1}{k^{3/2}} + c \sum_{k=m}^{+\infty} \frac{1}{k^{3/2}}.
\end{align*}
Consequently,
\[
\limsup_{n\to +\infty} I_0 \leq c\eta \e^{(1+\ll) y} (1+\max(y,0)) \left( \sum_{k=l}^{+\infty} \frac{1}{k^{3/2}} + \sum_{k=m}^{+\infty} \frac{1}{k^{3/2}} \right).
\]
Taking the limits as $l \to +\infty$ and $m\to +\infty$, proves the lemma.
\end{proof}

For any $l \geq 1$, $m\geq 1$ and $n \geq l+m+1$, consider the random variables 
\begin{align*}
&r_n^{(l,m)}(j) := 1 - f_{1,l} \left( \left[ 1-f_{l+1,n-m}'(1) \left( 1-f_{n-m+1,n}\left( f_j(0) \right) \right) \right]^+ \right) \\
&= 1- f_{X_1}\circ \cdots \circ f_{X_l} \left( \left[ 1-f_{X_{l+1}}'(1) \times \dots \times f_{X_{n-m}}'(1) \left( 1-f_{X_{n-m+1}}\circ \cdots \circ f_{X_n} \circ f_j(0) \right) \right]^+ \right),
\end{align*}
where $[t]^+ = \max(t,0)$ for any $t\in \bb R$. The random variable $r_n^{(l,m)}(j)$ approximates $q_n\left( f_j(0) \right)$ in the following sense:

\begin{lemma} Assume that the conditions of Theorem \ref{cape} are satisfied.
\label{carrousel}
For any $(i,j) \in \bb X^2$ and $y \in \bb R$,
\[
\lim_{l,m \to +\infty} \limsup_{n\to +\infty} n^{3/2} \tbb E_i \left( \abs{q_n\left( f_j(0) \right) - r_n^{(l,m)}(j)} \e^{-\ll S_n} \,;\, \tau_y > n \right) = 0.
\]
\end{lemma}

\begin{proof}
Fix $(i,j) \in \bb X^2$ and $y \in \bb R$. Since for any $i' \in \bb X$, $f_{i'}$ is increasing and convex, the function $f_{l+1,n-m}$ is convex. So, for any $l \geq 1$, $m \geq 1$ and $n \geq l+m+1$,
\[
f_{l+1,n}\left( f_j(0) \right) = f_{l+1,n-m} \left( f_{n-m+1,n}\left( f_j(0) \right) \right) \geq \left[ 1- f_{l+1,n-m}'(1) \left( 1-f_{n-m+1,n}\left( f_j(0) \right) \right) \right]^+.
\]
Since $f_{1,l}$ is increasing,
\[
q_n\left( f_j(0) \right) = 1-f_{1,n}\left( f_j(0) \right) \leq r_n^{(l,m)}(j),
\]
or equivalently
\[
0 \leq r_n^{(l,m)}(j) - q_n\left( f_j(0) \right).
\]
Moreover, by the convexity of $f_{1,l}$,
\begin{align*}
r_n^{(l,m)}(j) - q_n\left( f_j(0) \right) &= f_{1,l} \circ f_{l+1,n}\left( f_j(0) \right) - f_{1,l} \left( \left[ 1-f_{l+1,n-m}'(1) \left( 1-f_{n-m+1,n}\left( f_j(0) \right) \right) \right]^+ \right) \\
&\leq f_{1,l}'(1) \left( f_{l+1,n}\left( f_j(0) \right) - \left[ 1-f_{l+1,n-m}'(1) \left( 1-f_{n-m+1,n}\left( f_j(0) \right) \right) \right]^+ \right) \\
&\leq f_{1,l}'(1) \left( f_{l+1,n-m}'(1) q_{n-m,n}\left( f_j(0) \right) -q_{l,n}\left( f_j(0) \right) \right) \\
&= \e^{S_{n-m}} q_{n-m,n}\left( f_j(0) \right) - \e^{S_l} q_{l,n}\left( f_j(0) \right) \\
&= \e^{S_{n-m}} q_{n-m,n}\left( f_j(0) \right) \e^{S_l} q_{l,n}\left( f_j(0) \right) \\
&\qquad \times \left( \e^{-S_l} q_{l,n}\left( f_j(0) \right)^{-1} - \e^{-S_{n-m}} q_{n-m,n}\left( f_j(0) \right)^{-1} \right).
\end{align*}
By \eqref{noisette}, we have $q_{l,n}\left( f_j(0) \right) \leq \e^{S_n-S_l}$ and so
\[
r_n^{(l,m)}(j) - q_n\left( f_j(0) \right) \leq \e^{2S_n} \left( \e^{-S_l} q_{l,n}\left( f_j(0) \right)^{-1} - \e^{-S_{n-m}} q_{n-m,n}\left( f_j(0) \right)^{-1} \right).
\]
In addition, by the definition of $r_n^{(l,m)}(j)$ and $q_n\left( f_j(0) \right)$, we have $r_n^{(l,m)}(j) - q_n\left( f_j(0) \right) \leq 1$. Therefore, $\tbb P_i\text{-a.s.}$ it holds,
\[
r_n^{(l,m)}(j) - q_n\left( f_j(0) \right) \leq \min\left( 1,\e^{2S_n} \left( \e^{-S_l} q_{l,n}\left( f_j(0) \right)^{-1} - \e^{-S_{n-m}} q_{n-m,n}\left( f_j(0) \right)^{-1} \right) \right).
\]
Using the previous bound, it follows that, for any integer $N \geq 1$,
\begin{align*}
I_0 &:= n^{3/2} \tbb E_i \left( \abs{q_n\left( f_j(0) \right) - r_n^{(l,m)}(j)} \e^{-\ll S_n} \,;\, \tau_y > n \right) \\
&\leq \e^{2(N-y)} n^{3/2} \tbb E_i \left( \abs{\e^{-S_l} q_{l,n}\left( f_j(0) \right)^{-1} - \e^{-S_{n-m}} q_{n-m,n}\left( f_j(0) \right)^{-1}} \e^{-\ll S_n} \,;\, \tau_y > n \right) \\
&\hspace{5cm} + n^{3/2} \tbb E_i \left( \e^{-\ll S_n} \,;\, y+S_n > N \,,\, \tau_y > n \right). 
\end{align*}
Moreover, using the point \ref{gorilleBP} of Proposition \ref{goliane},
\begin{align*}
n^{3/2} \tbb E_i \left( \e^{-\ll S_n} \,;\, y+S_n > N \,,\, \tau_y > n \right) &\leq \sum_{p=N}^{+\infty} \e^{\ll y} \e^{-\ll p} n^{3/2} \tbb P_i \left( y+S_n \in [p,p+1] \,,\, \tau_y > n \right) \\
&\leq c \e^{\ll y} (1+\max(y,0)) \sum_{p=N}^{+\infty} \e^{-\ll p} (1+p).
\end{align*}
Consequently, using Lemma \ref{constellation}, we obtain that
\[
\lim_{l,m \to +\infty} \limsup_{n\to+\infty} I_0 \leq c \e^{\ll y} (1+\max(y,0)) \sum_{p=N}^{+\infty} \e^{-\ll p} (1+p).
\]
Taking the limit as $N \to +\infty$, proves the lemma.
\end{proof}

We now introduce the following random variable: for any $j \in \bb X$, $u \in \bb R$, $l \geq 1$ and $m \geq 1$
\[
r_{\infty}^{(l,m)}(j,u) := 1-f_{X_1} \circ \cdots \circ f_{X_l} \left( \left[ 1- \e^{-S_l} \e^{u} q_m^*(j) \right]^+ \right) \in [0,1],
\]
where, as in \eqref{chemin} and \eqref{cheminbis}, for any $m \geq 1$,
\[
q_m^*(j) := \e^{S_m^*} \left( 1- f_{X_m^*} \circ \cdots \circ f_{X_1^*} \circ f_j (0) \right) = \left[ \frac{1}{1-f_j(0)} + \sum_{k=1}^{n} \e^{-S_k^*} \eta_k^*(j) \right]^{-1}
\]
and as in \eqref{chemin003}, for any $k \geq 2$,
\[
\eta_k^*(j) := g_{X_k^*} \left( f_{X_{k-1}^*} \circ \cdots \circ f_{X_1^*} \circ f_j (0) \right) \qquad \text{and} \qquad \eta_1^* := g_{X_1^*} \left( f_j(0) \right).
\]

For any $(i,y) \in \supp(\tt V_{\ll})$ and $(j,z) \in \supp(\tt V_{\ll}^*)$, let $\tbb P_{i,y,j,z}^+$ and $\tbb E_{i,y,j,z}^+$ be, respectively, the probability and its associated expectation defined for any $n \geq 1$ and any function $g$: $\bb X^{l,m} \to \bb C$ by
\begin{align}
\tbb E_{i,y,j,z}^+ &\left( g \left( X_1, \dots, X_l,X_m^*,\dots,X_1^* \right) \right) = \tbb E_{i,j} \left( g \left( X_1, \dots, X_l,X_m^*,\dots,X_1^* \right)  \frac{\tt V_{\ll} \left( X_l, y+S_l \right)}{\tt V_{\ll}(i,y)} \times \right. \nonumber \\
&\hspace{8cm} \left. \frac{\tt V_{\ll}^* \left( X_m^*, z+S_m^* \right)}{\tt V_{\ll}^*(j,z)} \,;\, \tau_y > l \,,\, \tau_z^* > m \right).
\label{theiere}
\end{align}

For any $j \in \bb X$ let $z_0(j) \in \bb R$ be the unique real such that $(j,z) \in \supp \left( \tt V_{\ll}^* \right)$ for any $z > z_0$ and $(j,z) \notin \supp \left( \tt V_{\ll}^* \right)$ for any $z < z_0$ (see \cite{grama_limit_2016-1} for details on the domain of positivity of the harmonic function). Set $z_0(j)^+=\max\left\{z_0(j), 0 \right\}$.

\begin{lemma} Assume that the conditions of Theorem \ref{cape} are satisfied.
\label{corbeau}
For any $j \in \bb X$, $(i,y) \in \supp\left( \tt V_{\ll} \right)$, $l \geq 1$ and $m \geq 1$,
\begin{align*}
&\lim_{n\to +\infty} n^{3/2} \tbb E_i \left( r_n^{(l,m)}(j) \e^{-\ll S_n} \,;\, X_{n+1} = j \,,\, \tau_y > n \right) \\
&\hspace{2cm} = \frac{2}{\sqrt{2\pi} \sigma^3} \e^{\ll y} \int_{z_0(j)^+}^{+\infty} \e^{-\ll z} \tbb E_{i,y,j,z}^+ \left( r_{\infty}^{(l,m)}(j,z-y) \right) \tt V_{\ll}(i,y) \tt V_{\ll}^*(j,z) \dd z \tbs \nu_{\ll}(j).
\end{align*}
\end{lemma}

\begin{proof}
Fix $(i,y) \in \supp\left( \tt V_{\ll} \right)$, $j \in \bb X$, $l\geq 1$ and $m \geq 1$ and let $g$ be a function $\bb X^{l+m} \times \bb R \to \bb R_+$ defined by
\begin{align*}
&g(i_1, \dots, i_l,i_{n-m+1},\dots, i_n, z) = \e^{\ll y} \e^{-\ll z} \bbm 1_{\{ z \geq 0 \}} \tbf P_{\ll}(i_n,j) \left[ 1 \right. \\
&\left. -f_{i_1} \circ \cdots \circ f_{i_l} \left( \left[ 1- \e^{z-y-\rho(i_n)-\cdots-\rho(i_{n-m+1})-\rho(i_l)-\dots-\rho(i_1)} \left( 1-f_{i_{n-m+1}}\circ \cdots \circ f_{i_n} \circ f_j (0) \right) \right]^+ \right) \right]
\end{align*}
for all $(i_1, \dots, i_l,i_{n-m+1},\dots, i_n, z) \in \bb X^{l+m} \times \bb R $ and note that on $\{ \tau_y > n \}$,
\[
g(X_1,\dots,X_l,X_{n-m+1},\dots,X_n,y+S_n) = r_n^{(l,m)}(j) \e^{-\ll S_n} \tbf P_{\ll}(i_n,j).
\]
Observe also that since $0 \leq g(i_1, \dots, i_l,i_{n-m+1},\dots, i_n, z) \leq \e^{\ll y} \e^{-\ll z} \bbm 1_{\{ z \geq 0 \}}$, the function $g$ belongs to the set, say $\scr C^+ \left( \bb X^{l+m} \times \bb R_+ \right)$, of non-negative function $g$: $\bb X^{l+m} \times \bb R_+ \to \bb R_+$ satisfying the following properties:
\begin{itemize}
\item for any $(i_1,\dots,i_{l+m}) \in \bb X^{l+m}$, the function $z \mapsto g(i_1,\dots,i_{l+m},z)$ is continuous,
\item there exists $\ee >0$ such that $\max_{i_1,\dots i_{l+m} \in \bb X} \sup_{z\geq 0} g(i_1,\dots,i_{l+m},z) (1+z)^{2+\ee} < +\infty$.
\end{itemize}
Therefore, by the Markov property and Proposition \ref{sorcier}, 
we obtain that
\begin{align*}
I_0 &:= \lim_{n\to +\infty} n^{3/2} \tbb E_i \left( r_n^{(l,m)}(j) \e^{-\ll S_n} \,;\, X_{n+1} = j \,,\, \tau_y > n \right) \\
&= \lim_{n\to +\infty} n^{3/2} \tbb E_i \left( g \left(X_1, \dots, X_l, X_{n-m+1}, \dots, X_n, y+S_n \right) \,;\, \tau_y > n \right) \\
&= \frac{2}{\sqrt{2\pi} \sigma^3} \int_0^{+\infty} \e^{-\ll (z-y)} \sum_{j' \in \bb X} \tbb E_{i,j'} \left( r_{\infty}^{(l,m)}(j,z-y) \tbf P_{\ll}(X_1^*,j) \tt V_{\ll}\left( X_l,y+S_l \right) \right. \\
&\hspace{6cm} \left. \times \tt V_{\ll}^* \left( X_m^*, z+S_m^* \right) \,;\, \tau_y > l \,,\, \tau_z^* > m \right) \tbs \nu_{\ll}(j') \dd z.
\end{align*}
Since $\tbs \nu_{\ll}$ is $\tbf P_{\ll}^*$-invariant, we write
\begin{align*}
I_0 &= \frac{2}{\sqrt{2\pi} \sigma^3} \int_0^{+\infty} \e^{-\ll (z-y)} \sum_{j_1 \in \bb X} \tbf P_{\ll}(j_1,j) \tbs \nu_{\ll}(j_1) \tbb E_i \left( r_{\infty}^{(l,m)}(j,z-y)  \tt V_{\ll}\left( X_l,y+S_l \right) \right. \\
&\hspace{6cm} \left. \times \sachant{\tt V_{\ll}^* \left( X_m^*, z+S_m^* \right) \,;\, \tau_y > l \,,\, \tau_z^* > m}{X_1^*=j_1} \right)  \dd z.
\end{align*}
Using the definition of $\tbf P_{\ll}^*$ in \eqref{monument}, we have
\begin{align*}
I_0 &= \frac{2}{\sqrt{2\pi} \sigma^3} \int_0^{+\infty} \e^{-\ll (z-y)} \tbs \nu_{\ll}(j)  \tbb E_{i,j} \left( r_{\infty}^{(l,m)}(j,z-y)  \tt V_{\ll}\left( X_l,y+S_l \right) \right. \\
&\hspace{6cm} \left. \times \tt V_{\ll}^* \left( X_m^*, z+S_m^* \right) \,;\, \tau_y > l \,,\, \tau_z^* > m \right)  \dd z.
\end{align*}
Now, note that when $(j,z) \notin \supp\left( \tt V_{\ll}^* \right)$, using the point \ref{sable001} of Proposition \ref{sable},
\begin{align*}
\tbb E_{i,j} &\left( r_{\infty}^{(l,m)}(j,z-y)  \tt V_{\ll}\left( X_l,y+S_l \right) \tt V_{\ll}^* \left( X_m^*, z+S_m^* \right) \,;\, \tau_y > l \,,\, \tau_z^* > m \right) \\
&\leq \tbb E_i \left(  \tt V_{\ll}\left( X_l,y+S_l \right) \,;\, \tau_y > l \right) \tbb E_j^* \left( \tt V_{\ll}^* \left( X_m^*, z+S_m^* \right) \,;\, \tau_z^* > m \right) = \tt V_{\ll}(i,y) \tt V_{\ll}^*(j,z) = 0.
\end{align*}
Together with \eqref{theiere}, it proves the lemma.
\end{proof}

Consider for any $l \geq 1$, $j \in \bb X$ and $u \in \bb R$,
\begin{equation}
\label{mousse}
r_{\infty}^{(l,\infty)}(j,u) = 1-f_{X_1} \circ \cdots \circ f_{X_l} \left( \left[ 1- \e^{-S_l} \e^{u} q_{\infty}^*(j) \right]^+ \right) \in [0,1],
\end{equation}
where as in \eqref{potion002},
\[
q_{\infty}^*(j) = \left[ \frac{1}{1-f_j(0)} + \sum_{k=1}^{\infty} \e^{-S_k^*} \eta_k^*(j) \right]^{-1}.
\]

\begin{lemma} Assume that the conditions of Theorem \ref{cape} are satisfied.
\label{fontaine001}
For any $u \in \bb R$, $(i,y) \in \supp \left( \tt V_{\ll} \right)$, $(j,z) \in \supp \left( \tt V_{\ll}^* \right)$ and $l\geq 1$,
\[
\lim_{m\to +\infty} \tbb E_{i,y,j,z}^+ \left( \abs{r_{\infty}^{(l,m)}(j,u) - r_{\infty}^{(l,\infty)}(j,u)} \right) = 0.
\]
\end{lemma}

\begin{proof}
Fix $(i,y) \in \supp \left( \tt V_{\ll} \right)$, $(j,z) \in \supp \left( \tt V_{\ll}^* \right)$, $l\geq 1$ and $u \in \bb R$. By the convexity of $f_{1,l}$, for any $m \geq 1$, we have $\tbb P_{i,y,j,z}^+$ a.s.,
\begin{align*}
\abs{r_{\infty}^{(l,m)}(j,u) - r_{\infty}^{(l,\infty)}(j,u)} &\leq \left(f_{X_1} \circ \cdots \circ f_{X_l}\right)'(1) \abs{\left[ 1- \e^{-S_l} \e^{u} q_m^*(j) \right]^+ - \left[ 1- \e^{-S_l} \e^{u} q_{\infty}^*(j) \right]^+} \\
&\leq \e^{S_l}  \abs{\e^{-S_l} \e^{u} q_m^*(j) - \e^{-S_l} \e^{u} q_{\infty}^*(j)} \\
&= \e^{u} \abs{q_m^*(j)q_{\infty}^*(j)} \abs{ \left(q_{\infty}^*(j) \right)^{-1} - \left(q_{m}^*(j) \right)^{-1}}.
\end{align*}
Moreover, for any $m \geq 1$,
\begin{align*}
	q_m^*(j) &= \left[ \frac{1}{1-f_j(0)} + \sum_{k=1}^{m} \e^{-S_k^*} \eta_k^*(j) \right]^{-1} \in (0,1],\\
	q_{\infty}^*(j) &= \left[ \frac{1}{1-f_j(0)} + \sum_{k=1}^{\infty} \e^{-S_k^*} \eta_k^*(j) \right]^{-1} \in [0,1]
\end{align*}
and by Lemma \ref{pieuvre}, for any $k \geq 1$,
\begin{equation}
	\label{clocher001}
	0 \leq \eta_k^*(j) \leq \eta.
\end{equation}
Therefore,
\[
\abs{r_{\infty}^{(l,m)}(j,u) - r_{\infty}^{(l,\infty)}(j,u)} \leq \e^{u} \eta \sum_{k=m+1}^{+\infty} \e^{-S_k^*}.
\]
Using Lemma \ref{soir} and the Lebesgue dominated convergence theorem,
\[
\tbb E_{i,y,j,z}^+ \left( \abs{r_{\infty}^{(l,m)}(j,u) - r_{\infty}^{(l,\infty)}(j,u)} \right) \leq \e^{u} \eta \sum_{k=m+1}^{+\infty} \tbb E_{i,y,j,z}^+ \left( \e^{-S_k^*} \right).
\]
By Lemma \ref{soir}, we conclude that
\[
\lim_{m\to+\infty} \tbb E_{i,y,j,z}^+ \left( \abs{r_{\infty}^{(l,m)}(j,u) - r_{\infty}^{(l,\infty)}(j,u)} \right) = 0.
\]
\end{proof}

For any $l\geq 1$, $j \in \bb X$ and $u \in \bb R$, set
\begin{equation}
	\label{pollen}
	s_l(j,u) = \left[ 1- \e^{-S_l} \e^{u} q_{\infty}^*(j) \right]^+.
\end{equation}
Note that, by Lemma \ref{soir}, $\left( q_{\infty}^*(j) \right)^{-1}$ is integrable and so finite a.s.\ (see \eqref{panda}). Therefore $s_l(j,u) \in [0,1)$. In addition, by the convexity of $f_{X_{l+1}}$, we have for any $j \in \bb X$, $u \in \bb R$ and $l \geq 1$,
\[
f_{X_{l+1}}(s_{l+1}(j,u)) \geq 1-f_{X_{l+1}}'(1) \left( 1 - s_{l+1}(j,u) \right) \geq 1-\e^{\rho(X_{l+1})} \e^{-S_{l+1}} \e^{u} q_{\infty}^*(j) = 1- \e^{-S_l} \e^{u} q_{\infty}^*(j).
\]
Since $f_{X_{l+1}}$ is non-negative on $[0,1]$, we see that $f_{X_{l+1}}(s_{l+1}(j,u)) \geq s_l(j,u)$ and so for any $k \geq 1$, $\left( f_{k+1,l}(s_l(j,u)) \right)_{l\geq k}$ is non-decreasing and bounded by $1$. Using the continuity of $g_{X_k}$ and \eqref{champ}, we deduce that $\left( \eta_{k,l}(s_l(j,u)) \right)_{l\geq k}$ converges and we denote for any $k \geq 1$,
\begin{equation}
	\label{balcon}
	\eta_{k,\infty}(j,u) := \lim_{l \to +\infty} \eta_{k,l}(s_l(j,u)).
\end{equation}
Moreover, 
by Lemma \ref{pieuvre}, we have for any $k \geq 1$, $l \geq k$ and $u \in \bb R$,
\begin{equation}
\label{clocher002}
0 \leq \eta_{k,l}(s_l(j,u)) \leq \eta \qquad \text{and} \qquad 0 \leq \eta_{k,\infty}(j,u) \leq \eta.
\end{equation}
For any $j \in \bb X$ and $u \in \bb R$, set 
\[
r_{\infty}(j,u) := \left[ \frac{\e^{-u}}{q_{\infty}^*(j)} + \sum_{k=0}^{+\infty} \e^{-S_k} \eta_{k+1,\infty}(j,u) \right]^{-1}.
\]

\begin{lemma} Assume that the conditions of Theorem \ref{cape} are satisfied.
\label{fontaine002}
For any $u \in \bb R$, $(i,y) \in \supp \left( \tt V_{\ll} \right)$ and $(j,z) \in \supp \left( \tt V_{\ll}^* \right)$,
\[
\lim_{l\to +\infty} \tbb E_{i,y,j,z}^+ \left( \abs{r_{\infty}^{(l,\infty)}(j,u) - r_{\infty}(j,u)} \right) = 0.
\]
\end{lemma}

\begin{proof}
Fix $(i,y) \in \supp \left( \tt V_{\ll} \right)$, $(j,z) \in \supp \left( \tt V_{\ll}^* \right)$ and $u \in \bb R$. By \eqref{mousse}, Lemma \ref{foin} and \eqref{pollen}, we have
\[
\left( r_{\infty}^{(l,\infty)}(j,u) \right)^{-1} = \frac{\e^{-S_l}}{1-s_l(j,u)} + \sum_{k=0}^{l-1} \e^{-S_k} \eta_{k+1,l}(s_l(j,u)).
\]
So, for any $p \geq 1$ and $l \geq p$, using \eqref{clocher002},
\begin{align*}
\abs{\left( r_{\infty}^{(l,\infty)}(j,u) \right)^{-1} - r_{\infty}(j,u)^{-1}} &\leq  \sum_{k=0}^{p} \e^{-S_k} \abs{\eta_{k+1,l}(s_l(j,u)) - \eta_{k+1,\infty}(j,u)} \\
&\qquad + \abs{\frac{\e^{-u}}{q_{\infty}^*(j)} - \frac{\e^{-S_l}}{1-s_l(j,u)}} + 2\eta \sum_{k=p+1}^{+\infty} \e^{-S_k}.
\end{align*}
Therefore, 
\begin{align*}
	I_0 &:= \tbb E_{i,y,j,z}^+ \left( \abs{\left( r_{\infty}^{(l,\infty)}(j,u) \right)^{-1} - r_{\infty}(j,u)^{-1}} \right) \\
	&\leq \sum_{k=0}^{p} \tbb E_{i,y,j,z}^+ \left( \e^{-S_k} \abs{\eta_{k+1,l}(s_l(j,u)) - \eta_{k+1,\infty}(j,u)} \right) \\
	&\qquad+ \tbb E_{i,y,j,z}^+ \left( \abs{\frac{\e^{-u}}{q_{\infty}^*(j)} - \e^{-S_l}} \,;\, \e^{-S_l} > \frac{\e^{-u}}{q_{\infty}^*(j)} \right) + 2\eta \tbb E_{i,y}^+ \left( \sum_{k=p+1}^{+\infty} \e^{-S_k} \right),
\end{align*}
where $\tbb P_{i,y}^+$ is the marginal law of $\tbb P_{i,y,j,z}^+$ on $\sigma\left( X_n \,,\, n \geq 1 \right)$. Using Lemma \ref{soir} and the Lebesgue dominated convergence theorem,
\begin{align*}
	I_0 &\leq \tbb E_{i,y}^+ \left( \e^{-S_l} \right) + \sum_{k=0}^{p} \tbb E_{i,y,j,z}^+ \left( \e^{-S_k} \abs{\eta_{k+1,l}\left(s_l(j,u)\right) - \eta_{k+1,\infty}(j,u)} \right) + 2\eta \sum_{k=p+1}^{+\infty} \tbb E_{i,y}^+ \left( \e^{-S_k} \right) \\
	&\leq \frac{c \left( 1+\max(y,0) \right)\e^{y}}{V(i,y)} \left( \frac{1}{l^{3/2}} + \sum_{k=p+1}^{+\infty} \frac{\eta}{k^{3/2}} \right) \\
	&\hspace{3cm} + \sum_{k=0}^{p} \tbb E_{i,y,j,z}^+ \left( \e^{-S_k} \abs{\eta_{k+1,l}(s_l(j,u)) - \eta_{k+1,\infty}(j,u)} \right).
\end{align*}
Since $\abs{\eta_{k+1,l}(s_l(j,u)) - \eta_{k+1,\infty}(j,u)} \leq 2\eta$, by the Lebesgue dominated convergence theorem and \eqref{balcon}
\[
\limsup_{l\to+\infty} I_0 \leq \frac{c \left( 1+\max(y,0) \right)\e^{y}}{V(i,y)} \sum_{k=p+1}^{+\infty} \frac{\eta}{k^{3/2}}.
\]
Letting $p \to +\infty$, we obtain that $\lim_{l\to+\infty} I_0 = 0$. Moreover, by \eqref{mousse} for any $l \geq 1$, $r_{\infty}^{(l,\infty)}(j,u)  \in [0,1]$. In the same manner as we proved \eqref{echo}, we have also
\[
r_{\infty}(j,u) \leq 1.
\]
Consequently,
\[
\lim_{l\to +\infty} \tbb E_{i,y,j,z}^+ \left( \abs{r_{\infty}^{(l,\infty)}(j,u) - r_{\infty}(j,u)} \right) \leq \lim_{l\to+\infty} I_0 = 0.
\]
\end{proof}

We now consider the function
\[
U(i,y,j) := \frac{2}{\sqrt{2\pi} \sigma^3} \frac{v_{\ll}(i)}{v_{\ll}(j)} \e^{\ll (y-\rho(j))} \int_{z_0(j)^+}^{+\infty} \e^{-\ll z} \tbb E_{i,y,j,z}^+ \left( r_{\infty} (j,z-y) \right) \tt V_{\ll} (i,y) \tt V_{\ll}^* (j,z) \dd z \tbs \nu_{\ll}(j).
\]
Using \eqref{clocher001}, \eqref{clocher002} and Lemma \ref{soir}, for any $(i,y) \in \supp \left( \tt V_{\ll} \right)$, $(j,z) \in \supp \left( \tt V_{\ll}^* \right)$ and $u \in \bb R$,
\[
\tbb E_{i,y,j,z}^+ \left( r_{\infty} (j,u)^{-1} \right) \leq \e^{-u} \left( \frac{1}{1-f_j(0)} + \eta \tbb E_{i,y,j,z}^+ \left( \sum_{k=1}^{+\infty} \e^{-S_k^*} \right) \right) + \eta \tbb E_{i,y,j,z}^+ \left( \sum_{k=1}^{+\infty} \e^{-S_k} \right) < +\infty.
\]
So $r_{\infty} (j,u) > 0$ $\tbb P_{i,y,j,z}^+$-a.s.\ and therefore, for any $(i,y) \in \supp \left( \tt V_{\ll} \right)$, $j \in \bb X,$
\begin{equation}
\label{canopee}
U(i,y,j) > 0.
\end{equation}

\begin{lemma} Assume that the conditions of Theorem \ref{cape} are satisfied.
\label{chaumiere}
For any $(i,y) \in \supp \left( \tt V_{\ll} \right)$ and $j \in \bb X$, we have
\[
\bb E_i \left( q_{n+1} \,;\, X_{n+1} = i \,,\, \tau_y > n \right) \underset{n\to+\infty}{\sim} \frac{U(i,y,j) k(\ll)^{n+1}}{(n+1)^{3/2}}.
\]
\end{lemma}

\begin{proof}
Fix $(i,y) \in \supp \left( \tt V_{\ll} \right)$ and $j \in \bb X$.  By \eqref{falaise}, for any $n \geq 1$,
\begin{align*}
I_0 &:= \frac{(n+1)^{3/2}}{k(\ll)^{n+1}} \bb E_i \left( q_{n+1} \,;\, X_{n+1} = i \,,\, \tau_y > n \right) \\
&= \frac{v_{\ll}(i)\e^{-\ll \rho(j)}}{v_{\ll}(j)} (n+1)^{3/2} \tbb E_i \left( \e^{-\ll S_n} q_{n+1} \,;\, X_{n+1} = j \,,\, \tau_y > n \right).
\end{align*}
Using Lemmas \ref{carrousel} and \ref{corbeau},
\begin{align*}
\lim_{n\to +\infty} I_0 &= \lim_{(l,m) \to +\infty} \lim_{n\to +\infty} \frac{v_{\ll}(i)\e^{-\ll \rho(j)}}{v_{\ll}(j)} (n+1)^{3/2} \tbb E_i \left( r_n^{(l,m)}(j) \e^{-\ll S_n}  \,;\, X_{n+1} = j \,,\, \tau_y > n \right) \\
&= \lim_{(l,m) \to +\infty} \frac{2 v_{\ll}(i)}{\sqrt{2\pi} \sigma^3v_{\ll}(j)} \e^{\ll (y - \rho(j))} \int_{z_0(j)^+}^{+\infty} \e^{-\ll z} \tbb E_{i,y,j,z}^+ \left( r_{\infty}^{(l,m)}(j,z-y) \right) \\
&\hspace{9cm} \times \tt V_{\ll}(i,y) \tt V_{\ll}^*(j,z) \dd z \tbs \nu_{\ll}(j).
\end{align*}
Since for any $l \geq 1$, $m \geq 1$ and $u \in \bb R$, $r_{\infty}^{(l,m)}(j,u) \leq 1$, by the Lebesgue dominated convergence theorem and Lemmas \ref{fontaine001} and \ref{fontaine002},
\begin{align*}
\lim_{n\to +\infty} I_0 &= \frac{2v_{\ll}(i)}{\sqrt{2\pi} \sigma^3v_{\ll}(j)} \e^{\ll (y-\rho(j))} \int_{z_0(j)^+}^{+\infty} \e^{-\ll z} \lim_{l \to +\infty} \tbb E_{i,y,j,z}^+ \left( r_{\infty}^{(l,\infty)}(j,z-y) \right) \\
&\hspace{8cm} \times \tt V_{\ll}(i,y) \tt V_{\ll}^*(j,z) \dd z \tbs \nu_{\ll}(j) \\
&= \frac{2v_{\ll}(i)}{\sqrt{2\pi} \sigma^3v_{\ll}(j)} \e^{\ll (y-\rho(j))} \int_{z_0(j)^+}^{+\infty} \e^{-\ll z} \tbb E_{i,y,j,z}^+ \left( r_{\infty}(j,z-y) \right)  \tt V_{\ll}(i,y) \tt V_{\ll}^*(j,z) \dd z \tbs \nu_{\ll}(j) \\
&= U(i,y,j).
\end{align*} 
\end{proof}

\textbf{Proof of Theorem \ref{cape}.}
We use arguments similar to those of the proof of Lemma \ref{dieu}. 
Fix $(i,j) \in \bb X^2$. For any $y \in \bb R$ and $n \geq 1$, let
\[
I_0 := \frac{(n+1)^{3/2}}{k(\ll)^{n+1}} \bb E_i \left( q_{n+1} \,;\, X_{n+1} = j \right)
\]
and
\begin{align}
\label{celeste}
I_1 &:= I_0 - \frac{(n+1)^{3/2}}{k(\ll)^{n+1}} \bb E_i \left( q_{n+1} \,;\, X_{n+1} = j \,,\, \tau_y > n \right) \\
&=  \frac{(n+1)^{3/2}}{k(\ll)^{n+1}} \bb E_i \left( q_n\left( f_j(0) \right) \,;\, X_{n+1} = j \,,\, \tau_y \leq n \right). \nonumber
\end{align}
By Lemma \ref{foin}, we have $q_n\left( f_j(0) \right) \leq \e^{S_n}$. Using the fact that $\left( q_k\left( f_j(0) \right) \right)_{k\geq 1}$ is non-increasing, it holds $q_n\left( f_j(0) \right) \leq \e^{\min_{1\leq k \leq n} S_k}$. Therefore, as in
\eqref{butte},
\begin{align*}
	I_1 &\leq \frac{(n+1)^{3/2}}{k(\ll)^{n+1}} \bb E_i \left( \e^{\min_{1\leq k \leq n} S_k} \,;\, X_{n+1} = j \,,\, \tau_y \leq n \right) \\
	&\leq \frac{(n+1)^{3/2}}{k(\ll)^{n+1}} \e^{-y} \sum_{p=0}^{+\infty} \e^{-p} \bb P_i \left(  X_{n+1} = j \,,\, \tau_{y+p+1} > n \right).
\end{align*}
By \eqref{samovar},
\begin{align*}
	I_1 &\leq \frac{(n+1)^{3/2}}{n^{3/2}} \frac{v_{\ll}(i)}{v_{\ll}(j)} \e^{-y-\ll \rho(j)} \sum_{p=0}^{+\infty} \e^{-p} \tbb E_i \left( \e^{-\ll S_n} \,;\, \tau_{y+p+1} > n \right) \\
	&\leq c\frac{v_{\ll}(i)}{v_{\ll}(j)} \e^{-y-\ll \rho(j)} \sum_{p=0}^{+\infty} \e^{-p} \sum_{l=0}^{+\infty} \e^{\ll (y+p+1)} \e^{-\ll l} \\
	&\hspace{5cm} \times n^{3/2} \tbb P_i \left( y+p+1+S_n \in [l,l+1] \,;\, \tau_{y+p+1} > n \right).
\end{align*}
Using the point \ref{gorilleBP} of Proposition \ref{goliane},
\begin{align*}
I_1 &\leq c\frac{v_{\ll}(i)\e^{-\ll \rho(j)}}{v_{\ll}(j)} \e^{-(1-\ll)y} \sum_{p=0}^{+\infty} \e^{-(1-\ll)p} \sum_{l=0}^{+\infty} \e^{-\ll l} (1+\max(y+p+1,0))(1+l) \\
&\leq c\frac{v_{\ll}(i)\e^{-\ll \rho(j)}}{v_{\ll}(j)} \e^{-(1-\ll)y} (1+\max(y,0)).
\end{align*}
Moreover, there exists $y_0(i) \in \bb R$ such that, for any $y \geq y_0(i)$ it holds $(i,y) \in \supp \left( \tt V_{\ll} \right)$. Using \eqref{celeste} and Lemma \ref{chaumiere}, we obtain that, for any $y \geq y_0(i)$,
\begin{align}
	U(i,y,j) \leq \liminf_{n\to +\infty} I_0 \leq \limsup_{n\to +\infty} I_0 &\leq U(i,y,j) \nonumber\\
	&\qquad + c\frac{v_{\ll}(i)\e^{-\ll \rho(j)}}{v_{\ll}(j)} \e^{-(1-\ll)y} (1+\max(y,0)).
	\label{accordeon}
\end{align}
This proves that $\limsup_{n\to +\infty} I_0$ is a finite real which does not depend on $y$ and so $y \mapsto U(i,y,j)$ is a bounded function. Moreover, by Lemma \ref{chaumiere}, 
\[
U(i,y,j) = \lim_{n\to\infty} \frac{(n+1)^{3/2}}{k(\ll)^{n+1}} \bb E_i \left( q_{n+1} \,;\, X_{n+1} = j \,,\, \tau_y > n \right)
\]
and so $y \mapsto U(i,y,j)$ is non-decreasing. Let $u$ be its limit:
\[
u(i,j) := \lim_{y\to+\infty} U(i,y,j) \in \bb R.
\]
By \eqref{canopee}, for any $y \geq y_0(i)$,
\[
u(i,j) \geq U(i,y,j) > 0.
\]
Taking the limit as $y \to +\infty$ in \eqref{accordeon},
\[
\lim_{n \to +\infty} I_0 = u(i,j).
\]
Finally, by \eqref{ange},
\[
\lim_{n \to +\infty} \frac{(n+1)^{3/2}}{k(\ll)^{n+1}} \bb P_i \left( Z_{n+1} > 0 \,,\, X_{n+1} = j \right) = \lim_{n \to +\infty} \frac{(n+1)^{3/2}}{k(\ll)^{n+1}} \bb E_i \left( q_{n+1} \,;\, X_{n+1} = j \right) = u(i,j).
\]

\bibliographystyle{plain}
\bibliography{biblioT}

\vskip5mm

\end{document}